\documentclass[a4paper, 11pt]{amsart}

\usepackage{amsmath,amsthm,amssymb,graphicx,pinlabel}
\usepackage[all]{xy}
\usepackage[dvipsnames]{xcolor} \SelectTips{cm}{}

\allowdisplaybreaks

\usepackage{hyperref}
\hypersetup{colorlinks=true,citecolor=brown,linkcolor=brown,urlcolor=black,filecolor=black}

\theoremstyle{definition}
\newtheorem{dfn}{Definition}[section]
\newtheorem{rem}[dfn]{Remark}
\newtheorem{example}[dfn]{Example}

\theoremstyle{plain}
\newtheorem{thm}[dfn]{Theorem}
\newtheorem{prop}[dfn]{Proposition}
\newtheorem{lem}[dfn]{Lemma}
\newtheorem{cor}[dfn]{Corollary}

\makeatletter
    
    \@addtoreset{equation}{section}
  \makeatother

\newcommand{\up}{\vspace{-0.5cm}}

\newcommand{\centre}[1]{\begin{array}{c} #1 \end{array}}

\newcommand{\hpi}{\widehat{\pi}}

\newcommand{\Q}{\mathbb{Q}}
\newcommand{\Z}{\mathbb{Z}}

\newcommand{\calI}{\mathcal{I}}
\newcommand{\calM}{\mathcal{M}}
\newcommand{\calC}{\mathcal{C}}
\newcommand{\calW}{\mathcal{W}}
\newcommand{\calT}{\mathcal{T}}

\newcommand{\frakL}{\mathfrak{L}}
\newcommand{\frakH}{\mathfrak{h}}
\newcommand{\frakM}{\mathfrak{M}}

\newcommand{\Aut}{\operatorname{Aut}}
\newcommand{\IAut}{\operatorname{IAut}}
\newcommand{\Hom}{\operatorname{Hom}}

\begin{document}

\title{Generalized Dehn twists on surfaces and~homology~cylinders}

\author{Yusuke Kuno}
\address{Department of Mathematics, Tsuda University,
2-1-1 Tsuda-machi, Kodaira-shi Tokyo 187-8577, Japan}
\email{kunotti@tsuda.ac.jp}

\author{Gw\'ena\"el Massuyeau}
\address{ {IMB}, Universit\'e  Bourgogne Franche-Comt\'e \& CNRS, 21000 Dijon, France }
\email{{gwenael.massuyeau@u-bourgogne.fr}}

\date{}

\subjclass[2010]{}
\keywords{}
\thanks{Y.K. is supported by JSPS KAKENHI 18K03308.  
 G.M. is partly supported by the  project ITIQ-3D (funded by the
``R\'egion Bourgogne Franche-Comt\'e'') and by the EIPHI Graduate School (contract ANR-17-EURE-0002).}

\begin{abstract}
Let $\Sigma$ be a compact oriented surface.
The Dehn twist along every simple closed curve $\gamma \subset \Sigma$ 
induces an automorphism of the fundamental group $\pi$ of $\Sigma$.
There are two possible ways to generalize such automorphisms if the curve $\gamma$ is allowed to have self-intersections.
One way is to consider the ``generalized Dehn twist'' along $\gamma$:
an automorphism of the Malcev completion of $\pi$ whose definition involves intersection operations  
and only depends on the homotopy class $[\gamma]\in \pi$ of~$\gamma$.
Another way is to choose in the usual cylinder $U:=\Sigma \times [-1,+1]$ a  knot~$L$  projecting onto $\gamma$,
to perform a surgery along $L$ so as to get a homology cylinder~$U_L$,
and let $U_L$ act on every nilpotent quotient $\pi/\Gamma_{j} \pi$ of~$\pi$
(where $\Gamma_j\pi$ denotes the subgroup of $\pi$ generated by commutators of length $j$).
In this paper, assuming that $[\gamma]$ is in $\Gamma_k \pi$ for some  $k\geq 2$,
we prove that (whatever the choice of $L$ is) the automorphism of~$\pi/\Gamma_{2k+1} \pi$ 
induced by $U_L$ agrees with the generalized Dehn twist along~$\gamma$
and we  explicitly compute this  automorphism in terms of $[\gamma]$ modulo ${\Gamma_{k+2}}\pi$.
 As applications, we obtain new formulas for certain evaluations of the Johnson homomorphisms
showing, in particular,  how to realize any element of their targets  
by some explicit homology cylinders and/or generalized Dehn~twists.
\end{abstract}

\maketitle

\vspace{-0.5cm}
{ \setcounter{tocdepth}{1} \tableofcontents}

\section{Introduction and statement of the results} \label{sec:intro}

\subsection{The Dehn--Nielsen representation of the mapping class group} \label{subsec:DN}

Let $\Sigma$ be a connected  compact  oriented surface. 
We assume for simplicity  that $\Sigma$ has exactly one boundary component,
although the discussion below and most of the results presented thereafter could be adapted to a surface 
with arbitrary  non-empty boundary.
The \emph{mapping class group} of the surface~$\Sigma$ 
$$
\calM := 
\big\{ \hbox{\small diffeomorphisms $f\colon \Sigma \to \Sigma$ such that $f\vert_{\partial \Sigma} =\operatorname{id}_{\partial \Sigma}$ }\big\}/\hbox{\small diffeotopy}
$$  
acts in the canonical way on the fundamental group
$\pi:= \pi_1(\Sigma,\star)$  based at a point $\star \in \partial \Sigma$.
By a classical result of Dehn and Nielsen, the resulting  homomorphism
$$
\rho \colon \calM \longrightarrow \Aut(\pi)
$$
is injective so that the mapping class group can be regarded as a subgroup of the automorphism group of  the free group $\pi$. 

The automorphism group of $\pi$ can be ``enlarged'' as  follows. 
The \emph{Malcev completion} of $\pi$ is  the inverse limit
\begin{equation} \label{eq:Malcev}
\hpi :=  \varprojlim_{k} \Big( \frac{\pi}{\Gamma_k \pi}  \otimes \Q \Big)
\end{equation}
where $\pi= \Gamma_1 \pi \supset \Gamma_2\pi \supset   \cdots$ denotes the lower central series of~$\pi$,
and the ``uniquely divisible closure'' of  a nilpotent group $N$ is denoted by $N\otimes \Q$ (see \cite{KM} for instance).
Equivalently \cite{Jennings,Quillen}, $\hpi$ can be understood  as the group-like part of the complete Hopf algebra
$$
\hat{A} := \varprojlim_{k} \frac{\Q[\pi]}{I^k}
$$
obtained by $I$-adic completion of the group algebra $A:=\Q[\pi]$, where $I$~denotes the augmentation ideal.
The filtration $\hat A \supset  \hat{I}^1 \supset  \hat{I}^2 \supset\cdots$ on~$\hat{A}$ given by completion
of the successive powers of $I$ induces a filtration $\hpi = \hpi_1 \supset \hpi_2 \supset \cdots$ on $\hpi$ 
defined by $\hpi_k:= \hpi \cap\big(1+\hat{I}^k\big)$:
let $\Aut(\hpi)$ be the group of its filtration-preserving automorphisms.
Being a free group, $\pi$ embeds into~$\hpi$, which implies injectivity for  the canonical map
$$
j\colon \Aut(\pi) \longrightarrow \Aut(\hpi).
$$

This paper is aimed at comparing two possible generalizations 
of the Dehn--Nielsen representation $j\circ \rho$  of the mapping class group  in the larger group $\Aut(\hpi)$.
Both generalizations have arisen recently from the study of the Johnson homomorphisms.

\subsection{Generalized Dehn twists and  homology cobordisms} \label{subsec:gDt_and_hc}

On the one hand, we recall that the group $\calM$ is generated by (right-handed) Dehn twists $t_\gamma$
along simple closed curves $\gamma \subset \Sigma$. 
The value of $t_\gamma \in \Aut(\pi)$ on (the homotopy class of) a loop  in  general position with $\gamma$ 
is, roughly speaking, obtained by inserting a copy of $\gamma^\pm$ at each intersection point of that loop with $\gamma$. 
For instance, on the first homology group
$$
H:= H_1(\Sigma;\Z) \cong  \pi/[\pi,\pi],
$$
one gets the well-known formula
\begin{equation} \label{eq:transvection}
t_\gamma(x) =x + \omega([\gamma],x)\cdot [\gamma] \quad \hbox{for all } x\in H
\end{equation}
where $\omega\colon H \times H \to \Z$ denotes the homology intersection form of $\Sigma$.
A~generalization of the transvection formula \eqref{eq:transvection} for
the  group $\pi$ itself  was obtained in \cite{KK14} (see also \cite{MT13} and~\cite{KK15}).
This formula asserts that the extension of $t_\gamma\colon \pi \to \pi $ 
(by linearity and continuity) to the complete algebra $\hat A$ 
is the exponential of a certain derivation
$$
 D_{\gamma}  \colon \hat A  \longrightarrow \hat A.
$$
The latter  is fully  defined from  the free homotopy class of $\gamma$ using certain intersection operations on $\pi$ which refine $\omega$ (see Section \ref{sec:gdt} for further detail).
Therefore a \emph{generalized Dehn twist} $t_\gamma$ can be defined as well  for any (possibly non-simple) loop $\gamma$ by setting
$$
t_\gamma := \exp(  D_{\gamma} ) \in \Aut( \hat A). 
$$
Note that, if $\gamma$ has self-intersections, this automorphism $t_\gamma$ is in general not induced by an element of $\calM$, 
i.e$.$ it does not necessarily  preserve $\pi \subset \hat A$
as shown in \cite{Ku13, MT13, KK15}.
Nevertheless $t_\gamma$ is known to preserve the {Malcev completion} of $\pi$.
Therefore, we can consider the subgroup
$$
\calW  \subset  \Aut(\hpi)
$$
that is generated by generalized Dehn twists. This is a first generalization of the Dehn--Nielsen representation:
\begin{equation} \label{eq:DN1}
\xymatrix{
\calM \ar[r]^-{\rho} \ar[d]_-{j\rho} & \Aut(\pi) \ar[d]^-j \\
\calW\, \ar@{^{(}->}[r] & \Aut(\hpi)
}
\end{equation}

On the other hand, we  consider \emph{homology cobordisms} of $\Sigma$ which are pairs $(C,c)$
consisting of a compact oriented 3-manifold   $C$ and  an orientation-preserving diffeomorphism
$c\colon {\partial (\Sigma \times [-1,+1])} \to \partial C$
such that the inclusion maps $c_{\pm}\colon \Sigma \to C$ defined by $c_\pm (x) := c(x,\pm 1)$ induce isomorphisms in homology.
Thus $C$ is a cobordism between two copies of $\Sigma$, namely $\partial_+C:= c_+(\Sigma)$ and  $\partial_-C:= c_-(\Sigma)$:
$$
\labellist
\scriptsize\hair 2pt
\pinlabel {$\partial_+ C$} [r] at 1 186
 \pinlabel {$\partial_-C$} [r] at 0 75
 \pinlabel {$C$}  at 92 128
 \pinlabel {$c_+$} [r] at 91 222
 \pinlabel {$c_-$} [r] at 90 38
 \pinlabel {$\Sigma$} [r] at 2 255
 \pinlabel {$\Sigma$} [r] at 3 4
\endlabellist
\centering
\includegraphics[scale=0.28]{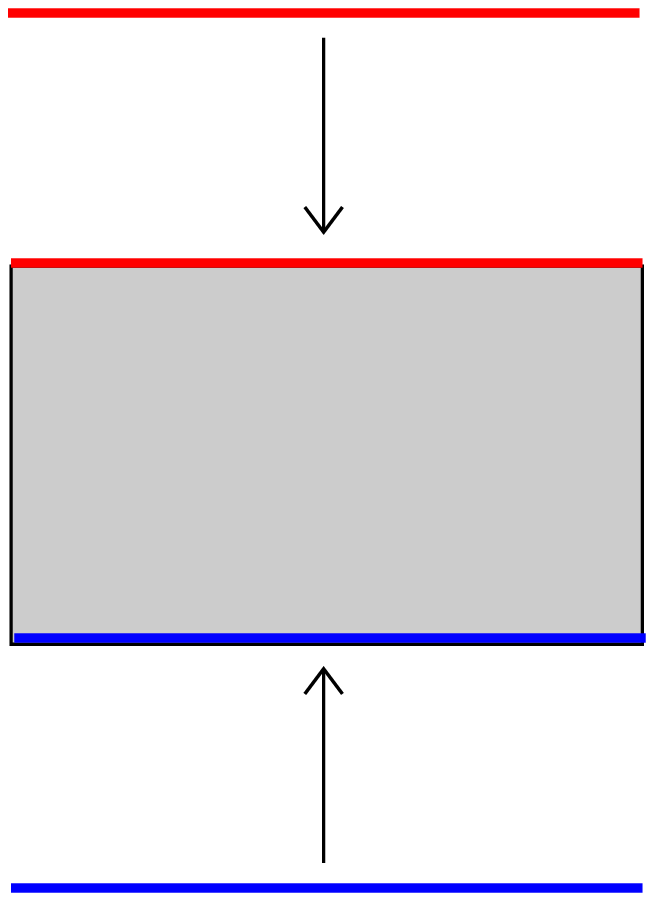}
$$
Two homology cobordisms $(C,c)$ and $(D,d)$ of $\Sigma$
can be \emph{multiplied} by gluing $D$ \lq\lq on the top of\rq\rq{} $C$, using the boundary parametrizations $d_-$ and $c_+$ to identify $\partial_- D$ with $\partial_+ C$:
$$
\centre{\labellist
\scriptsize\hair 2pt
 \pinlabel {$C$}  at 92 128
 \pinlabel {$c_+$} [r] at 91 222
 \pinlabel {$c_-$} [r] at 90 38
 \pinlabel {$\Sigma$} [r] at 2 255
 \pinlabel {$\Sigma$} [r] at 3 4
\endlabellist
\centering
\includegraphics[scale=0.28]{cobordism}} \quad \circ \quad  
\centre{\labellist
\scriptsize\hair 2pt
 \pinlabel {$D$}  at 92 128
 \pinlabel {$d_+$} [r] at 91 222
 \pinlabel {$d_-$} [r] at 90 38
 \pinlabel {$\Sigma$} [r] at 2 255
 \pinlabel {$\Sigma$} [r] at 3 4
\endlabellist
\includegraphics[scale=0.28]{cobordism}} 
\quad := \quad 
\centre{\centering
\labellist
\scriptsize\hair 2pt
 \pinlabel {$\Sigma$} [r]  at 3 362
 \pinlabel {$\Sigma$} [r] at 2 4
 \pinlabel {$D$}  at 93 232
 \pinlabel {$C$}  at 92 126
 \pinlabel {$c_{-}$} [r] at 92 37
 \pinlabel {$d_+$} [r] at 91 331
\endlabellist
\centering
\includegraphics[scale=0.28]{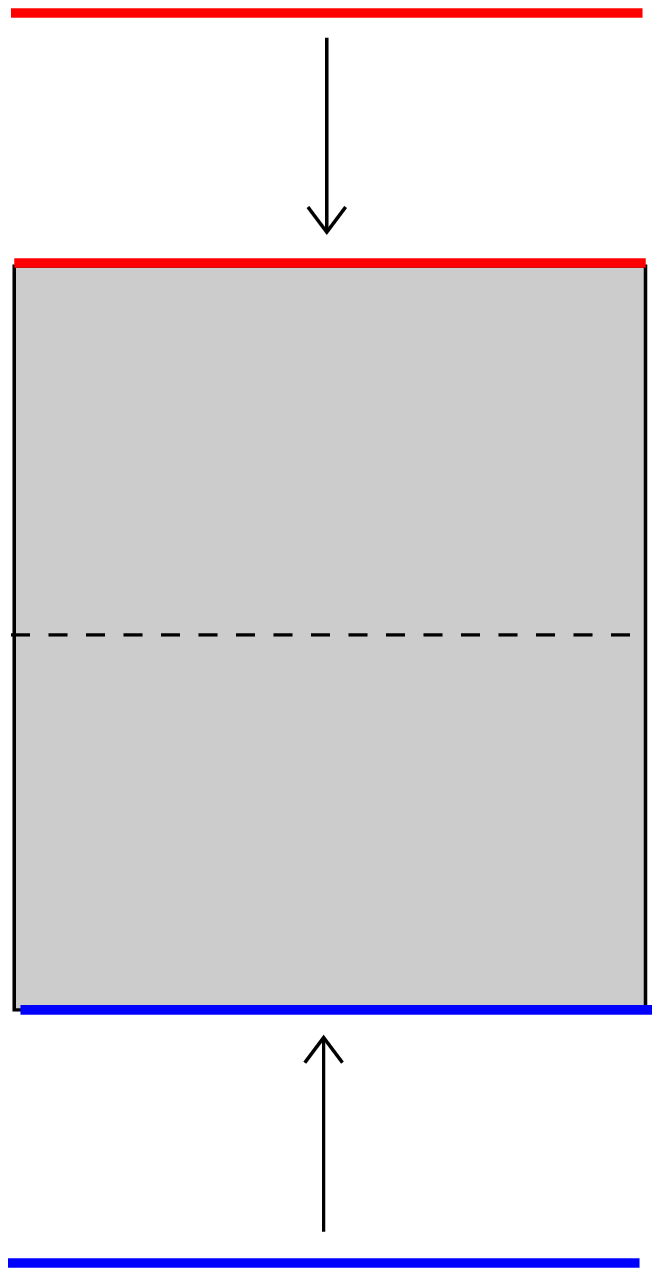}}
$$
With this operation $\circ$, the set 
$$
\calC := \{ \hbox{\small homology cobordisms of $\Sigma$} \} / \hbox{\small diffeomorphism}
$$
constitutes a monoid into which the mapping class group embeds by the ``mapping cylinder'' construction:
specifically, a diffeomorphism $f\colon \Sigma \to \Sigma$ defines a homology cobordism
whose underlying $3$-manifold is the cylinder $U:= \Sigma \times [-1,+1 ]$
and whose boundary parametrization  ${\partial (\Sigma \times [-1,+1])} \to {\partial U }$ is
given by $f$ on the top surface $\Sigma \times \{+1\}$ and by the identity elsewhere. 
Thus we  will view $\calM$ as a submonoid of $\calC$.
By Stallings' theorem~\cite{Sta}, for each $(C,c)\in \calC$ and  $k\ge 1$,
 the maps $c_+$ and $c_-$ induce isomorphisms between the $k$th nilpotent quotient of  $\pi$ and that of the fundamental group of $C$.
(Here and henceforth, a homology cobordism $C$ is 
based at any point of the arc $c(\{\star\}\times [-1,+1])$, which we will abusively denote by $\star\in C$.)
Thus there is a monoid homomorphism
\[
\rho_k \colon \calC \longrightarrow \Aut(\pi/\Gamma_{k+1}\pi),
\quad  (C,c)  \longmapsto \big( {(c_-)}^{-1} \circ c_+\big).
\]
Using \eqref{eq:Malcev},
the family $(\rho_k)_{k\geq 1}$ induces a homomorphism
$\hat\rho \colon \calC \to \Aut(\hpi)$, which gives a second generalization of the Dehn--Nielsen representation:
\begin{equation} \label{eq:DN2}
\xymatrix{
\calM \ar[r]^-{\rho} \ar@{^{(}->}[d]  & \Aut(\pi) \ar[d]^-j \\
\calC\, \ar[r]^-{\hat \rho} & \Aut(\hpi)
}
\end{equation}

\subsection{Main result of the paper}

In order to compare those two generalizations \eqref{eq:DN1} and \eqref{eq:DN2}
of the Dehn--Nielsen representation,
i.e$.$ to relate in some way  generalized Dehn twists to homology cobordisms, we need to add some homological conditions on both sides.

On the one hand, we  shall only consider generalized Dehn twists along \emph{null-homologous} closed curves.
On the other hand, we shall restrict ourselves to  \emph{homology cylinders}, 
which are homology cobordisms $(C,c)$  such that ${c_+=c_-}$ in homology.
The set
$$
\mathcal{IC} := \{ \hbox{\small homology cylinders of $\Sigma$} \} / \hbox{\small diffeomorphism}
$$
is a submonoid of $\calC$; its intersection with $\calM$ is the subgroup acting trivially in homology,
namely the \emph{Torelli group} $\calI$ of $\Sigma$. 

Let $\operatorname{pr}\colon U \to \Sigma$ denote the cartesian projection of $U= \Sigma \times [-1,+1]$.
A~\emph{knot resolution} of a closed curve  $\gamma \subset \Sigma$ is a  knot $L\subset U$ such that $\operatorname{pr}(L)=\gamma$.
Of course, $\gamma$ is null-homologous in $\Sigma$ if and only if $L$ is null-homologous in $U$
and, in that case,  the cobordism $U_L$ obtained from $U$ by  surgery along $L$ with  framing $+1$ or~$-1$ is  a homology cylinder.
The curve~$\gamma$ (resp$.$ the knot $L$) is said to be \emph{of nilpotency class $\geq k$} if, 
when it is oriented and connected to $\star$ by an arc (in an arbitrary way),
its homotopy class $[\gamma] \in \pi$ (resp$.$ $[L] \in \pi_1(U,\star) \cong  \pi$) belongs to $\Gamma_k \pi$.
With this terminology, we can state our most general result.\\

\noindent
\textbf{Theorem A.} {\it Let $k\geq 2$ be an integer 
and let $\gamma \subset \Sigma$  be a closed curve of nilpotency class ${\geq k}$.
Then, for any knot resolution $L\subset U$ of $\gamma$   with framing $\varepsilon \in \{-1,+1\}$,
 the following diagram is commutative:
$$
\xymatrix{
{\pi}/{\Gamma_{2k+1} \pi} \ar[d]_-{\rho_{2k}(U_L)}  \ar[r] & \hpi / \hpi_{2k+1} \ar[d]^-{(t_\gamma)^{\varepsilon}} \\
{\pi}/{\Gamma_{2k+1} \pi}   \ar[r] & \hpi / \hpi_{2k+1} 
}
$$
}

When $\gamma$ is  a bounding \emph{simple} closed curve in $\Sigma$, 
any knot resolution $L$ is obtained by pushing $\gamma \times \{ +1 \} \subset \Sigma \times [-1,+1]$ off the boundary,
and we know from Lickorish's trick that $U_L$ is merely the mapping cylinder of the ordinary Dehn twist $(t_\gamma)^{\varepsilon}$:
so  we have $$\rho_j(U_L)= (t_\gamma)^{\varepsilon} \in \Aut(\pi/\Gamma_j\pi) \hbox{ for any $j\geq 2$} $$ 
and Theorem A immediately follows in that case. 

However, the previous argument can not be extended to a curve $\gamma$ with \emph{self-intersection} points: 
in particular, it is not clear at all why $\rho_{2k}(U_L)$ should only depend on the homotopy class of $L$.
Actually, we will give two explicit formulas which compute $t_\gamma$ and $\rho_{2k}(U_L)$, respectively,  
in terms of $[\gamma]=[L]$ modulo $\Gamma_{k+2} \pi$:
see  Theorem \ref{thm:gdtaction} and Theorem \ref{thm:HCmonodromy} for precise statements.  
Then Theorem A will follow from the simple  observation that these two formulas are the same.

As an immediate consequence of Theorem A,  $t_\gamma \in \Aut(\hpi/\hpi_{2k+1})$
restricts to an automorphism of $\pi/\Gamma_{2k+1} \pi$ whenever $\gamma$ is of nilpotency class~${\geq k}$. More importantly,
Theorem A and its two companion formulas have applications to the study of Johnson homomorphisms.
We now briefly review the latter  referring  to the survey papers \cite{Satoh,HM} for further detail.

\subsection{Johnson homomorphisms} \label{subsec:Joh}

Recall first, from the works of Johnson and Morita \cite{Joh80,Joh83,Mor93}, that the \emph{Johnson filtration}
$$
\calM= \calM[0] \supset \calM[1] \supset \calM[2] \supset  \cdots  
$$
consists of the kernels of the actions of the mapping class group on the successive nilpotent quotients of $\pi$:
specifically, for all $j\geq 1$, we have
$$
 \calM[j] := \ker \big(\rho_j\colon \calM \longrightarrow \Aut(\pi/\Gamma_{j+1}\pi) \big).
$$
Furthermore, recall that the restriction of $\rho_{j+1}$ to  the $j$th term of this filtration
is encoded by the \emph{$j$th Johnson homomorphism}
$$
\tau_j\colon \calM[j] \longrightarrow \Hom\big(H, \Gamma_{j+1} \pi/\Gamma_{j+2}\pi \big) 
\cong H \otimes \big( \Gamma_{j+1} \pi/ \Gamma_{j+2}\pi \big) ;
$$
see Section \ref{sec:applications} for the definition.
Since $\pi$ is a free group, the associated graded of its lower central series (with its canonical structure of Lie ring)
can be identified with the graded  Lie ring
$$
\frakL = \bigoplus_{j=1}^{+\infty} \frakL_j
$$
freely generated by the abelian group $ \frakL_1  =H$: 
therefore the values of $\tau_j$ can be identified to  derivations of $\frakL$ that increase degrees by $j$.
We denote by 
$$
\frakH = \bigoplus_{j=1}^{+\infty} \frakH_j
$$
 the Lie ring of \emph{symplectic derivations},
i.e$.$ derivations of  $\frakL$ that strictly increase degrees and vanish on the bivector
$$
\omega \in \Lambda^2H \cong  \frakL_2
$$
dual to the homology intersection form. It turns out that $\tau_j$ is valued in~$\frakH_j$.
Then, an important and long-standing problem in the study of mapping class groups is to identify the image of $\tau_j\colon \calM[j] \to \frakH_j$.

The previous notions extend  directly to the framework of homology cobordisms, as observed by Garoufalidis and Levine \cite{GL}.
Thus the \emph{Johnson filtration} is defined as well on the monoid of homology cobordisms:
$$
\calC= \calC[0] \supset \calC[1] \supset \calC[2] \supset  \cdots  
$$
and, for every $j\geq 1$, there is also the \emph{$j$th Johnson homomorphism}
$$
\tau_j\colon \calC[j] \longrightarrow  \frakH_j
$$
which, in contrast with the case of $\calM$, turns out to be surjective \cite{GL,Hab}.
Note that $\calC[1]$ is the monoid $\calI\calC$ of homology cylinders,
and $\calM[1]= \calC[1] \cap \calM$ is the Torelli group $\calI$ of $\Sigma$.

The same notions  extend to the setting of generalized Dehn twists too, but it is then necessary to consider Malcev completions:
thus, by defining $\calW[j]$ as the kernel of the canonical homomorphism
$\calW \subset \Aut(\hpi) \to \Aut(\hpi/\hpi_{j+1})$ for every $j\geq 1$, we get
the \emph{Johnson filtration}
$$
\calW= \calW[0] \supset \calW[1] \supset \calW[2] \supset  \cdots ;
$$ 
then, for every $j\geq 1$,  the \emph{$j$th Johnson homomorphism}
$$
\tau_j\colon \calW[j] \longrightarrow  \frakH_j^\Q
$$
takes values in  the degree $j$ part of  the Lie algebra of \emph{symplectic derivations},
i.e$.$ derivations of the free Lie algebra $\frakL^\Q$ (with rational coefficients)
generated  by $H^\Q := H\otimes \Q$ that vanish on $\omega$.

\subsection{Jacobi diagrams} \label{subsec:Jacobi}

The Lie ring $\frakH$, and its rational version $\frakH^\Q$, have a very useful description in terms of ``Jacobi diagrams'',
which appear in Kontsevich's graph homology  and in the theory of finite-type invariants.
This description stems out from the simple fact that a linear combination~$T$
of planar binary rooted trees  with $j$ leaves colored by $H$ defines  an element $\operatorname{comm}(T)\in \frakL_j$, for instance:

$$
\operatorname{comm}\Bigg(\begin{array}{c}
\labellist \small \hair 2pt
\pinlabel {$h_1$} [b] at 2 162
\pinlabel {$h_2$} [b] at 111 162
\pinlabel {$h_3$} [b] at 167 162
\pinlabel {$h_4$} [b] at 222 162
\pinlabel {root} [t] at 108 5
\endlabellist
\includegraphics[scale=0.3]{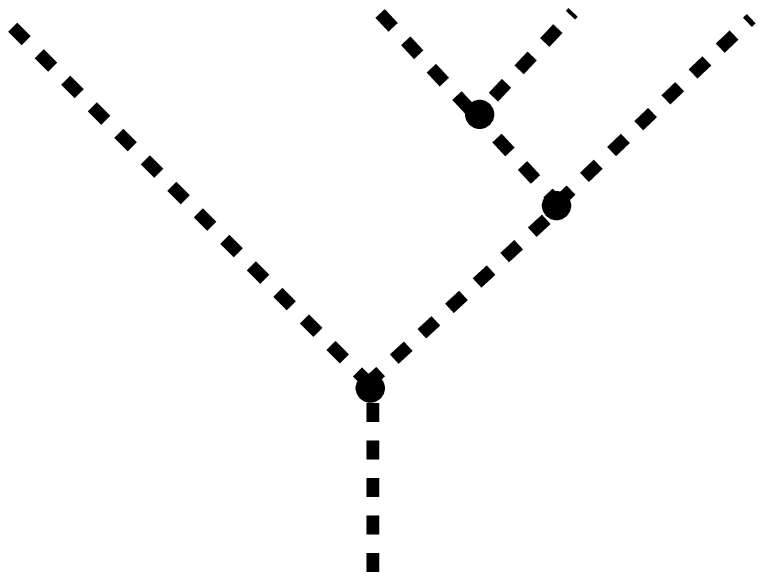} 
\end{array}\Bigg) =  [h_1,[[h_2,h_3],h_4]]
\quad \quad  (\hbox{with } h_1,\dots,h_4 \in H).
$$

A \emph{Jacobi diagram} is a unitrivalent graph whose trivalent vertices are oriented
(i.e$.$ edges are cyclically ordered around each trivalent vertex); it is said  \emph{$H$-colored} if it comes with a map
from the set of its univalent vertices to~$H$; its \emph{degree} is the number of trivalent vertices.
In this paper, we will always assume that Jacobi diagrams are finite, tree-shaped and connected. 
For example, here is a Jacobi diagram of degree $3$:

$$
{\labellist \small \hair 2pt
\pinlabel {$h_1$} [tr] at 0 0
\pinlabel {$h_2$} [br] at 0 59
\pinlabel {$h_3$} [b] at 69 76
\pinlabel {$h_4$} [bl] at 138 63
\pinlabel {$h_5$} [tl] at 135 0
\endlabellist}
\includegraphics[scale=0.3]{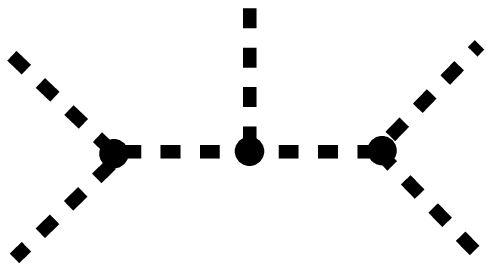}
\quad \quad \quad \quad  (\hbox{with } h_1,\dots,h_5 \in H)
$$
\vspace{-0.3cm}

\noindent
where, by convention, vertex orientations are given by the trigonometric orientation of the plan.
Let
$$
\calT = \bigoplus_{d=1}^{+ \infty} \calT_d
$$ 
be the graded abelian group of $H$-colored Jacobi diagrams modulo the AS, IHX and multilinearity relations: \\[0.2cm]

\begin{center}
\labellist \small \hair 2pt
\pinlabel {AS} [t] at 102 -5
\pinlabel {IHX} [t] at 543 -5
\pinlabel {multilinearity} [t] at 1036 -5
\pinlabel {$= \ -$}  at 102 46
\pinlabel {$-$} at 484 46
\pinlabel {$+$} at 606 46
\pinlabel {$=0$} at 721 46 
\pinlabel {$+$} at 1106 46
\pinlabel {$=$} at 961 46
\pinlabel{$h_1+h_2$} [b] at 881 89
\pinlabel{$h_1$} [b] at 1042 89
\pinlabel{$h_2$} [b] at 1170 89
\endlabellist
\centering
\includegraphics[scale=0.3]{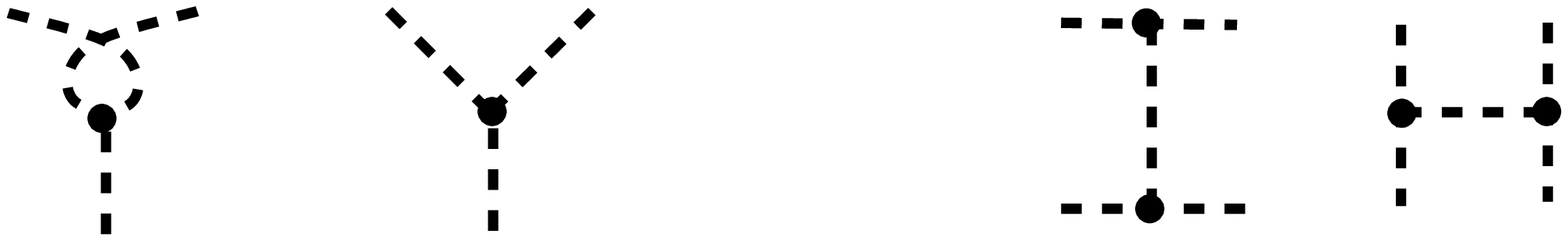}
\end{center}
\vspace{0.5cm}

With the above definitions, we can now introduce the homomorphism 
$
\eta_j \colon \calT_j {\longrightarrow}\ \frakH_j \subset H\otimes \frakL_{j+1}.
$
For any  $H$-colored Jacobi diagram $T$, set
$$
\eta_j(T) := \sum_{v} \operatorname{col}(v) \otimes \operatorname{comm}(T_v)
$$
where the sum is over all univalent vertices $v$ of $T$, $ \operatorname{col}(v) $ denotes the element of $H$ carried by $v$
and $T_v$ is the  tree $T$ rooted at $v$.
There is also a rational version of the previous map
$$
\eta_j^\Q \colon \calT_j^\Q {\longrightarrow} \frakH_j^\Q \subset H^\Q\otimes \frakL_{j+1}^\Q
$$
which is known to be an isomorphism  \cite{HabM,GL,Hab,HP}. 
As for the integral coefficients, Levine proved in \cite{Levine} that $\eta_{2k-1}$ is surjective for any $k\geq 1$:
then,  by the IHX relation, $\frakH_{2k-1}$ is generated by  the elements of the form
\begin{equation} \label{eq:gen_odd}
\eta_{2k-1} (\widetilde{x} \hbox{\textbf{\,-\! -\! -}\,} \widetilde{y}) \quad \hbox{for all } x \in \frakL_{k} \hbox{ and } y \in \frakL_{k+1}
\end{equation}
where  $\widetilde{x}$ (resp$.$ $\widetilde{y}$) is a linear combination of  $H$-colored  planar binary rooted trees such that $\operatorname{comm}(\widetilde{x})=x$
(resp$.$ $\operatorname{comm}(\widetilde{y})=y$),
and $\widetilde{x} \hbox{\textbf{\,-\! -\! -}\,} \widetilde{y} \in \calT_{2k-1}$ is  obtained by gluing ``root-to-root''  $\widetilde{x}$ and $\widetilde{y}$.
Moreover, for any $k\geq 2$, it  follows from the results of \cite{Levine} that $\frakH_{2k-2}$ is generated by  the elements of the form
\begin{equation} \label{eq:gen_even}
\frac{1}{2}\eta_{2k-2}(\widetilde{x} \hbox{\textbf{\,-\! -\! -}\,} \widetilde{x}) \quad \hbox{for all } x \in \frakL_{k}
\end{equation}
where  $\widetilde{x} \hbox{\textbf{\,-\! -\! -}\,} \widetilde{x} \in \calT_{2k-2}$ is  obtained by gluing ``root-to-root'' two copies of~$\widetilde{x}$.

\subsection{Values of the Johnson homomorphisms}

As was mentioned above, the Johnson homomorphism $\tau_j\colon \calM[j] \to \frakH_j$ is  not surjective in general,
but Garoufalidis and Levine, and later Habegger, proved that surjectivity 
holds for homology cylinders.
The proof of the surjectivity of  $\tau_j\colon \calC[j] \to \frakH_j$ 
given in~\cite{GL} is based on the adaptation of results in surgery theory to the dimension three.
(This is in a way similar to the realization
of ``nilpotent homotopy types'' of $3$-manifolds  by Turaev  \cite{Tu84}.)
As for the proof contained in \cite{Hab}, it reduces the problem of the surjectivity of the Johnson homomorphisms
to the surjectivity of the Milnor invariants, using a ``Milnor--Johnson'' correspondence between string-links and homology cobordisms 
through which the eponymous invariants correspond each other.

The surjectivity of the Milnor invariants for links is due to Orr \cite{Orr89};
next, Cochran showed how to realize any element in the target of Milnor invariants by an explicit link 
using the so-called ``Bing double'' construction \cite{Coc90,Coc91}; 
later, Habegger and  Lin showed the same realization results  for string links \cite{HL}.
Furthermore, it is known that Cochran's construction  can be reformulated 
in terms of surgeries along ``tree claspers'' in the sense of Habiro~\cite{Habiro}.

Here, as an application of Theorem A (and its  two companion formulas), 
we prove the surjectivity of  $\tau_j\colon \calC[j]  \to \frakH_j$ in an alternative way 
and in relation with the homomorphism $\tau_j\colon \calW[j] \to \frakH_j^\Q$.
Indeed, this surjectivity  will follow from the computation of $\tau_j$ 
for certain ``nilpotent'' surgeries on the usual cylinder $U= \Sigma \times [-1,+1]$.
This requires to distinguish the case where $j$ is even from the case where it is odd:\\

\noindent
\textbf{Theorem B.} {\it  Let $k\geq 2$ be an integer 
and let $\gamma \subset \Sigma$  be a closed curve of nilpotency class ${\geq k}$.
Let $L \subset U$ be a knot resolution of $\gamma$ with framing $\varepsilon \in \{-1,+1\}$.
Then, we have $U_L  \in \mathcal{C}[2k-2]$, $t_{\gamma} \in \calW[2k-2]$ and 
\begin{equation} \label{eq:even}
\tau_{2k-2} (U_L) = \tau_{2k-2}\big((t_{\gamma})^{\varepsilon}\big) 
=  \frac{\varepsilon}{2}\, \eta_{2k-2}\big( \{ [\gamma] \}_k \hbox{\textbf{\,-\! -\! -}\,} \{ [\gamma] \}_k\big)
\end{equation}
where $\{ [\gamma] \}_k \in \Gamma_k \pi/ \Gamma_{k+1}\pi \cong  \frakL_k$ 
denotes the class of $[\gamma]$ modulo $ \Gamma_{k+1}\pi$.}\\

\noindent
\textbf{Theorem C.} 
{\it  Let $k\geq 2$ be an integer 
and let $\gamma_+,\gamma_- \subset \Sigma$  be  closed curves of nilpotency class ${\geq k}$.
Assume that $\delta : = [\gamma_+] [{\gamma_-}]^{-1}  \in \Gamma_{k+1}\pi$.
Let $L_{\pm} \subset U$ be a knot resolution of $\gamma_{\pm}$ with  framing number $\pm 1$.
Then, we have  $U_{L_-} U_{L_+} \in \calC[2k-1]$,  $(t_{\gamma_-})^{-1} t_{\gamma_+} \in \calW[2k-1]$ and
\begin{equation} \label{eq:odd}
\tau_{2k-1}(U_{L_-} U_{L_+}) = \tau_{2k-1}\big((t_{\gamma_-})^{-1} t_{\gamma_+}\big)
= \eta_{ 2k-1} \big( \{ [\gamma_+] \}_k \hbox{\textbf{\,-\! -\! -}\,} \{ \delta \}_{k+1}\big)
\end{equation}
where $\{ [\gamma_+] \}_k \in \Gamma_k \pi/ \Gamma_{k+1}\pi \cong  \frakL_k$ denotes the class of $[\gamma_+]$ modulo $ \Gamma_{k+1}\pi$
and  $\{ \delta \}_{k+1}$ has a similar meaning.}\\

Theorems B \& C can be regarded as far-reaching generalizations of two well-known results:
\eqref{eq:odd} generalize the computation of $\tau_1$ by Johnson \cite{Joh80} on so-called ``BP maps'' 
which generate the Torelli group $\calI=\calM[1]$, and \eqref{eq:even} generalize
the computation of $\tau_2$ by Morita \cite{Mor89} on so-called ``BSCC maps'' 
which generate the Johnson subgroup $\calM[2]$.

The kind of ``nilpotent'' surgeries that are involved in Theorems B \&~C, 
firstly appeared in the work of Cochran, Gerges and Orr \cite{CGO}.
They are much more general than the $3$-dimensional constructions 
that have been considered so far for realizing values of the Milnor invariants/Johnson homomorphisms,
and which have been evoked above.  
In fact, our topological arguments  to prove Theorems A, B and  C will be self-contained
and they will not involve $3$-dimensional techniques 
such as Bing doubling,  clasper surgery or theory of gropes.
Nonetheless, the reader familiar with these
``geometric'' versions of  commutator calculus
may observe similarities between these techniques  and the $3$-dimensional arguments that we have developed.
In the instances where those ``nilpotent'' surgeries arise from clasper surgeries,
the conclusions of Theorems B \&~C for homology cylinders could also be proved by clasper  techniques:
see Remark \ref{rem:special_cases_B} and Remark \ref{rem:special_cases_C}.

We  can now  conclude to the  surjectivity of the Johnson homomorphisms  for homology cylinders as follows.
On the one hand, since $\frakH_{2k-2}$ is generated by  the elements of the form~\eqref{eq:gen_even}, 
Theorem B implies that $\tau_{2k-2}$ is surjective. 
On the other hand, since $\frakH_{2k-1}$ is generated by  the elements of the form~\eqref{eq:gen_odd}, 
 Theorem C gives the surjectivity of $\tau_{2k-1}$.

Finally, Theorems B \& C also have  implications for generalized Dehn twists.
Indeed, we deduce that $\tau_{2k-2}(t_{\gamma})$ is  integral for any closed curve $\gamma$ of nilpotency class $\geq k$,
(i.e$.$ it belongs to $\frakH_{2k-2} \subset \frakH_{2k-2}^\Q$),
and we also obtain that $\tau_{2k-1}\big((t_{\gamma_-})^{-1} t_{\gamma_+}\big)$ is integral 
for any closed curves $\gamma_\pm$ of nilpotency class $\geq k$
whose ``ratio'' is in $\Gamma_{k+1} \pi$.

\subsection{Organization of the paper}
 
 In Section \ref{sec:gdt}, we review the definition of a generalized Dehn twist $t_\gamma$ along a closed curve $\gamma \subset \Sigma$;
 assuming that $\gamma$ is of nilpotency class $\geq k$, we prove Theorem \ref{thm:gdtaction} 
 which computes $t_\gamma \in \Aut(\hpi/\hpi_{2k+1})$
 explicitly in terms of $\{ [\gamma] \}_{k+1} \in \pi/ \Gamma_{k+2}\pi$.

The next two sections are devoted to the study of homology cylinders.
The key feature of our study is the use of the $k$th Morita homomorphism $M_k\colon \calC[k] \to H_3(\pi/\Gamma_{k+1}\pi;\Z)$,
which is equivalent to the monoid homomorphism $\rho_{2k}\colon  \calC[k] \to \Aut(\pi/\Gamma_{2k+1}\pi)$. 
In Section \ref{sec:hc}, after having recalled all the necessary definitions,
we prove that the variation of $\rho_{2k}$ under surgery along a $(\pm 1)$-framed knot of nilpotency class $\geq k$ 
only depends on the homotopy class of that knot.
Next, in Section \ref{sec:main_theorem}, 
we prove Theorem \ref{thm:HCmonodromy} which, for a $(\pm 1)$-framed knot $L\subset U$ of nilpotency class $\geq k$,
 computes $\rho_{2k}(U_L)$ explicitly in terms of $\{ [L] \}_{k+1} \in \pi/ \Gamma_{k+2}\pi$. This will conclude the proof of Theorem A.

In Section  \ref{sec:diag_descriptions}, 
we reformulate Theorem A in a diagrammatic way using a ``symplectic'' expansion of the free group $\pi$.
In Section \ref{sec:applications}, first we
review the definition of the Johnson homomorphisms.
Then, using the diagrammatic version of Theorem A, we prove Theorem~B and Theorem~C 
which we next generalize to Theorem \ref{thm:generalized_B} and Theorem \ref{thm:generalized_C}, respectively.
These generalizations, where $U$ is replaced by any homology cylinder of appropriate depth in the Johnson filtration,
is the only place where we need a little bit of  clasper techniques.

\subsection{Remark}
\label{subsec:remark}
Some of the above results (including Theorem A,  and its  two companion formulas: Theorems \ref{thm:gdtaction} \& \ref{thm:HCmonodromy})
hold as well for any compact oriented surface $\Sigma$ with arbitrary  non-empty boundary.
But, the other results (including Theorems B \& C) have to be adapted in the multi-boundary setting:
 they have to be  stated with appropriate analogues of Johnson homomorphisms,
and to be proved with suitable analogues of ``symplectic'' expansions of the free group $\pi$.
In the particular case of a disk $\Sigma$ with punctures, 
Johnson homomorphisms turn into Milnor invariants of string-links in homology balls,
while ``symplectic'' expansions have to be replaced by the ``special'' expansions appearing in \cite{Ma_Formal,AKKN}.

\subsection{Conventions}
\label{subsec:convention}

Given a group $G$, we write  $\{x\}_k$ for any integer $k\geq 0$ and any $x\in G$
to denote the class  of $x$ modulo $\Gamma_{k+1} G$.
For any $x,y\in G$, we denote
\[
[x,y]:=xyx^{-1}y^{-1}, \quad x^y := yxy^{-1}, \quad \text{and} \quad \bar{x} := x^{-1}.
\]
Unless explicitly stated otherwise, group (co)homology is taken
with ordinary coefficients in $\Z$.

Given a $\Q$-vector space $V$ with filtration $V=V_0 \supset V_1 \supset V_2 \supset \dots$,
we write $u \equiv_j v$  for any integer $j\ge 0$ and $u,v\in V$ if $u-v \in V_j$.

\section{Generalized Dehn twists}   \label{sec:gdt}

We start with recalling intersection operations for loops in $\Sigma$, 
which we need to define and study generalized Dehn twists.
On this purpose, we fix the following notations.

\subsection{Notations}

The inverse path of an oriented path $\alpha$ in the surface $\Sigma$ is denoted by $\bar{\alpha}$.
For an immersed oriented curve $\alpha$ in $\Sigma$ and two simple points $p,q$ on it, we denote by $\alpha_{pq}$ the path in $\Sigma$ which is obtained by traversing $\alpha$ from $p$ to $q$.
(When $\alpha$ is not a closed curve, we implicitly assume that $p$ appears before $q$ in a parametrization of $\alpha$.)
Similarly, for an immersed oriented closed curve $\alpha$ in $\Sigma$ and a simple point $p$ on it, we denote by $\alpha_p$ the loop $\alpha$ based at $p$.
The product of paths always reads from left to right.
That is, if $\alpha$ and $\beta$ are paths in $\Sigma$ such that the terminal point of $\alpha$ is the same as the initial point of $\beta$, then $\alpha\beta$ denotes the path which traverses first $\alpha$ and then $\beta$.

Let $\alpha$ and $\beta$ be two immersed oriented curves in $\Sigma$ 
which intersect transversely at a double point $p\in \operatorname{int}(\Sigma)$.
We define the number $\varepsilon(p;\alpha,\beta) \in \{ \pm 1\}$ by setting $\varepsilon(p;\alpha,\beta):=+1$ if the ordered pair of tangent vectors of $\alpha$ and $\beta$ at $p$ gives a positive frame for the tangent space $T_p\Sigma$, and $\varepsilon(p;\alpha,\beta):=-1$ otherwise.

 The homotopy class of a based loop $\alpha \subset \Sigma$ is denoted by $[\alpha]$.
However, when there is no fear of confusion, we drop the brackets and loosely denote it by $\alpha$.

\subsection{Intersection operations for loops in a surface}

\label{subsec:intop}

Let $A:=\Q [\pi]$ be the group algebra of $\pi=\pi_1(\Sigma,\star)$ where $\star \in \partial \Sigma$,
and let $\mathfrak{g}=\mathfrak{g}(\Sigma)$ be the $\Q$-vector space spanned by the homotopy classes of free loops in $\Sigma$.
Recall from \cite[\S 3.2]{KK14} the map
\[
\sigma\colon \mathfrak{g} \otimes A \longrightarrow A
\]
defined in the following way.
Let $\alpha$ be a free loop and $\beta$ a based loop, and assume that their intersections consist of
finitely many transverse double points.
Then,
\[
\sigma(\alpha \otimes \beta) = \sigma(\alpha) \beta :=
\sum_{p\in \alpha \cap \beta} \varepsilon(p; \alpha,\beta)\, \beta_{\star p} \alpha_p \beta_{p\star}.
\]

The space $\mathfrak{g}$ has the structure of a Lie algebra given by Goldman \cite{Go86}, and the map $\sigma$ makes $A$ into a $\mathfrak{g}$-module. Moreover, for any $u\in \mathfrak{g}$, $\sigma(u)$ is a derivation of $A$ and annihilates the boundary-parallel element $\zeta := [\partial \Sigma] \in \pi$.
In summary, we obtain the Lie algebra homomorphism
\begin{equation}
\label{eq:siglie}
\sigma\colon \mathfrak{g} \longrightarrow {\rm Der}_{\zeta}(A), \quad u \longmapsto \sigma(u),
\end{equation}
where the target is the Lie algebra of derivations of $A$ which annihilates $\zeta$.

Let $I$ be the augmentation ideal of $A$.
That is, $I$ is the kernel of the algebra homomorphism $A \to \Q$ that maps any $x\in \pi$ to $1$.
The powers $I^m$, $m\ge 0$, define a multiplicative filtration of $A$.
Let $|\cdot |\colon A \to \mathfrak{g}$ be the projection map 
which is obtained by forgetting the base point of loops representing elements of $\pi$.
The map $|\cdot |$ induces a filtration of the space $\mathfrak{g}$ whose $m$th term is given by $\mathfrak{g}_m := |I^m|$.

As was shown in \cite{KK15, KK16}, the operation $\sigma$ is of degree $(-2)$ 
with respect to the filtrations of $A$ and $\mathfrak{g}$ described above:
\begin{equation}
\label{eq:sigfilt}
\sigma (\mathfrak{g}_m \otimes I^n ) \subset I^{m+n-2},
\quad \text{for any $m,n\ge 0$.}
\end{equation}
(Here we understand that $I^0 = I^{-1} = I^{-2} =A$.)
Let $\hat A$ be the completion of the algebra $A$ with respect to the filtration $\{ I^m \}_m$, and define $\hat{\mathfrak{g}}$ similarly.
The projection $|\cdot |$ naturally induces the map $\hat A \to \hat{\mathfrak{g}}$ 
which we denote in the same way.
Both the spaces $\hat A$ and $\hat{\mathfrak{g}}$ have natural filtrations 
whose $m$th terms are denoted by $\hat{I}^m$ and $\hat{\mathfrak{g}}_m$, respectively.
By property \eqref{eq:sigfilt}, the operation $\sigma$ extends to the completions $\hat A$ and $\hat{\mathfrak{g}}$.
In particular, the homomorphism \eqref{eq:siglie} naturally induces the Lie algebra homomorphism
\begin{equation}
\label{eq:sigliecomp}
\sigma\colon \hat{\mathfrak{g}} \longrightarrow {\rm Der}_{\zeta}(\hat{A})
\end{equation}
(for simplicity we use the same letter $\sigma$).
Here, the target is the space of derivations of $\hat A$ that  annihilates $\zeta$ 
and are continuous with respect to the filtration $\{ \hat{I}^m \}_m$.

In later sections we will use another intersection operation
\[
\kappa\colon A\otimes A \longrightarrow A\otimes A
\]
which was introduced in \cite{MT14,KK15},
and is equivalent to the ``homotopy'' intersection pairing introduced in \cite{Tu78}.
To define this operation, take a second base point $\bullet  \in \partial \Sigma$ different from $\star$
and let $\nu \subset \partial \Sigma$ be the {orientation-preserving} arc connecting $\bullet$ to $\star$.
Let $\alpha$ and $\beta$ be elements in $\pi$.
Using the isomorphism $\pi_1(\Sigma,\star) \cong \pi_1(\Sigma,\bullet)$ given by the path $\nu$, we represent $\alpha$ as a loop in $\Sigma$ based at $\bullet$ and $\beta$ as a loop in $\Sigma$ based at $\star$, for which we use the same letters $\alpha$ and~$\beta$.
Assume also that their intersections consist of finitely many transverse double points.
Then, we set
\begin{equation} \label{eq:kappa}
\kappa(\alpha,\beta) := \sum_{p\in \alpha \cap \beta}
\varepsilon(p;\alpha,\beta)\, \beta_{\star p} \alpha_{p\bullet} \nu \otimes \bar{\nu} \alpha_{\bullet p} \beta_{p \star}.
\end{equation}
This map $\kappa$ is a refinement of $\sigma$:
for all $u,v \in A$ we have
\begin{equation}
\label{eq:kappasigma}
\sigma( |u| )(v) = \kappa'(u,v)\, \kappa''(u,v),
\end{equation}
where we use Sweedler's notation $\kappa(u,v) = \kappa'(u,v)\otimes \kappa''(u,v)$.

A slight variation of $\kappa$ is shown in \cite{MT14} to be a quasi-Poisson double bracket in the sense of van den Bergh \cite{vdB08},
which induces  on the representation spaces of $\pi$ the quasi-Poisson structures defined in \cite{AKSM}.
In the sequel, we will only need the following property of $\kappa$: for any $u,v\in A$, {we have}
\begin{equation}
\label{eq:kappa1}
\kappa(uv,w) = (1\otimes u) \kappa(v,w) + \kappa(u,w) (v\otimes 1).
\end{equation}
(Here, the product is taken in the tensor product of two copies of the algebra~$A$.)
The operation $\kappa$ is a map of degree $(-2)$ with respect to the natural filtration 
on the tensor product $A\otimes A$, as was shown in \cite{KK15, KK16}.

For our purpose, it is convenient to use another variant of $\kappa$.
For $a=a'\otimes a'', b=b'\otimes b'' \in A\otimes A$ and $c\in A$, we denote
\[
a\diamond b := a'b' \otimes b''a'' \in A\otimes A,
\quad \text{and} \quad
a\diamond c:= a'ca'' \in A.
\]
The operations $\diamond$ are associative in the following sense:
if $a,b\in A\otimes A$, then
$(a\diamond b) \diamond c = a \diamond (b\diamond c)$ for any $c\in A$ or any $c\in A\otimes A$.
\begin{dfn}
For $\alpha,\beta \in \pi$, set
\begin{equation}
\label{eq:Phi}
\widetilde{\kappa}(\alpha,\beta) := (1\otimes \beta^{-1}) \diamond \kappa(\alpha,\beta) \diamond (\alpha^{-1} \otimes 1).
\end{equation}
\end{dfn}
\noindent
More explicitly, if the loops $\alpha$ and $\beta$ are arranged as in \eqref{eq:kappa}, we get
\begin{equation}
\label{eq:Phiexplicit}
\widetilde{\kappa}(\alpha,\beta) = \sum_{p\in \alpha\cap \beta} \varepsilon(p;\alpha,\beta)\, \beta_{\star p} \overline{\alpha_{\bullet p}} \nu \otimes \bar{\nu} \alpha_{\bullet p} \overline{\beta_{\star p}}.
\end{equation}

\begin{lem} \label{lem:Phiproperties}
We have the following  for any $\alpha,\alpha_1,\dots,\alpha_n,\beta,x\in \pi$:
\begin{eqnarray*}
(1)&\qquad  \widetilde{\kappa}(\alpha \beta,x)
 &= \ \widetilde{\kappa}(\alpha,x) + \widetilde{\kappa}(\beta,x) \diamond (\bar{\alpha} \otimes \alpha);\\
 (2)&\qquad \widetilde{\kappa}(\alpha_1\cdots \alpha_n,x) &= \
\sum_{i=1}^n \widetilde{\kappa}(\alpha_i,x) \diamond
(\overline{\alpha_1\cdots \alpha_{i-1}} \otimes \alpha_1\cdots \alpha_{i-1});\\
(3)&\qquad \widetilde{\kappa}({\bar \alpha},x) &= \ - \widetilde{\kappa}(\alpha,x) \diamond (\alpha \otimes \bar{\alpha});\\
(4)&\qquad \widetilde{\kappa}([\alpha,\beta],x) &= \ \widetilde{\kappa}(\alpha,x) \diamond (1\otimes 1 - {\bar{\beta}}^{\alpha} \otimes \beta^{\alpha}) \\
&& \quad + \ \widetilde{\kappa}(\beta,x) \diamond ( 1\otimes 1-\alpha^{\beta} \otimes \bar{\alpha}^{\beta}) \diamond (\bar{\alpha} \otimes \alpha).
\end{eqnarray*}
\end{lem}

\begin{proof}
Using the notation $\diamond$, the formula \eqref{eq:kappa1} can be rewritten as
\begin{equation}
\label{eq:kappa2}
\kappa(uv,w) = \kappa(u,w) \diamond (v \otimes 1) + \kappa(v,w) \diamond (1\otimes u).
\end{equation}
Hence
\begin{align*}
\widetilde{\kappa}(\alpha\beta,x) &= (1\otimes {\bar x}) \diamond \kappa(\alpha\beta,x) \diamond (\overline{\alpha\beta} \otimes 1) \\
&= (1\otimes {\bar x}) \diamond \kappa(\alpha,x) \diamond (\beta \otimes 1) \diamond (\overline{\alpha\beta} \otimes 1) \\
& \quad + (1\otimes {\bar x}) \diamond \kappa(\beta,x) \diamond (1\otimes \alpha) \diamond (\overline{\alpha\beta} \otimes 1) \\
&= \widetilde{\kappa}(\alpha,x) + \widetilde{\kappa}(\beta,x) \diamond (\bar{\alpha} \otimes \alpha).
\end{align*}
This proves (1), and (2) follows by a straightforward induction on $n\geq 1$.

To prove (3), observe that $\widetilde{\kappa}(1,x)=0$ since $\kappa(1,x)=0$.
Hence (1) implies~that
\[
0= \widetilde{\kappa}({\bar \alpha}\alpha,x)
=\widetilde{\kappa}({\bar \alpha},x) + \widetilde{\kappa}(\alpha,x) \diamond (\alpha \otimes \bar{\alpha}).
\]
As for the identity (4), it follows from (2) and (3):
\begin{align*}
\widetilde{\kappa}([\alpha,\beta],x) &=
\widetilde{\kappa}(\alpha,x) + \widetilde{\kappa}(\beta,x) \diamond (\bar{\alpha}\otimes \alpha) \\
& \quad +\widetilde{\kappa}({\bar \alpha},x) \diamond (\overline{\alpha\beta}\otimes \alpha\beta) + \widetilde{\kappa}({\bar \beta},x) \diamond ( \bar{\beta}^{\alpha} \otimes \beta^{\alpha}) \\
&= \widetilde{\kappa}(\alpha,x) + \widetilde{\kappa}(\beta,x) \diamond (\bar{\alpha} \otimes \alpha) \\
& \quad -\widetilde{\kappa}(\alpha,x) \diamond (\alpha \otimes \bar{\alpha}) \diamond (\overline{\alpha\beta} \otimes \alpha\beta) \\
& \quad  -\widetilde{\kappa}(\beta,x) \diamond (\beta \otimes \bar{\beta}) \diamond ( \bar{\beta}^{\alpha} \otimes \beta^{\alpha}).
\end{align*}

\up
\end{proof}

\subsection{Generalized Dehn twists}
\label{subsec:gdt}

Let $\gamma \subset \Sigma$ be a closed curve.
Choose an orientation of $\gamma$ and  an \emph{arc-basing at $\bullet$} of $\gamma$,
which means that $\gamma$ is made into a loop based at $\bullet$ by fixing an arc from $\bullet$ to a simple point of $\gamma$.
We denote the resulting element in $\pi_1(\Sigma,\bullet) \cong \pi_1(\Sigma,\star) = \pi$ by the same letter~$\gamma$.
Set
\[
L(\gamma) := \left| \frac{1}{2} (\log \gamma)^2 \right| \in \hat{\mathfrak{g}}_2
\]
where
$$
\log x := \sum_{m\ge 1} \frac{(-1)^{m-1}}{m}  (x-1)^m \in \hat{A} \quad \hbox{ for any } x\in 1 + \hat{I}.
$$
One can see that $L(\gamma)$ is independent of the choices made as above.

It turns out that the exponential of the derivation
\[
D_{\gamma}:=\sigma(L(\gamma)) \in {\rm Der}_{\zeta}(\hat A)
\]
converges and defines a Hopf algebra automorphism of $\hat A$ which preserves $\zeta$.
Therefore, by restriction to the group-like part of $\hat A$,
\[
\exp( D_{\gamma}):=\sum_{n\geq 0} \frac{1}{n!} {D_{\gamma}}^n
\]
defines an automorphism of the Malcev completion $\hpi$.

\begin{dfn}
The \emph{generalized Dehn twist along $\gamma$} is the element
\[
t_{\gamma}:= \exp(D_{\gamma}) \in {\rm Aut}_{\zeta}(\hpi).
\]
\end{dfn}

If $\gamma$ is simple, i.e., if $\gamma$ has no self-intersection, then the generalized Dehn twist $t_{\gamma}$ 
coincides with the image of the ordinary (right-handed) Dehn twist along $\gamma$ by the map $j\circ \rho\colon \calM \to {\rm Aut}(\widehat{\pi})$  that has been considered in Section~\ref{subsec:DN}.
Our terminology comes from this fact.
See \cite{KK14} for further explanations, and see \cite{MT13} for group-theoretical generalizations of that.

We define $\calW$ to be 
the subgroup of ${\rm Aut}_{\zeta}(\widehat{\pi})$ that is generated by generalized Dehn twists.
Since the mapping class group $\calM$ is generated by ordinary Dehn twists, one can regard $\calM$ as a subgroup of $\calW$ through $j\circ \rho$.

\subsection{Actions on quotients of the Malcev completion of $\pi$}
\label{subsec:trunc}

Hereafter let $k\ge 2$ and assume that $\gamma$ is of nilpotency class $\ge k$.
Our goal in this section is to describe the action of $t_{\gamma}$ on the quotient group $\widehat{\pi}/\widehat{\pi}_{2k+1}$.
Since $\gamma-1 \in I^{k}$, we have
\[
\frac{1}{2}(\log \gamma)^2 \equiv_{3k} \frac{1}{2}(\gamma-1)^2 \in \hat A.
\]
(Following the convention of Section \ref{subsec:convention},
the symbol $\equiv_j$ means an equivalence modulo $\hat{I}^j$.) Hence  we have 
\begin{equation}
\label{eq:Dgam}
D_{\gamma} = D'_{\gamma} + \sigma(u)
\quad \hbox{where} \
D'_{\gamma} : = \sigma \left( \frac{1}{2} \big|(\gamma-1)^2 \big| \right)
\end{equation}
and $u=u(\gamma)$ is some element in $\hat{\mathfrak{g}}_{3k}$.

\begin{prop}
\label{prop:tgamx}
 Let $\varepsilon \in \{ -1, +1\}$. 
For any $x\in \pi$, we have
\[
 (t_{\gamma})^{\varepsilon} (x)\, x^{-1} -1 \equiv_{3k-1} \varepsilon\,  \widetilde{\kappa}(\gamma,x) \diamond (\gamma - 1).
\]
\end{prop}

\begin{proof}
First we show that $(t_{\gamma})^{\varepsilon}(x)\, x^{-1} -1 \equiv_{3k-1} \varepsilon\,  D'_{\gamma}(x)x^{-1}$.
Since the derivation $D_{\gamma}$ is of degree $2k-2$,  we have
\[
{D_{\gamma}}^n(x) = {D_{\gamma}}^n(x-1)
\in \hat{I}^{n(2k-2)+1} \subset \hat{I}^{4k-3} \subset \hat{I}^{3k-1}
\]
for any $n\ge 2$. Therefore,
\[
t_{\gamma}(x) = x + D_{\gamma}(x) + \sum_{n\ge 2} \frac{1}{n!} {D_{\gamma}}^n(x)
\equiv_{3k-1} x + D_{\gamma}(x).
\]
By \eqref{eq:Dgam}, we have $D_{\gamma}(x) \equiv_{3k-1} D'_{\gamma}(x)$.
Hence $t_{\gamma}(x) \equiv_{3k-1} x + D'_{\gamma}(x)$.
 Similarly, for $(t_{\gamma})^{-1} = \exp(-D_{\gamma})$, we compute $(t_{\gamma})^{-1}(x) \equiv_{3k-1} x - D'_{\gamma}(x)$.
This proves the assertion. 

Next we prove that $D'_{\gamma}(x)x^{-1} \equiv_{3k-1} \widetilde{\kappa}(\gamma,x) \diamond (\gamma - 1)$, 
which will complete the proof of the proposition.
By \eqref{eq:kappa2}, we have
\begin{align*}
 \kappa \left( \frac{1}{2}(\gamma -1)^2,x \right) 
 = \frac{1}{2} \big( \kappa(\gamma,x) \diamond (1\otimes (\gamma -1)) + \kappa(\gamma,x) \diamond ((\gamma-1) \otimes 1) \big).
\end{align*}
Thus \eqref{eq:kappasigma} shows that
\[
D'_{\gamma}(x) = \sigma \left( \frac{1}{2} \big|(\gamma-1)^2\big| \right) (x)
= \kappa(\gamma,x) \diamond (\gamma -1),
\]
and we obtain
\[
D'_{\gamma}(x)x^{-1} =
(1\otimes x^{-1}) \diamond \kappa(\gamma,x) \diamond (\gamma-1)
=\widetilde{\kappa}(\gamma,x) \diamond \gamma (\gamma-1).
\]
The element $\kappa(\gamma,x)=\kappa(\gamma-1,x-1)$ is of degree $k-1$ in $A\otimes A$, so is $\widetilde{\kappa}(\gamma,x)$.
Moreover $\gamma-1\in I^k$.
Hence $\widetilde{\kappa}(\gamma,x)\diamond \gamma(\gamma-1) \equiv_{3k-1} \widetilde{\kappa}(\gamma,x)\diamond (\gamma-1)$, which proves the assertion.
\end{proof}

\begin{cor}
\label{cor:tgamx}
Let $\varepsilon \in \{ -1, +1\}$.
\begin{enumerate}
\item
For any $x\in \pi$, $(t_{\gamma})^\varepsilon (x)\, x^{-1} -1 \equiv_{2k+1}
\varepsilon\, \widetilde{\kappa}(\gamma,x) \diamond (\gamma - 1).$
\item
We have $(t_{\gamma})^\varepsilon \in \mathcal{W}[2k-2]$.
\item
The action of $(t_{\gamma})^\varepsilon$ 
on $\widehat{\pi}/\widehat{\pi}_{2k+1}$ depends only on the homotopy class of $\gamma$ modulo $\Gamma_{k+2} \pi$.
\end{enumerate}
\end{cor}

\begin{proof}
(1) immediately follows from Proposition \ref{prop:tgamx} since ${2k+1} \le 3k-1$.
Because $\widetilde{\kappa}(\gamma,x) \diamond (\gamma-1) = \widetilde{\kappa}(\gamma-1,x-1) \diamond (\gamma-1) \in I^{2k-1}$, 
we have $t_{\gamma} (x) = x$ modulo $\widehat{\pi}_{2k-1}$ which proves (2).

We now prove (3). Let $\gamma_1 \in \Gamma_k \pi$ such that 
$\delta:= \gamma^{-1} \gamma_1$ belongs to $\Gamma_{k+2}\pi$.
Since $\gamma_1 =\gamma \delta$, by Lemma \ref{lem:Phiproperties}\! (1) we compute
\begin{align*}
\widetilde{\kappa}(\gamma_1,x) \diamond (\gamma_1 -1 )
&= \big( \widetilde{\kappa}(\gamma,x) + \widetilde{\kappa}(\delta,x) \diamond (\bar{\gamma} \otimes \gamma) \big) \diamond (\gamma \delta-1) \\
&= \widetilde{\kappa}(\gamma,x) \diamond (\gamma \delta -1 ) + \widetilde{\kappa}(\delta,x) \diamond (\delta \gamma-1).
\end{align*}
The second term lies in $I^{2k+1}$ 
since $\widetilde{\kappa}(\delta,x) = \widetilde{\kappa}(\delta-1,x-1)$ is of degree $k+1$ and $(\delta \gamma -1) \in I^k$.
By a similar argument, the first term is equal to
\[
 \widetilde{\kappa}(\gamma,x) \diamond \big(  (\gamma-1)+\gamma (\delta-1)  \big)  \equiv_{2k+1}
\widetilde{\kappa}(\gamma,x) \diamond (\gamma-1).
\]
Hence
\[
\widetilde{\kappa}(\gamma_1,x) \diamond (\gamma_1-1)
\equiv_{2k+1}
\widetilde{\kappa}(\gamma,x) \diamond (\gamma-1).
\]
Now the assertion follows from (1).
\end{proof}

\subsection{A construction on group commutators}
\label{subsec:algorithm}

The next step to compute the action of $t_{\gamma}$ on $\widehat{\pi}/\widehat{\pi}_{2k+1}$ 
is to expand the right-hand side of the identity in Corollary \ref{cor:tgamx}\! (1)
by expressing $\gamma$ as a product of commutators of order $k$.
For this purpose we use the following construction.
The same construction will appear in Section \ref{sec:main_theorem} where we consider homology cylinders.
This fact is a key point to our proof of Theorem A.

Let $\mathcal{T}$ be a planar binary rooted tree.
Each edge $e$ of $\mathcal{T}$ is oriented so that it points away from the root: 
the initial and terminal vertices of $e$ are denoted by $i(e)$ and $t(e)$, respectively. 
The \emph{root edge} of $\mathcal{T}$ is the unique edge whose initial vertex is the root of $\mathcal{T}$.
An edge $e$ of $\mathcal{T}$ is called a \emph{leaf edge} if $t(e)$ is a leaf of $\mathcal{T}$.

Now suppose that $\mathcal{T}$ has $m$ leaves, and the leaves 
are colored by elements $g_1,\ldots,g_m$ of a group $G$.
To simplify notation, we denote $\vec{g}=(g_1,\ldots,g_m)$.
By reading each trivalent vertex as a group commutator, 
we can define an element $\mathcal{T}(\vec{g}) = \mathcal{T}(g_1,\ldots,g_m) \in \Gamma_m G$.
For instance, if
\begin{equation}
\label{eq:Texample}
\mathcal{T}=
\begin{array}{c}
\labellist \small \hair 2pt
\pinlabel {root} [t] at 108 5
\endlabellist
\includegraphics[scale=0.3]{rooted_tree} 
\end{array}
\end{equation}
then $\mathcal{T}(g_1,g_2,g_3,g_4) = [g_1,[[g_2,g_3],g_4]]$.

Let $e$ be an edge of $\mathcal{T}$. 
If $e$ is a leaf edge, let $\xi(e)\in G$ be the color of the leaf $t(e)$. Otherwise cut $\mathcal{T}$ at $i(e)$: then the connected component containing $e$ can be considered as a planar binary tree rooted at $i(e)$ and its leaves inherit a $G$-coloring from that of $\mathcal{T}$;
therefore it defines a group commutator $\xi(e) \in \Gamma_2 G$ by using the procedure in the preceding paragraph.
In this way any edge $e$ of $\mathcal{T}$ carries an element $\xi(e)$ of $G$.

\begin{dfn}
\label{dfn:xi_j}
Let $\mathcal{T}$ be a planar binary rooted tree with $m$ leaves colored by $\vec{g} \in G^m$ and let $h\in G$.
For each edge $e$ of $\mathcal{T}$, we define the element
\[
\Phi_e(\mathcal{T},\vec{g},h) \in G
\]
by the following inductive procedure:
\begin{itemize}
\item If $e$ is the root edge of $\mathcal{T}$, then define $\Phi_e(\mathcal{T},\vec{g},h):=h$.
\item Let $e$ be a non-root edge of $\mathcal{T}$.
Set $v:=i(e)$ and let $e_0$ be the edge of $\mathcal{T}$ with $t(e_0)=v$.
There are two cases: when one goes through the edge $e_0$ from $i(e_0)$, 
the edge $e$ appears either on the left (Case~1) or on the right (Case~2).
In Case 1, let $e_l:=e$ 
and let $e_r$ be the edge such that $e_r \neq e_l$ and $i(e_r) = v$.
In Case 2, let $e_r:= e$ 
and let $e_l$ be the edge such that $e_l \neq e_r$ and $i(e_l) = v$.
In both cases, let $t_l := \xi(e_l)$ and $t_r := \xi(e_r)$.
Now define
\[
\Phi_e(\mathcal{T},\vec{g},h) :=
\begin{cases}
\big[\Phi_{e_0}(\mathcal{T},\vec{g},h),{\overline{t_r}}^{t_l}\big] & \text{in Case 1,} \\
\big[\Phi_{e_0}(\mathcal{T},\vec{g},h)^{\overline{t_l}},{t_l}^{t_r}\big] & \text{in Case 2.}
\end{cases}
\]
\end{itemize}
For $1\le j\le m$, we denote $\Phi_j(\mathcal{T},\vec{g},h) = \Phi_{e_j}(\mathcal{T},\vec{g},h)$, where $e_j$ is the leaf edge of $\mathcal{T}$ such that $t(e_j)$ is the $j$th leaf.
\end{dfn}

Note that if $h\in \Gamma_a G$ for some integer $a$, 
then $\Phi_j (\mathcal{T},\vec{g},h) \in \Gamma_{a+m-1}G$
for any $j \in \{1,\dots,m\}$.

\begin{example}
If $\mathcal{T}$ is as in \eqref{eq:Texample}, then
\[
\Phi_1(\mathcal{T},\vec{g},h) =
\left[ h, \overline{[[g_2,g_3],g_4]}^{g_1} \right],
\]
and
\[
\Phi_2(\mathcal{T},\vec{g},h) =
\left[
\left[
[h^{\overline{g_1}},{g_1}^{[[g_2,g_3],g_4]}],
\overline{g_4}^{[g_2,g_3]}
\right],
\overline{g_3}^{g_2}
\right].
\]
\end{example}

Let $e$ be an edge of $\mathcal{T}$.
If $e$ is a leaf edge, let $\mathcal{T}_e := \mathcal{T}$.
Otherwise, cut~$\mathcal{T}$ at $t(e)$.
Then $\mathcal{T}$ decomposes into three parts:
let $\mathcal{T}_e$ be the connected component containing $e$ (and the root of $\mathcal{T}$),
and call $\mathcal{T}_e$ the \emph{pruning} of $\mathcal{T}$ at~$e$.
The value $\Phi_e(\mathcal{T},\vec{g},h)$ can be computed from the subtree $\mathcal{T}_e$ 
and its $G$-coloring induced by $\vec{g}$,
but it depends on the shape of $\mathcal{T} \setminus \mathcal{T}_e$  only through the commutator $\xi(e)$.
More explicitly, the following holds by construction.

\begin{lem}
\label{lem:TT_e}
Let $e$ be an edge of $\mathcal{T}$ and let $\vec{g}_e$ be the $G$-coloring of $\mathcal{T}_e$ induced by $\vec{g}$:
that is, the leaf of $\mathcal{T}_e$ that was the terminal vertex of $e$ is colored with~$\xi(e)$
and  all the other leaves of $\mathcal{T}_e$ are colored according to $\vec{g}$.
Then we have $\ \Phi_e(\mathcal{T},\vec{g},h) = \Phi_e(\mathcal{T}_e,\vec{g}_e,h)$.
\end{lem}

\subsection{A more explicit formula}

\label{subsec:mef}

We prove the main result of this section.
Recall that $k\geq 2$ is an integer, $\gamma$ is a closed curve in $\Sigma$ of nilpotency class~$\geq k$,
and, after an arbitrary choice of orientation and arc-basing  at   $\bullet$, 
the curve $\gamma$ is regarded as an element  of $\pi_1(\Sigma,\bullet)\cong \pi$.
Since $\gamma \in \Gamma_k \pi$, we can write
\begin{equation}
\label{eq:gamcomm}
\gamma = \prod_{i=1}^l \mathcal{T}_i(\gamma_{i1},\ldots,\gamma_{ik}),
\end{equation}
where $\mathcal{T}_i$ is a planar binary rooted tree with $k$ leaves and $\gamma_{i1},\ldots,\gamma_{ik} \in \pi$.
We represent each $\gamma_{ij}$ as an immersed loop based at $\bullet$.
Let $\vec{\gamma}_i := (\gamma_{i1},\ldots,\gamma_{ik})$.
With the notation of the previous subsection, let
\begin{equation} \label{eq:lambda_ij}
\lambda_{ij} := \Phi_j(\mathcal{T}_i,\vec{\gamma}_i,\gamma)
\in \Gamma_{2k-1} \pi.
\end{equation}

\begin{thm}
\label{thm:gdtaction}
Let $k\ge 2$, let $\varepsilon \in \{ -1, +1 \}$ and 
let $\gamma \subset \Sigma$ be a closed curve of nilpotency class $\ge k$ expressed as in \eqref{eq:gamcomm}.
Then, for any  $x\in \pi$ which we represent by
an immersed loop  based at $\star$ and intersecting
each $\gamma_{ij}$ in finitely many transverse double points,  we have
\begin{equation} \label{eq:t(x)/x}
(t_{\gamma})^{\varepsilon}(x)\, x^{-1} 
\equiv \Big( \prod_{i=1}^l \prod_{j=1}^k \prod_{p\in \gamma_{ij} \cap x}
\big( {\lambda_{ij}}^{\varepsilon(p;\gamma_{ij},x)} \big)^{x_{\star p}\overline{(\gamma_{ij})_{\bullet p}} \nu} \Big)^{\varepsilon} 
\mod \hpi_{2k+1}. 
\end{equation}
\end{thm}

\noindent
Note that since $\lambda_{ij} \in \Gamma_{2k-1}\pi$, the order of product in \eqref{eq:t(x)/x} can be arbitrary.

\begin{cor}
Let $k\ge 2$ and let $\gamma \subset \Sigma$ be a closed curve of nilpotency class $\ge k$.
Then $t_\gamma \in \Aut(\widehat{\pi}/\widehat{\pi}_{2k+1})$ restricts to an automorphism of~$\pi/\Gamma_{2k+1}\pi$.
\end{cor}

\begin{proof}
Since $\pi$ is a free group, the canonical map $\pi\to  \hpi$ induces an injection 
$\pi/\Gamma_{2k+1}\pi \to \widehat{\pi}/\widehat{\pi}_{2k+1}$, 
so that we can regard $\pi/\Gamma_{2k+1}\pi$ as a subgroup of $\widehat{\pi}/\widehat{\pi}_{2k+1}$.
Then  \eqref{eq:t(x)/x} implies that $ (t_\gamma)^{\varepsilon} \in \Aut(\widehat{\pi}/\widehat{\pi}_{2k+1})$ 
leaves $\pi/\Gamma_{2k+1}\pi$ globally invariant  for each $\varepsilon \in \{ -1,+1\}$. 
The conclusion follows.
\end{proof}

The rest of this subsection is devoted to the proof of Theorem \ref{thm:gdtaction}.

\begin{lem}
\label{lem:ux}
Let $x\in \pi$ and $u\in \Gamma_k\pi$. Then
\[
(1\otimes 1-x\otimes \bar{x})\diamond (u-1)
\equiv_{2k+1} [u,x]-1.
\]
\end{lem}

\begin{proof}
We have
$(1\otimes 1-x\otimes \bar{x})\diamond (u-1)
= u-xu\bar{x} = (1-[x,u])u$.
Since $u\in \Gamma_k\pi$ and $[x,u]\in \Gamma_{k+1}\pi$, we have
$(1-[x,u])u \equiv_{2k+1} (1-[x,u])$.
Furthermore, $1-[x,u] \equiv_{2k+2} [u,x]-1$ since
\[
0 \equiv_{2k+2} ([x,u]-1)([u,x]-1) = (1-[x,u]) - ([u,x]-1).
\]
This proves the lemma.
\end{proof}

\begin{lem}
\label{lem:petticom}
Let $x,y\in \Gamma_{k+1}\pi$ for some $k\ge 1$.
Then, we have
\begin{enumerate}
\item $(x-1) + (y-1) \equiv_{2k+2} (xy-1)$,
\item $(x^{-1}-1) + (x-1) \equiv_{2k+2} 0$,
\item $xy-1 \equiv_{2k+2} yx-1$.
\end{enumerate}
\end{lem}

\begin{proof}
Since $x-1, y-1 \in I^{k+1}$, we have
\[
0 \equiv_{2k+2} (x-1)(y-1) = (xy-1) - \left( (x-1) + (y-1) \right).
\]
This proves (1).
(2) follows from (1) by putting $y=x^{-1}$, and (3) is a direct consequence of (1).
\end{proof}

\begin{prop}
\label{lem:Phigamx}
Let $1\le q \le k$, let $\mathcal{T}$ be a planar binary rooted tree with $q$ leaves colored by $\vec{\gamma}=(\gamma_1,\ldots,\gamma_q) \in \pi^q$, and let $\delta \in \Gamma_k\pi$.
Set $\gamma=\mathcal{T}(\vec{\gamma})$.
Then, for any $x\in \pi$, we have
\[
\widetilde{\kappa}(\gamma,x) \diamond (\delta-1)
\equiv_{2k+1}
\sum_{j=1}^q \widetilde{\kappa}(\gamma_j,x) \diamond
\big( \Phi_j(\mathcal{T},\vec{\gamma},\delta)-1 \big).
\]
\end{prop}

\begin{proof}
We use induction on $q$. The case $q=1$ is clear. 

Let $q\ge 2$.
There is an index $i$ such that the $i$th 
and $(i+1)$st leaves of~$\mathcal{T}$ share their initial vertex, which we denote by $v$.
Let $e_0$ be the edge of $\mathcal{T}$ such that $t(e_0)=v$.
We denote by $\vec{\gamma}_{e_0}=(\gamma_1^0,\ldots,\gamma_{q-1}^0)$
the $\pi$-coloring of the pruning $\mathcal{T}_{e_0}$ induced by $\vec \gamma$.
Since $\gamma = \mathcal{T}_{e_0}(\vec{\gamma}_{e_0})$, by the inductive assumption we compute
\[
\widetilde{\kappa}(\gamma,x) \diamond (\delta -1) \equiv_{2k+1}
\sum_{j=1}^{q-1} \widetilde{\kappa}(\gamma_j^0,x) \diamond \big(\Phi_j(\mathcal{T}_{e_0},\vec{\gamma}_{e_0},\delta)-1\big).
\] 
If $j\le i-1$, $\gamma_j^0 = \gamma_j$.
If $j\ge i+1$, $\gamma_j^0 = \gamma_{j+1}$.
By the same reason as Lemma~\ref{lem:TT_e}, we have
\[
\Phi_j(\mathcal{T}_{e_0},\vec{\gamma}_{e_0},\delta) =
\begin{cases}
\Phi_j(\mathcal{T},\vec{\gamma},\delta) & \text{if $j\le i-1$}, \\
\Phi_{j+1}(\mathcal{T},\vec{\gamma},\delta) & \text{if $j\ge i+1$}.
\end{cases}
\] 
Since $\gamma_i^0 = [\gamma_i,\gamma_{i+1}]$, by Lemma \ref{lem:Phiproperties}\! (4) and Lemma \ref{lem:ux} we compute
\begin{align*}
& \widetilde{\kappa}(\gamma_i^0,x) \diamond (\Phi_i(\mathcal{T}_{e_0},\vec{\gamma}_{e_0},\delta)-1) \\
&= \widetilde{\kappa}(\gamma_i,x) \diamond (1\otimes 1- \overline{\gamma_{i+1}}^{\gamma_i} \otimes {\gamma_{i+1}}^{\gamma_i}) \diamond (\Phi_i(\mathcal{T}_{e_0},\vec{\gamma}_{e_0},\delta)-1) \\
& \quad +\widetilde{\kappa}(\gamma_{i+1},x) \diamond (1\otimes 1-{\gamma_i}^{\gamma_{i+1}} \otimes \overline{\gamma_i}^{\gamma_{i+1}}) \diamond (\overline{\gamma_i} \otimes \gamma_i) \diamond (\Phi_i(\mathcal{T}_{e_0},\vec{\gamma}_{e_0},\delta)-1) \\
&\equiv_{2k+1} \widetilde{\kappa}(\gamma_i,x) \diamond
\left( \big[ \Phi_i(\mathcal{T}_{e_0},\vec{\gamma}_{e_0},\delta), \overline{\gamma_{i+1}}^{\gamma_i}\big] -1 \right) \\
& \hspace{3em} +\widetilde{\kappa}(\gamma_{i+1},x) \diamond
\left( \big[ \Phi_i(\mathcal{T}_{e_0},\vec{\gamma}_{e_0},\delta)^{\overline{\gamma_i}},{\gamma_i}^{\gamma_{i+1}}\big] -1 \right) \\
&= \widetilde{\kappa}(\gamma_i,x) \diamond \Phi_i(\mathcal{T},\vec{\gamma},\delta) + \widetilde{\kappa}(\gamma_{i+1},x) \diamond \Phi_{i+1}(\mathcal{T},\vec{\gamma},\delta).
\end{align*}
In the last line we have used Lemma~\ref{lem:TT_e} and 
the defining formula for $\Phi_e(\mathcal{T},\vec{g},h)$ in Definition~\ref{dfn:xi_j}.
This completes the proof of the  proposition.
\end{proof}

\begin{proof}[Proof of Theorem \ref{thm:gdtaction}]
For each $i\in\{1,\dots,l\}$,
let $\gamma_i := \mathcal{T}_i(\gamma_{i1},\ldots,\gamma_{ik}) \in \Gamma_k \pi$, 
so that $\gamma = \gamma_1\cdots \gamma_l$. 
By Corollary \ref{cor:tgamx} and Lemma \ref{lem:Phiproperties}\! (2), we have
\begin{align*}
(t_{\gamma})^{\varepsilon}(x)\, x^{-1} -1 & \equiv_{2k+1}  \varepsilon\,  \widetilde{\kappa}(\gamma,x) \diamond (\gamma - 1) \\
& =  \varepsilon  \sum_{i=1}^l \widetilde{\kappa}(\gamma_i,x) \diamond (\overline{\gamma_1\cdots \gamma_{i-1}} \otimes \gamma_1\cdots \gamma_{i-1}) \diamond (\gamma-1) \\
&\equiv_{2k+1}  \varepsilon  \sum_{i=1}^l \widetilde{\kappa}(\gamma_i,x) \diamond (\gamma-1).
\end{align*} 
In the last line we have used the fact that $\gamma^{\overline{\gamma_1\ldots\gamma_{i-1}}} \equiv_{2k} \gamma$ and that $\widetilde{\kappa}(\gamma_i,x)=\widetilde{\kappa}(\gamma_i-1,x-1)$ is of degree $k-1$. 

Applying Proposition \ref{lem:Phigamx} with $q=k$, we obtain
\[
\widetilde{\kappa}(\gamma_i,x) \diamond (\gamma-1)
\equiv_{2k+1} \sum_{j=1}^k \widetilde{\kappa}(\gamma_{ij},x) \diamond (\lambda_{ij}-1).
\]
By using \eqref{eq:Phiexplicit} we can express $\widetilde{\kappa}(\gamma_{ij},x)$ as a sum over the intersections $\gamma_{ij} \cap x$.
Therefore,
\begin{align*}
(t_{\gamma})^{\varepsilon}(x)\, x^{-1}  -1 & \equiv_{2k+1}
\varepsilon  \sum_{i=1}^l \sum_{j=1}^k \widetilde{\kappa}(\gamma_{ij},x) \diamond (\lambda_{ij}-1) \\
&=  \varepsilon  \sum_{i=1}^l \sum_{j=1}^k \sum_{p\in \gamma_{ij}\cap x}
\varepsilon(p;\gamma_{ij},x)
\left( x_{\star p}\, \overline{(\gamma_{ij})_{\bullet p}}\, \nu\, \lambda_{ij}\, \bar{\nu}\, 
(\gamma_{ij})_{\bullet p}\, \overline{x_{\star p}} -1 \right) \\
&\equiv_{2k+1}
\Big(  \prod_{i=1}^l \prod_{j=1}^k \prod_{p\in \gamma_{ij} \cap x}
\big( {\lambda_{ij}}^{\varepsilon(p;\gamma_{ij},x)} \big)^{x_{\star p}\overline{(\gamma_{ij})_{\bullet p}} \nu} \Big)^{\varepsilon}  -1.
\end{align*}
In the last line we have used Lemma \ref{lem:petticom}.
This completes the proof of Theorem \ref{thm:gdtaction}.
\end{proof}

\section{Homology cylinders}   \label{sec:hc}

In this section, we are interested in the monoid homomorphisms
$$
\rho_{2k}\colon  \calC[k] \longrightarrow \Aut(\pi/\Gamma_{2k+1}\pi)
$$
defined on the  $k$th  term of the Johnson filtration of the monoid $\calC$  for all~${k\geq 1}$.
We first review an equivalent description of those homomorphisms.

\subsection{Morita homomorphisms} \label{subsec:Morita}

Let $k\geq 1$ be an integer.
The next lemma,  which is well-known, 
shows that the image of $\rho_{2k}$ is a commutative submonoid of $\Aut(\pi/\Gamma_{2k+1}\pi)$.

\begin{lem} \label{lem:motivation}
If  $\phi, \psi\in \Aut(\pi/\Gamma_{2k+1}\pi)$ induce the identity on $\pi/\Gamma_{k+1}\pi$, then they commute.
\end{lem}
 
\noindent
In fact, it can be shown that $\ell:=2k$ is the greatest integer $\ell\geq k$ such that $\phi \circ \psi = \psi \circ \phi$
for any $\phi, \psi\in \Aut(\pi/\Gamma_{\ell+1}\pi)$ inducing the identity on $\pi/\Gamma_{k+1}\pi$.
 
\begin{proof}[Proof of Lemma \ref{lem:motivation}]
Let $x\in \pi$.
We have
$$
\phi(\{x\}_{2k}) = \{x\}_{2k}\, \{y\}_{2k} \quad \hbox{and} \quad \psi(\{x\}_{2k}) = \{x\}_{2k}\, \{z\}_{2k}
$$
where $y,z\in \Gamma_{k+1}\pi$. Hence
\begin{eqnarray*}
(\psi\circ \phi)(\{x\}_{2k}) &=&  \psi(\{x\}_{2k})\, \psi( \{y\}_{2k} ) \\
&=&  \{x\}_{2k}\, \{z\}_{2k} \, \psi( \{y\}_{2k} ) 
\ = \  \{x\}_{2k}\, \{z\}_{2k} \,  \{y\}_{2k}
\end{eqnarray*}
since, by assumption on $\psi$, we have $\psi( \{y\}_{2k} ) = \{y\}_{2k}$. A similar computation gives $(\phi\circ \psi)(\{x\}_{2k}) =  \{x\}_{2k}\, \{y\}_{2k} \,  \{z\}_{2k}$
and, since $\Gamma_{k+1}\pi/\Gamma_{2k+1}\pi$ is abelian, we conclude that $(\phi\circ \psi)(\{x\}_{2k}) = (\psi\circ \phi)(\{x\}_{2k}) $.
\end{proof} 

As we shall now recall, one can swap $\rho_{2k}\colon  \calC[k] \to \Aut(\pi/\Gamma_{2k+1}\pi)$ 
for another homomorphism whose target is explicitly an abelian group.
The latter  is the \emph{ $k$th  Morita homomorphism}, denoted by
$$
M_k\colon \calC[k] \longrightarrow H_3(\pi / \Gamma_{k+1}\pi).
$$ 
It was originally defined for the mapping class group of $\Sigma$ \cite{Mor93}
and, next, it has been  extended to the setting of homology  cobordisms \cite{Sakasai}. 

One direct way to define $M_k$ is to use the notion of ``$k$th nilpotent homotopy type'' 
for a closed oriented connected $3$-manifold $N$:
following Turaev~\cite{Tu84}, we define the latter to be the image
$$
[N]_k \in H_3\big(\pi_1(N)/\Gamma_{k+1}\pi_1(N)\big)
$$
of the fundamental class $[N]\in H_3(N)$ by the canonical homotopy class of maps $N\to K\big(\pi_1(N)/\Gamma_{k+1}\pi_1(N),1\big)$.
Then, for  any  $C \in \calC[k]$, one defines
$$
M_k(C) := [\widehat C\, ]_k \  \in  H_3(\pi / \Gamma_{k+1}\pi).
$$
 Here (and thereafter) the boundary parametrization $c$ of a homology cobordism $C$
is implicit in our notation, and  $\widehat C$ is the closed oriented $3$-manifold 
 obtained from $C$ by collapsing its ``vertical'' boundary $c(\partial \Sigma \times [-1,+1])$
to the circle $b :=c(\partial \Sigma \times \{0\})$ and  by gluing $\partial_+ C$ to $\partial_- C$
via the boundary parametrizations $c_\pm$ 
(thus, $\widehat{C}$ comes with an open-book decomposition with binding $b$ 
and  leaf $\partial_\pm C\cong \Sigma$).
Note that the group 
   $\pi_1(\widehat C)/\Gamma_{k+1} \pi_1(\widehat C)$ is identified with $\pi / \Gamma_{k+1}\pi$ via the composition
\begin{equation} \label{eq:nota}
\widehat{c}_\pm:= \big(\, \Sigma \stackrel{c_\pm}{\longrightarrow} C \hookrightarrow\widehat C\, \big).
\end{equation}
The equivalence between the above definition of $M_k$ (using fundamental classes of $3$-manifolds) 
and Morita's original definition (using the bar complex to compute group homology)  is observed in \cite{Heap}. See also \cite{Sakasai}.

The next statement says that $\rho_{2k}$ is tantamount to $M_k$.

\begin{thm}[Heap, Sakasai] \label{thm:HS}
For any $C,C'\in \calC[k]$, we have $\rho_{2k}(C)=\rho_{2k}(C')$ if and only if $M_k(C)=M_k(C')$.
\end{thm}

\begin{proof}[About the proof]
Two homology cobordisms $C$ and $C'$ of $\Sigma$ are said to be \emph{homologically cobordant} 
if there exists a compact oriented $4$-manifold $W$ such that 
$$
\partial W= C' \cup_{c'\circ c^{-1}} (-C)
$$
and the inclusion $C \hookrightarrow W$ (resp. $C' \hookrightarrow W$) induces an isomorphism in homology.
The quotient  of the monoid $\calC$ by this equivalence relation is a group~\cite{GL}.
It turns out that each of the monoid homomorphisms $\rho_{2k}$ and $M_k$ preserves  this equivalence relation
and, so, each of them factorizes to  a group homomorphism. Thus, the theorem reduces to the fact that 
$$
\ker M_k = \calC[2k]
$$
which has been proved in \cite[Theorem 7.1]{Sakasai}.
(See \cite[Corollary 6]{Heap} in the case of the mapping class group.)
\end{proof}

\subsection{Surgeries along knots} \label{subsec:surgeries}

Let $N$ be a compact oriented  $3$-manifold,
with a null-homologous oriented  knot $K$ in its interior.
The tubular neighborhood of $K$ is denoted by $\operatorname{T}(K)$. 
A \emph{parallel} of $K$ is an oriented  simple closed curve on $\partial \operatorname{T}(K)$
which is homotopic to $K$ in $\operatorname{T}(K)$.
Let
$$
\lambda_0(K) \subset \partial \operatorname{T}(K)
$$ 
be the \emph{preferred} parallel of $K$, 
i.e$.$ the unique parallel that is null-homologous in $N \setminus \operatorname{int\, T}(K)$.
Let also
$$
\mu(K )\subset \partial \operatorname{T}(K)
$$ 
be the \emph{oriented meridian} of $K$, i.e$.$ the unique oriented simple closed curve 
 bounding an oriented  disk in  $\operatorname{T}(K)$
whose intersection number with $K$ is~$+1$.
The knot $K$ is said to be (integrally) \emph{framed} if a parallel is specified, 
and it is \emph{$p$-framed} if this parallel is $p\, \mu(K )+ \lambda_0(K)$   in homology, for some $p\in \Z$.
Recall that $K$ is \emph{of nilpotency class $\geq r$} if $[K]$ belongs to $\Gamma_r \pi_1(N)$.

\begin{prop}[Cochran--Gerges--Orr]  \label{prop:CGO}
Let $N$ be a compact and oriented $3$-manifold
with a $(\pm 1)$-framed knot $K$ in its interior, and assume that $K$ is of nilpotency class $\geq r$ with  $r\geq 2$.
We denote by $N':=N_K$ the $3$-manifold that is obtained from $N$ by surgery along~$K$. 
Then, we have the following:
\begin{itemize}
 \item[(i)] The homotopy classes of $\mu(K)$ and $\lambda_0(K)$ belong to $\Gamma_r \pi_1(N')$.
 \item[(ii)] There is a unique isomorphism $\phi$ 
 such that the following diagram is commutative:
\begin{equation}  \label{eq:triangle}
\xymatrix @!0 @R=1cm @C=3cm {
  & \frac{\pi_1(N )}{\Gamma_r \pi_1(N)}  \ar@{-->}[dd]^-\phi \\
\frac{\pi_1\big(N \setminus \operatorname{int\, T}(K)  \big)}{\Gamma_r \pi_1\big(N \setminus \operatorname{int\, T}(K)  \big)} \ar@{->>}[ru]  \ar@{->>}[rd] & \\
  & \frac{\pi_1(N')}{\Gamma_r \pi_1(N')} 
 }   
\end{equation} 
 \item[(iii)] If $N$ is furthermore closed and connected, then the nilpotent homotopy types $[N]_{r-1}$ and  $[N']_{r-1}$ 
 correspond each other through the  isomorphism
 $$
 H_3\Big( \frac{\pi_1(N )}{\Gamma_r \pi_1(N)} \Big) \stackrel{\phi_*}{\longrightarrow}  H_3\Big( \frac{\pi_1(N')}{\Gamma_r \pi_1(N')} \Big).
 $$
\end{itemize}
\end{prop}

\begin{proof}
The three assertions can be found in \cite{CGO} but, for the sake of completeness, 
let us prove each of them. We denote by  $E:= N \setminus \operatorname{int\, T}(K)$ the exterior of the knot $K$,
and we recall that $N'$ is obtained from $E$ by gluing a solid torus $S^1\times D^2$: 
specifically, $S^1 \times \partial D^2$ is identified to the simple closed curve of $\partial \operatorname{T}(K)=\partial E$ 
that is homologous to $\pm \mu(K) + \lambda_0(K)$, where $\pm 1$ is the framing number.

(i) This is proved in \cite[Proposition 2.1]{CGO} using gropes, 
but we prefer to reprove this here in a rather different way.
Observe  that $\mu(K )^{\pm 1}\cdot \lambda_0(K) \in \pi_1(\partial \operatorname{T}(K))$ is trivial in $\pi_1(N')$:
hence $\mu(K) \in \Gamma_r \pi_1(N')$ if and only if $ \lambda_0(K) \in \Gamma_r \pi_1(N')$.
The proof is then by induction on $r\geq 2$. There is nothing to check for $r=2$ 
since $\lambda_0(K)$ is already null-homologous in $E$. Assume that (i) holds true at rank $r-1$.
Since $\lambda_0(K)$ is homotopic to $K$ in $N$, we have $\lambda_0(K) \in \Gamma_r \pi_1(N)$.
Since the kernel of the  homomorphism $j\colon \pi_1(E) \to \pi_1(N)$ induced by the inclusion 
is normally generated by $\mu(K)$,  we can write
\begin{equation} \label{eq:lambda_0}
\pi_1(E) \ni \lambda_0(K) = c \cdot \prod_{i=1}^s \big(x_i\, \mu(K)^{\epsilon_i}\,  x_i^{-1}\big)
\end{equation}
where $c \in \Gamma_r \pi_1(E)$, $\epsilon_1,\dots, \epsilon_s$ are signs
 and $x_1,\dots,x_s \in \pi_1(E)$. 
 Since $K$ is of nilpotency class $\geq r$, it is certainly of nilpotency class $\geq (r-1)$:
hence, by the induction hypothesis, we have $\mu(K)\in \Gamma_{r-1} \pi_1(N')$.
 Therefore, denoting by $j'\colon \pi_1(E) \to \pi_1(N')$
 the homomorphism induced by the inclusion, 
 we obtain
\begin{eqnarray*}
\pi_1(N') \ni \lambda_0(K) &=& j'(c) \cdot \prod_{i=1}^s \big[j'(x_i) , \mu(K)^{\epsilon_i}\big] \mu(K)^{\epsilon_i} \\
&\equiv  & \prod_{i=1}^s \mu(K)^{\epsilon_i}  \mod \Gamma_r \pi_1(N').
\end{eqnarray*}
Besides, since $\lambda_0(K)$ is null-homologous in $E$, formula \eqref{eq:lambda_0} implies that $\sum_{i=1}^s \epsilon_i=0$.
We conclude that $\lambda_0(K) \in \Gamma_r \pi_1(N')$.

For future use, observe  that $\mu(K) \in \Gamma_r \pi_1(N')$ implies the following: 
$N$~is obtained from $N'$ by $(\mp 1)$-surgery along a knot $K'$ of nilpotency class $\geq r$.

(ii) 
The unicity of a group isomorphism $\phi$ such that \eqref{eq:triangle}
commutes follows from the fact that the group homomorphism $j\colon \pi_1(E) \to \pi_1(N)$ is surjective.
The inclusion $E \subset N$ induces an isomorphism
$$
\frac{\pi_1(E)}{\langle\!\langle \Gamma_r \pi_1(E), \mu(K) \rangle\!\rangle} \stackrel{\psi}{\longrightarrow} \frac{\pi_1(N)}{ \Gamma_r \pi_1(N)} 
$$
and, according to (i),  the inclusion $E \subset N'$ induces a group homomorphism
$$
\frac{\pi_1(E)}{\langle\!\langle \Gamma_r \pi_1(E), \mu(K) \rangle\!\rangle} \stackrel{\psi'}{\longrightarrow} \frac{\pi_1(N')}{ \Gamma_r \pi_1(N')}. 
$$
Then $\phi := \psi' \psi^{-1}$ fits into a commutative triangle \eqref{eq:triangle}. 
Thanks to the observation that concluded the proof of (i),
we can exchange the roles of $N$ and $N'$: hence $\phi$ is an isomorphism.

(iii) This follows from the implication (A)$\Rightarrow$(B) in \cite[Theorem~6.1]{CGO}. 
We rephrase their arguments below.
Let $W$ be the compact oriented $4$-manifold that is obtained from $N\times [0,1]$ 
by attaching a $2$-handle along the solid torus $\operatorname{T}(K)\subset N\times \{1\}$ according to the given framing of $K$:
hence $W$ is a cobordism between $N$ and $N'$. 
Using the van Kampen theorem and the fact that $K\in \Gamma_r \pi_1(N)$, 
we see that the inclusion of $N$ into $W$ induces an isomorphism
$$
\beta\colon \pi_1(N)/\Gamma_r \pi_1(N) \longrightarrow \pi_1(W)/\Gamma_r \pi_1(W)
$$
and, by the symmetry observed at the end of the proof of (i), there is also an isomorphism
$$
\beta'\colon \pi_1(N')/\Gamma_r \pi_1(N') \longrightarrow \pi_1(W)/\Gamma_r \pi_1(W)
$$
induced by the inclusion of $N'$ into $W$. By an appropriate commutative diagram of inclusions, 
we see that $\phi = (\beta')^{-1} \beta$. Therefore, we are reduced to show that 
\begin{equation} \label{eq:k}
\beta_*([N]_{r-1}) = \beta'_*([N']_{r-1}) \in H_3\Big(\frac{\pi_1(W)}{\Gamma_r \pi_1(W)}\Big).
\end{equation}
The canonical homotopy class of maps $ W \to K(\pi_1(W)/\Gamma_r \pi_1(W),1)$ precomposed with the inclusions
$k\colon N\to W$ and $k'\colon N'\to W$ yields  homotopy classes of maps
$$
\zeta\colon N \to K(\pi_1(W)/\Gamma_r \pi_1(W),1) \quad \hbox{and} \quad \zeta'\colon N' \to K(\pi_1(W)/\Gamma_r \pi_1(W),1).
$$
Clearly $\zeta_*([N])= \beta_*([N]_{r-1})$ and $\zeta'_*([N'])= \beta'_*([N']_{r-1})$.
Thus \eqref{eq:k} follows from the obvious identity $k_*([N]) = k'_*([N']) \in H_3(W)$.
\end{proof}

\subsection{Behaviour under surgeries}

We now start to study the behaviour of the homomorphisms $\rho_{2j}\colon \calC[j] \to \Aut(\pi/\Gamma_{2j+1} \pi)$
under surgeries along null-homologous knots.
For that, we will use the equivalence between $\rho_{2j}$ and $M_j$ recalled in Section \ref{subsec:Morita}.

\begin{lem} \label{lem:surgery}
Let $k\geq 2$ be an integer.
Let $C\in \calC$ and let $C':=C_L$ be the homology cobordism
obtained by surgery along a $(\pm 1)$-framed knot 
 $L\subset C$  of nilpotency class~$\geq k$:
\begin{itemize}
\item[(1)] for all $j\in \{1,\dots,2k-2\}$, $C\in \calC[j]$ if and only if $C' \in \calC[j]$;
\item[(2)] if $C\in\calC[k-1]$, then $\rho_{2k-2}(C)= \rho_{2k-2}(C')$.
\end{itemize}
\end{lem}

\begin{proof}
According to Proposition \ref{prop:CGO}\! (ii), there exists  a group isomorphism 
$$\phi\colon \pi_1(C)/\Gamma_{k} \pi_1( C) \longrightarrow \pi_1( C')/\Gamma_{k} \pi_1(C')$$ 
such that 
$$
\phi \circ {c}_\pm = {c'}_\pm \colon 
\pi / \Gamma_{k} \pi \longrightarrow \pi_1(C')/\Gamma_{k} \pi_1(C').
$$
Therefore
$$
\rho_{k-1}(C') = (c'_-)^{-1} \circ c'_+= (\phi c_-)^{-1} \circ (\phi c_+) = (c_-)^{-1} \circ c_+ = \rho_{k-1}(C).
$$
We deduce that,  for all $j \in \{1,\dots,k-1\}$, $C\in \calC[j]$  if and only if  $C'\in \calC[j]$.

This proves ``half'' of assertion (1). We now  prove assertion (2) and, for that, we assume that $C\in\calC[k-1]$ and use the notation  \eqref{eq:nota}.
According to the statements (ii) and (iii) of Proposition \ref{prop:CGO}, there is an isomorphism
$$\widehat \phi\colon \pi_1(\widehat C)/\Gamma_{k} \pi_1(\widehat C) \longrightarrow \pi_1( \widehat C')/\Gamma_{k} \pi_1(\widehat C')$$ 
such that
\begin{equation} \label{eq:c_c'}
\widehat\phi \circ {\widehat c}_\pm = {\widehat c'}_\pm\colon 
\pi / \Gamma_{k} \pi \longrightarrow \pi_1(\widehat C')/\Gamma_{k} \pi_1(\widehat C')
\end{equation}
and
\begin{equation} \label{eq:phi}
\widehat \phi\big([\widehat C\, ]_{k-1}\big)  = [\widehat C'\, ]_{k-1} \in H_3\big( \pi_1(\widehat C')/\Gamma_{k} \pi_1(\widehat C') \big).
\end{equation}
By combining \eqref{eq:c_c'} with \eqref{eq:phi}, we obtain that $M_{k-1}(C)= M_{k-1}(C')$
or, equivalently by Theorem \ref{thm:HS}, $\rho_{2k-2}(C) =\rho_{2k-2}(C')$. 

Finally, we prove the second ``half'' of  (1).
Let $j\in\{k,\dots, 2k-2\}$ and assume that $C\in\calC[j]$.
Then $C \in \calC[k-1]$ and $\rho_{2k-2}(C) =\rho_{2k-2}(C')$ by the previous paragraph.
In particular, the action of $C'$ on $\pi/\Gamma_{j+1}\pi$ is the same as that of $C$, i.e$.$ is trivial: therefore $C'\in\calC[j]$.
The converse implication $C'\in\calC[j] \Rightarrow C\in\calC[j]$ is also true since surgery along a knot of nilpotency class $\geq k$
is invertible (see the proof of Proposition \ref{prop:CGO}\! (i)).
\end{proof}

\begin{lem} \label{lem:independence}
Let $k\geq 2$ be an integer and let $\varepsilon\in \{-1,+1\}$.
Let $C \in \calC[k]$  and let $L\subset C$ be an $\varepsilon$-framed knot of nilpotency class $\geq k$. 
Then
$$
\rho_{2k}(C_L) \in \Aut(\pi/\Gamma_{2k+1} \pi) 
$$
only depends on $C$, $\varepsilon$ and the free homotopy class of  $L$.
\end{lem}

\begin{proof}
We know from Lemma \ref{lem:surgery}\! (1) that $C_L \in \calC[k]$.
Let $L'$ be another $\varepsilon$-framed knot in $C$ which can be obtained from $L$ by a single ``crossing-change'' move.
Then it suffices to show that
\begin{equation} \label{eq:L_L'}
\rho_{2k}\big(C_L\big) = \rho_{2k}\big(C_{L'}\big).
\end{equation}

There exist $e\in \{-1,+1\}$ and a small disk $D$ in $C$ 
which meets  transversely~$L$ in two points with opposite signs,
such that surgery along the $e$-framed knot  $\partial D$ 
transforms $L$ to $L'$. 
By adding a tube to  $D \cap \big(C \setminus \operatorname{int\, T}(L)\big)$,
we obtain a Seifert surface of genus one for $\partial D$ in~$C_{L}$: therefore,
$\partial D$ is null-homologous in  $C_L$ and its framing number in $C_L$ is the same as in $C$.
Thus $C_{L'}$ is obtained from $C_{L}$ by surgery along the $e$-framed knot $\partial D$.
Furthermore, the loop $\partial D \subset C_L$ 
is homotopic to the concatenation of two meridians of $L$, say
$$
m,m'\subset C \setminus \operatorname{int\, T}(L) \subset C _{L}
$$
with opposite signs. Therefore
$$
[\partial D] = \mu (x \mu^{-1} x^{-1})  \in \pi_1\big(C_{L}\big)
$$
where $\mu :=[m] \in \pi_1\big(C_{L}\big)$, for some $x \in \pi_1\big(C_{L}\big)$.
By Proposition \ref{prop:CGO}\! (i) and our assumption on $L$, we have $\mu  \in \Gamma_{k}\pi_1\big(C_{L}\big)$.
It follows that $[\partial D]  \in \Gamma_{k+1} \pi_1\big(C_{L}\big)$ 
and \eqref{eq:L_L'} then follows from Lemma~\ref{lem:surgery}\! (2).
\end{proof}

\section{Proof of the main theorem}    \label{sec:main_theorem}

In this section, we prove Theorem A  in the introduction.

\subsection{Notations.}

Let $\gamma \subset \Sigma$ be a closed curve of nilpotency class $\ge k$, 
where $k\geq 2$ is an integer.
Take a representative of $\gamma$ as a loop based at~$\bullet$
 and express it as a product of commutators of length $k$ as in~\eqref{eq:gamcomm}:
\begin{equation} \label{eq:gamcomm_bis}
\gamma = \prod_{i=1}^l \mathcal{T}_i(\gamma_{i1},\ldots,\gamma_{ik}).
\end{equation}
As in Section \ref{subsec:intop}, $\pi_1(\Sigma,\bullet)$ is identified 
with $\pi_1(\Sigma,\star)=\pi$ using the orientation-preserving arc $\nu \subset \partial \Sigma$ connecting $\bullet$ to $\star$.

Let  $\star_{\pm} := (\star,\pm 1)$ and $\bullet_{\pm} := (\bullet,\pm 1)$ in $U =\Sigma \times [-1,+1]$.
We identify $\pi_1(U,\star_+)$ with $\pi_1(U,\star_-)$ 
by using the vertical segment $\{\star \} \times [-1,+1]$.
The same convention applies to the fundamental group of the knot exterior ${U \setminus \operatorname{int\, T}(L)}$ 
and that of the surgered manifold $U_L$ for any knot resolution $L$ of $\gamma$. In particular,
this convention is used in the definition of $\rho_{2k}(U_L)$ as the composition of two automorphisms.

\subsection{Adapted knot resolution}
\label{subsec:akr}

By Lemma \ref{lem:independence}, we know that $\rho_{2k}(U_L)$ is independent 
of the choice of a knot resolution $L$ of the curve $\gamma$.
Making use of the expression \eqref{eq:gamcomm_bis} of $\gamma$, 
we shall construct a knot resolution of $\gamma$ which will be convenient for our computation.

In what follows, we fix a sufficiently small positive real number $\delta$.

\vskip 1em
\noindent \textbf{Step 1.}
Let $D=[0,1]\times [0,1]$ be the unit square in $\mathbb{R}^2$.
Given a planar binary rooted tree $\mathcal{T}$ with $k$ leaves, we construct a tangle 
diagram $\widetilde{\mathcal{T}}$ on $D$  which represents an oriented tangle in $D\times [-1,+1]$.

First take an orientation-preserving embedding $\mathcal{T} \subset D$ such that the root of $\mathcal{T}$ is mapped to $(\frac{1}{2},0)$ and for each $j$ with $1\le j\le k$, the $j$th leaf of $\mathcal{T}$ is mapped to $\big(\frac{j}{k+1},1\big)$.

Next, we do the following operation around each trivalent vertex $v$ of~$\mathcal{T}$.
Let $Y_1$ be the following tangle diagram:
\[
{\unitlength 0.1in%
\begin{picture}(16.2000,10.4000)(2.0000,-14.0000)%
%
\special{pn 8}%
\special{ar 1160 1280 80 80 4.7123890 6.2831853}%
%
\special{pn 8}%
\special{ar 1440 1280 80 80 3.1415927 4.7123890}%
%
\special{pn 8}%
\special{pa 1160 1200}%
\special{pa 880 1200}%
\special{fp}%
%
\special{pn 8}%
\special{pa 1440 1200}%
\special{pa 1720 1200}%
\special{fp}%
%
\special{pn 8}%
\special{ar 880 1120 80 80 1.5707963 3.1415927}%
%
\special{pn 8}%
\special{ar 1720 1120 80 80 6.2831853 1.5707963}%
%
\special{pn 8}%
\special{ar 1700 1040 40 40 6.2831853 1.5707963}%
%
\special{pn 8}%
\special{pa 1240 1080}%
\special{pa 1360 1080}%
\special{fp}%
%
\special{pn 8}%
\special{ar 1360 1040 40 40 6.2831853 1.5707963}%
%
\special{pn 8}%
\special{ar 1240 1040 40 40 1.5707963 3.1415927}%
%
\special{pn 8}%
\special{ar 1100 1040 40 40 6.2831853 1.5707963}%
%
\special{pn 8}%
\special{ar 1500 1040 40 40 1.5707963 3.1415927}%
%
\special{pn 8}%
\special{pa 1240 1400}%
\special{pa 1240 1280}%
\special{fp}%
%
\special{pn 8}%
\special{pa 1360 1400}%
\special{pa 1360 1280}%
\special{fp}%
%
\special{pn 8}%
\special{ar 1280 1040 120 120 5.3558901 6.2831853}%
%
\special{pn 8}%
\special{ar 1280 1040 180 180 5.3558901 6.2831853}%
%
\special{pn 8}%
\special{ar 1320 1040 180 180 3.1415927 4.0688879}%
%
\special{pn 8}%
\special{ar 1320 1040 120 120 3.1415927 4.0688879}%
%
\special{pn 8}%
\special{ar 1100 530 180 180 2.2142974 3.1415927}%
%
\special{pn 8}%
\special{pa 1352 944}%
\special{pa 992 674}%
\special{fp}%
%
\special{pn 8}%
\special{ar 1100 530 120 120 2.2142974 3.1415927}%
%
\special{pn 8}%
\special{pa 1028 626}%
\special{pa 1388 896}%
\special{fp}%
%
\special{pn 8}%
\special{ar 900 1040 40 40 1.5707963 3.1415927}%
%
\special{pn 8}%
\special{pa 900 1080}%
\special{pa 1100 1080}%
\special{fp}%
%
\special{pn 8}%
\special{ar 1620 440 120 120 6.2831853 0.9272952}%
%
\special{pn 8}%
\special{ar 1620 440 180 180 6.2831853 0.9272952}%
%
\special{pn 8}%
\special{pa 1212 896}%
\special{pa 1236 878}%
\special{fp}%
%
\special{pn 8}%
\special{pa 1248 944}%
\special{pa 1288 914}%
\special{fp}%
%
\special{pn 8}%
\special{pa 1360 860}%
\special{pa 1728 584}%
\special{fp}%
%
\special{pn 8}%
\special{pa 1308 824}%
\special{pa 1692 536}%
\special{fp}%
%
\special{pn 8}%
\special{ar 1800 680 120 120 2.2142974 3.1415927}%
%
\special{pn 8}%
\special{ar 1800 680 180 180 2.2142974 3.1415927}%
%
\special{pn 8}%
\special{pa 1680 680}%
\special{pa 1680 640}%
\special{fp}%
%
\special{pn 8}%
\special{ar 1620 920 120 120 5.3558901 6.2831853}%
%
\special{pn 8}%
\special{ar 1620 920 180 180 5.3558901 6.2831853}%
%
\special{pn 8}%
\special{pa 1740 920}%
\special{pa 1740 1040}%
\special{fp}%
%
\special{pn 8}%
\special{pa 1800 920}%
\special{pa 1800 1120}%
\special{fp}%
%
\special{pn 8}%
\special{pa 1700 1080}%
\special{pa 1500 1080}%
\special{fp}%
%
\special{pn 8}%
\special{pa 800 1120}%
\special{pa 800 360}%
\special{fp}%
%
\special{pn 8}%
\special{pa 860 360}%
\special{pa 860 1040}%
\special{fp}%
%
\special{pn 8}%
\special{pa 920 360}%
\special{pa 920 530}%
\special{fp}%
%
\special{pn 8}%
\special{pa 980 360}%
\special{pa 980 530}%
\special{fp}%
%
\special{pn 8}%
\special{pa 1620 560}%
\special{pa 1620 360}%
\special{fp}%
%
\special{pn 8}%
\special{pa 1680 520}%
\special{pa 1680 360}%
\special{fp}%
%
\special{pn 8}%
\special{pa 1740 440}%
\special{pa 1740 360}%
\special{fp}%
%
\special{pn 8}%
\special{pa 1800 360}%
\special{pa 1800 440}%
\special{fp}%
\put(2.0000,-9.0000){\makebox(0,0)[lb]{$Y_1$\hspace{1em}$=$}}%
%
\special{pn 8}%
\special{pa 800 400}%
\special{pa 780 440}%
\special{fp}%
\special{pa 800 400}%
\special{pa 820 440}%
\special{fp}%
%
\special{pn 8}%
\special{pa 860 440}%
\special{pa 840 400}%
\special{fp}%
\special{pa 860 440}%
\special{pa 880 400}%
\special{fp}%
%
\special{pn 8}%
\special{pa 980 440}%
\special{pa 960 400}%
\special{fp}%
\special{pa 980 440}%
\special{pa 1000 400}%
\special{fp}%
%
\special{pn 8}%
\special{pa 920 400}%
\special{pa 900 440}%
\special{fp}%
\special{pa 920 400}%
\special{pa 940 440}%
\special{fp}%
%
\special{pn 8}%
\special{pa 1620 400}%
\special{pa 1600 440}%
\special{fp}%
\special{pa 1620 400}%
\special{pa 1640 440}%
\special{fp}%
%
\special{pn 8}%
\special{pa 1680 440}%
\special{pa 1660 400}%
\special{fp}%
\special{pa 1680 440}%
\special{pa 1700 400}%
\special{fp}%
%
\special{pn 8}%
\special{pa 1740 400}%
\special{pa 1720 440}%
\special{fp}%
\special{pa 1740 400}%
\special{pa 1760 440}%
\special{fp}%
%
\special{pn 8}%
\special{pa 1800 440}%
\special{pa 1780 400}%
\special{fp}%
\special{pa 1800 440}%
\special{pa 1820 400}%
\special{fp}%
%
\special{pn 8}%
\special{pa 1240 1280}%
\special{pa 1220 1320}%
\special{fp}%
\special{pa 1240 1280}%
\special{pa 1260 1320}%
\special{fp}%
%
\special{pn 8}%
\special{pa 1360 1320}%
\special{pa 1340 1280}%
\special{fp}%
\special{pa 1360 1320}%
\special{pa 1380 1280}%
\special{fp}%
%
\special{pn 8}%
\special{pa 1020 1080}%
\special{pa 980 1060}%
\special{fp}%
\special{pa 1020 1080}%
\special{pa 980 1100}%
\special{fp}%
%
\special{pn 8}%
\special{pa 1320 1080}%
\special{pa 1280 1060}%
\special{fp}%
\special{pa 1320 1080}%
\special{pa 1280 1100}%
\special{fp}%
%
\special{pn 8}%
\special{pa 1620 1080}%
\special{pa 1580 1060}%
\special{fp}%
\special{pa 1620 1080}%
\special{pa 1580 1100}%
\special{fp}%
%
\special{pn 8}%
\special{pa 1580 1200}%
\special{pa 1620 1180}%
\special{fp}%
\special{pa 1580 1200}%
\special{pa 1620 1220}%
\special{fp}%
%
\special{pn 8}%
\special{pa 980 1200}%
\special{pa 1020 1180}%
\special{fp}%
\special{pa 980 1200}%
\special{pa 1020 1220}%
\special{fp}%
\end{picture}}%
\]
For each positive integer $d$, let $L_d$ be the diagram
\begin{equation} \label{eq:cabling}
{\unitlength 0.1in%
\begin{picture}(19.2500,7.0500)(0.5000,-10.0500)%
%
\special{pn 8}%
\special{pa 1350 900}%
\special{pa 1350 300}%
\special{fp}%
%
\special{pn 8}%
\special{pa 1450 300}%
\special{pa 1450 900}%
\special{fp}%
%
\special{pn 8}%
\special{pa 1850 900}%
\special{pa 1850 300}%
\special{fp}%
%
\special{pn 8}%
\special{pa 1950 300}%
\special{pa 1950 900}%
\special{fp}%
%
\special{pn 8}%
\special{pa 1350 600}%
\special{pa 1325 650}%
\special{fp}%
\special{pa 1350 600}%
\special{pa 1375 650}%
\special{fp}%
%
\special{pn 8}%
\special{pa 1450 600}%
\special{pa 1425 550}%
\special{fp}%
\special{pa 1450 600}%
\special{pa 1475 550}%
\special{fp}%
%
\special{pn 8}%
\special{pa 1950 600}%
\special{pa 1925 550}%
\special{fp}%
\special{pa 1950 600}%
\special{pa 1975 550}%
\special{fp}%
%
\special{pn 8}%
\special{pa 1850 600}%
\special{pa 1825 650}%
\special{fp}%
\special{pa 1850 600}%
\special{pa 1875 650}%
\special{fp}%
\put(15.7500,-6.0000){\makebox(0,0)[lb]{$\cdots$}}%
%
\special{pn 8}%
\special{pa 800 300}%
\special{pa 800 900}%
\special{fp}%
\put(6.7500,-6.2500){\makebox(0,0)[lb]{$d$}}%
\put(10.0000,-6.2500){\makebox(0,0)[lb]{$=$}}%
\put(0.5000,-6.5000){\makebox(0,0)[lb]{$L_d$\hspace{1em}$=$}}%
\put(13.5000,-11.5000){\makebox(0,0)[lb]{$\underbrace{\hspace{4em}}_{2^d}$}}%
\end{picture}}%
\end{equation}
and let $Y_d$ be the tangle diagram obtained from $Y_1$ by replacing its strands with $L_{d-1}$.
(We understand that $L_0$ is a trivial strand.)
For simplicity, we denote
\[
{\unitlength 0.1in%
\begin{picture}(10.5000,7.0200)(2.5000,-12.0000)%
\put(2.5000,-9.2500){\makebox(0,0)[lb]{$Y_d$\hspace{1em}$=$}}%
%
\special{pn 8}%
\special{pa 900 800}%
\special{pa 1300 800}%
\special{pa 1300 1000}%
\special{pa 900 1000}%
\special{pa 900 800}%
\special{pa 1300 800}%
\special{fp}%
%
\special{pn 8}%
\special{pa 1100 1000}%
\special{pa 1100 1200}%
\special{fp}%
%
\special{pn 8}%
\special{ar 1042 606 120 120 2.2142974 3.1415927}%
%
\special{pn 8}%
\special{ar 898 798 120 120 5.3558901 6.2831853}%
%
\special{pn 8}%
\special{pa 922 606}%
\special{pa 922 566}%
\special{fp}%
%
\special{pn 8}%
\special{ar 1158 606 120 120 6.2831853 0.9272952}%
%
\special{pn 8}%
\special{ar 1302 798 120 120 3.1415927 4.0688879}%
%
\special{pn 8}%
\special{pa 1278 606}%
\special{pa 1278 566}%
\special{fp}%
\put(10.5000,-9.5000){\makebox(0,0)[lb]{$d$}}%
\put(10.0000,-11.7500){\makebox(0,0)[lb]{{\tiny $d$}}}%
\put(13.0000,-6.5000){\makebox(0,0)[lb]{{\tiny $d+1$}}}%
\put(6.5000,-6.5000){\makebox(0,0)[lb]{{\tiny $d+1$}}}%
\end{picture}}%
\]
Using the notation of Definition \ref{dfn:xi_j}, let $e_l, e_r,e_0$ be the edges around $v$.
Then replace a neighborhood of $v$ in $\mathcal{T}$ with the diagram $Y_{d(v)}$:
\[
{\unitlength 0.1in%
\begin{picture}(18.6000,6.9500)(2.4000,-12.0000)%
%
\special{pn 8}%
\special{pa 1700 800}%
\special{pa 2100 800}%
\special{pa 2100 1000}%
\special{pa 1700 1000}%
\special{pa 1700 800}%
\special{pa 2100 800}%
\special{fp}%
%
\special{pn 8}%
\special{pa 1900 1000}%
\special{pa 1900 1200}%
\special{fp}%
%
\special{pn 8}%
\special{ar 1842 606 120 120 2.2142974 3.1415927}%
%
\special{pn 8}%
\special{ar 1698 798 120 120 5.3558901 6.2831853}%
%
\special{pn 8}%
\special{pa 1722 606}%
\special{pa 1722 566}%
\special{fp}%
%
\special{pn 8}%
\special{ar 1958 606 120 120 6.2831853 0.9272952}%
%
\special{pn 8}%
\special{ar 2102 798 120 120 3.1415927 4.0688879}%
%
\special{pn 8}%
\special{pa 2078 606}%
\special{pa 2078 566}%
\special{fp}%
\put(17.5000,-9.7500){\makebox(0,0)[lb]{$d(v)$}}%
\put(16.7500,-11.7500){\makebox(0,0)[lb]{{\tiny $d(v)$}}}%
\put(21.0000,-6.5000){\makebox(0,0)[lb]{{\tiny $d(v)+1$}}}%
\put(13.0000,-6.5000){\makebox(0,0)[lb]{{\tiny $d(v)+1$}}}%
%
\special{pn 8}%
\special{pa 680 600}%
\special{pa 500 900}%
\special{fp}%
%
\special{pn 8}%
\special{pa 500 900}%
\special{pa 500 1200}%
\special{fp}%
%
\special{pn 8}%
\special{pa 320 600}%
\special{pa 500 900}%
\special{fp}%
\put(3.8000,-9.4000){\makebox(0,0)[lb]{$v$}}%
\put(3.6000,-11.0000){\makebox(0,0)[lb]{$e_0$}}%
\put(2.4000,-7.6000){\makebox(0,0)[lb]{$e_l$}}%
\put(6.4000,-7.6000){\makebox(0,0)[lb]{$e_r$}}%
%
\special{pn 4}%
\special{sh 1}%
\special{ar 500 900 16 16 0 6.2831853}%
\put(10.5000,-9.0500){\makebox(0,0)[lb]{$\leadsto$}}%
\end{picture}}%
\]
Here, given a (univalent or trivalent) vertex $w$ of $\mathcal{T}$,
we denote by $d(w)$ the \emph{depth} of $w$, 
that is the length of the shortest path from $w$ to the root of~$\mathcal{T}$.
Do this operation at each trivalent vertex of $\mathcal{T}$ 
and join the ends of the Y-shaped diagrams thus obtained along the edges of $\mathcal{T}$ in the obvious way.

Finally, we do the following operations around the root and the leaves of~$\mathcal{T}$.
In the above process, the root edge of $\mathcal{T}$ has been replaced with $L_1$:
we connect the two bottom ends of $L_1$ to $(\frac{1}{2}\mp \delta,0) \in \partial D$.
Let now $e$ be a leaf edge and $v=i(e)$.
Suppose that $t(e)$ is the $j$th leaf.
The edge $e$ has been replaced with $L_{d(v)+1}$, which consists of $2^{d(v)+1}$ strands.
We connect the left $2^{d(v)}$ strands, which we can view as $L_{d(v)}$, to $(\frac{j}{k+1}-\delta,1) \in \partial D$.
Similarly, we connect the right $2^{d(v)}$ strands to $(\frac{j}{k+1}+\delta,1)$.

The resulting tangle diagram in $D$ is denoted by $\widetilde{\mathcal{T}}$.

\vskip 1em
\noindent \textbf{Step 2.}
Let $R:=[0,l]\times [-\delta,1]$.
Applying the previous step to the trees $\mathcal{T}_1,\ldots,\mathcal{T}_l$, 
we obtain tangle diagrams $\widetilde{\mathcal{T}}_1,\ldots,\widetilde{\mathcal{T}}_l$ in $D$.
We put them in $R$ so that for each $i\in \{ 1,\dots, l\}$,
the $i$th diagram $\widetilde{\mathcal{T}}_i$ sits in the region $[i-1,i] \times [0,1]$.
Connect their bottom ends as follows.
First connect $(\frac{1}{2}-\delta,0)$ and $(l-\frac{1}{2}+\delta,0)$ by using a simple proper arc in $[0,l]\times [-\delta,0]$, 
and then for each $i\in \{ 1,\dots, l-1\}$,
connect $(i-\frac{1}{2}+\delta,0)$ and $(i+\frac{1}{2}-\delta,0)$ by using the horizontal segment between them:
\[
\input{connectedsum.tex}
\]
The resulting tangle diagram in $R$ is denoted by $\sharp_{i=1}^l \widetilde{\mathcal{T}}_i$.

Take an orientation-preserving embedding $\iota\colon R \to \Sigma$ such that $\iota^{-1}(\partial \Sigma) = [0,l] \times \{ -\delta \}$, $\bullet \in \iota(R)$, and $\star \notin \iota(R)$.
Hereafter we identify $R$ and $\iota(R)$.

\vskip 1em
\noindent \textbf{Step 3.}
Let $\operatorname{pr}\colon U = \Sigma \times [-1,+1] \to \Sigma$ be the projection onto the first factor.
For each $(i,j)$
with $1\le i\le l$ and $1\le j\le k$, we now construct 
a simple based loop $z_{ij}$ in $U$ such that the homotopy class of $\operatorname{pr}(z_{ij})$ is $\gamma_{ij}$.

Recall from Step 1 that the $j$th leaf of $\mathcal{T}_i$ has been
mapped to $v_{ij}:= \big(i-1+\frac{j}{k+1},1\big) \in R$.
Set
\[
v'_{ij}:=\big(i-1+\textstyle\frac{j}{k+1}-\delta,1\big),
\quad
v''_{ij}:=\big(i-1+\textstyle\frac{j}{k+1}+\delta,1\big).
\]
Take an immersed path $\gamma'_{ij}$ in $\Sigma \setminus \operatorname{int}(R)$ from $v'_{ij}$ to $v''_{ij}$ which represents~$\gamma_{ij}$.
We arrange that all the intersections and self-intersections of the paths $\{ \gamma'_{ij} \}_{i,j}$ consist of transverse double points.

Resolve the intersections and self-intersections of the paths $\{ \gamma'_{ij} \}_{i,j}$ in an arbitrary way 
to obtain a surface tangle  diagram $Z'$ in $\Sigma \setminus \operatorname{int}(R)$.
The corresponding tangle in  $(\Sigma \setminus \operatorname{int}(R)) \times [-1, +1]$
consists of  components $z'_{ij}$ for all $(i,j) \in \{1,\dots,l\}\times \{1,\dots,k\}$,
where $z'_{ij}$ is a string from $(v'_{ij},0)$ to $(v''_{ij},0)$ such that $\operatorname{pr}(z'_{ij})=\gamma'_{ij}$.

Finally, connect the initial point $(v'_{ij},0)$ of $z'_{ij}$ to $\bullet_-=(\bullet,-1)$ by using the vertical segment $\{ v'_{ij} \} \times [-1,0]$ and the horizontal straight segment from $(v'_{ij},-1)$ to $\bullet_-$.
Similarly connect the terminal point $(v''_{ij},0)$ to $\bullet_-$.
In this way the string 
$z'_{ij}$ determines a simple loop $z_{ij}$ in $U$.
Its homotopy class in $\pi_1(U,\bullet_-)$ is denoted by the same letter $z_{ij}$.

\vskip 1em
\noindent \textbf{Step 4.}
For each $(i,j) \in \{1,\dots,l\}\times \{1,\dots,k\}$,
let $d_{ij}:= d(v_{ij})-1$
where $d(v_{ij})$ is 
the depth of the $j$th leaf of $\mathcal{T}_i$.
We duplicate the part of the diagram $Z'$ corresponding to $z'_{ij}$  by replacing it with the diagram $L_{d_{ij}}$:
hence we obtain a new  surface tangle diagram $Z$ in $\Sigma \setminus \operatorname{int}(R)$.
Join this diagram 
and the tangle diagram $\sharp_{i=1}^l \widetilde{\mathcal{T}}_i$ in $R$ (constructed in Step~2) 
along $[0,l] \times \{1\} \subset \partial R$
in the obvious way. Then we obtain a surface knot diagram $\mathcal{Z}$ in $\Sigma$. 

\vskip 1em
We denote by $L\subset U$ the oriented knot represented by the diagram $\mathcal{Z}$.
By construction, the free homotopy class of $L$ is represented~by
\[
\prod_{i=1}^l \mathcal{T}_i(z_{i1},\ldots,z_{ik}) \in \pi_1(U,\bullet_-)
\]
and $L$ is a knot resolution of $\gamma$.

\begin{example}
If $l=1$ and
$\gamma = [[\gamma_1,\gamma_2],\gamma_3]$, then $\mathcal{Z}$ looks like 
\[
\input{resolutionexample.tex}
\]
\end{example}

\begin{lem}
\label{lem:parallel}
Let $L^{\parallel}$ be the parallel of the knot $L$ determined 
by the diagram~$\mathcal{Z}$ using the ``blackboard framing'' convention.
Then $L^{\parallel}$ is the preferred parallel of $L$ (in the sense of Section \ref{subsec:surgeries}).
\end{lem}

\begin{proof}
In $U= \Sigma \times [-1,+1]$, there is a notion of ``linking number'': see for instance \cite[Appendix B]{MM_Y3}.
(Actually three different notions of linking numbers ${\rm Lk}, {\rm Lk}_+, {\rm Lk}_-$ exist 
but they coincide for null-homologous knots.)
It easily follows from the definition of linking number that
the preferred parallel of $L$ is the unique parallel that 
has linking number zero with $L$. Therefore, it suffices to show that ${\rm Lk}(L,L^{\parallel}) =0$.

Now according to \cite[Lemma B.2 (1)]{MM_Y3}, 
the linking number can be computed from any surface link diagram.
In our situation, this reads\\
\begin{equation} \label{eq:lk}
{\rm Lk}(L,L^{\parallel})   = 
\sharp  \begin{array}{c}
\labellist
\small\hair 2pt
 \pinlabel {$L^{\parallel}$} [br] at 1 19
 \pinlabel {$L$} [bl] at 19 20
\endlabellist
\includegraphics[scale=1.0]{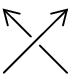}
\end{array}
-  \sharp  \begin{array}{c}
\labellist
\small\hair 2pt
 \pinlabel {$L$} [br] at 1 19
 \pinlabel {$L^{\parallel}$} [bl] at 19 20
\endlabellist
\includegraphics[scale=1.0]{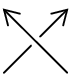}
\end{array}
\end{equation}
where we take as surface link diagram for $(L,L^{\parallel})$ the one 
resulting from the application of the ``blackboard framing'' convention to $\mathcal{Z}$.
Then any intersection of $L$ and $L^{\parallel}$ in that link diagram arises 
from a self-intersection of $L$ in $\mathcal{Z}$,
and all the self-intersections of $L$ can be formed into groups 
where each group has arisen 
either from an intersection/self-intersection 
in the surface tangle diagram $Z'$,
or, from a box in the Y-shaped diagrams that led to
$\sharp_{i=1}^l \widetilde{\mathcal{T}}_i$. 
Thus, with the notation \eqref{eq:cabling}, such a group of self-intersections of $L$ is always of the form
\[
{\unitlength 0.1in%
\begin{picture}(4.2000,4.0000)(0.6000,-8.0000)%
%
\special{pn 8}%
\special{pa 80 800}%
\special{pa 480 400}%
\special{fp}%
%
\special{pn 8}%
\special{pa 80 400}%
\special{pa 260 580}%
\special{fp}%
%
\special{pn 8}%
\special{pa 300 620}%
\special{pa 480 800}%
\special{fp}%
\put(0.6000,-7.2000){\makebox(0,0)[lb]{$d$}}%
\put(4.4000,-7.2000){\makebox(0,0)[lb]{$d'$}}%
\end{picture}}%
\]
for some $d,d'\ge 1$.
Since it contributes trivially to the right-hand side of~\eqref{eq:lk},
we have ${\rm Lk}(L,L^{\parallel}) =0$.
\end{proof}

\subsection{Multiple-meridians}
\label{subsec:mult-mer}

Now let $\varepsilon \in \{-1,+1\}$ and give $L$ the framing number $\varepsilon$.
For  each $(i,j) \in \{1,\ldots,l\} \times \{1,\ldots,k\}$, 
let $e_{ij}$ be the leaf edge of the $j$th leaf of $\mathcal{T}_i$.
In $L$, the edge $e_{ij}$ has been replaced with $L_{d_{ij}+1}$ which branches towards $(v_{ij},0)$:
one branch goes to $(v'_{ij},0)$ and the other to $(v''_{ij},0)$.
Locally both of them look like $L_{d_{ij}}$.
There is a small disk $\Delta$ in $\operatorname{int}(U)$ that intersects the first branch transversely near $(v'_{ij},0)$.
We orient it so that the orientation of $\Delta$ with the orientation of the edge $e_{ij}$
match the orientation of the ambient 3-manifold $U$.
Then the oriented boundary of $\Delta$, which ``clasps'' $2^{d_{ij}}$ strands of $L$ in the first branch,
 gives rise to a simple oriented closed curve in the complement of $L$ in $U$.
Further, connect $\partial \Delta$ to the base point $\star_- = (\star,-1)$ 
by using an arc of the vertical segment $\{ v'_{ij}\} \times [-1,0]$
and the horizontal straight segment from $(v'_{ij},-1)$ to $\star_-$.
Since this closed curve is in the complement of $L$, 
it determines an element $$\mu_{ij}\in \pi_1(U_L)=\pi_1(U_L,\star_-).$$

The rest of this subsection is devoted to the proof of the following proposition,
which is the key of Theorem A. 
Hereafter, we denote by $c_\pm: \Sigma \to U_L$ the boundary parametrizations of the homology cylinder $U_L$.

\begin{prop}
\label{prop:mu_ij}
We have $\mu_{ij} \in \Gamma_{2k-1}\pi_1(U_L)$.
Furthermore,
\[
\mu_{ij} \equiv c_-({\lambda_{ij}}^{-\varepsilon}) \quad \quad \mod \Gamma_{2k+1} \pi_1(U_L)
\]
where $\lambda_{ij}$ has been defined in \eqref{eq:lambda_ij}.
\end{prop}

Let $V$ be a handlebody of genus two,
and let $T$ be an oriented one-component tangle
in $V$ as shown in the left part of the following figure:
\[
\input{vt.tex}
\hspace{3em}
\input{vtmux.tex}
\]
Take a basepoint $*\in V\setminus T$ and consider the based loops $x_l$, $x_r$, $\mu$, $\mu_l$ 
and~$\mu_r$ as drawn in the right part of the figure above.
Note that $\mu$ is a meridian loop of $T$.
Then, it is easy to see that
\begin{equation}
\label{eq:handlebody}
\mu_l = [\mu,\overline{x_r}^{x_l}], \quad \text{and} \quad
\mu_r = [\mu^{\overline{x_l}},{x_l}^{x_r}] \quad \text{in $\pi_1(V\setminus T,*).$}
\end{equation}

For the moment, fix an index $i\in\{1,\dots,l\}$.
Let $e$ be an edge of $\mathcal{T}_{i}$ and $v:=i(e)$.
In the tangle diagram $\widetilde{\mathcal{T}}_{i}$, the edge $e$ has been replaced with $2^{d(v)+1}$ strands.
Let $\mu(e)$ be an oriented simple closed curve in $U\setminus L$ which ``clasps'' the left half $2^{d(v)}$ strands and connect it to the base point $\star_-$ in the same way as we have defined the curve $\mu_{ij}$.
For instance, if $e$ is the root edge of~$\mathcal{T}_i$, then $\mu_i:=\mu(e)$ is simply a meridian of $L$:
therefore we have
$$
\mu_i \in \Gamma_k \pi_1(U_L)
$$
by Proposition \ref{prop:CGO}. In general, we think of $\mu(e)$ as a ``multiple-meridian''.

\begin{lem}
\label{lem:mu(e)}
If $e$ is not the root edge of $\mathcal{T}_{i}$, then $\mu(e) \in \Gamma_{k+1} \pi_1(U_L)$.
\end{lem}

\begin{proof}
This follows from the facts that $\mu(e)$ is the product of an even number of meridians of $L$ with opposite signs 
and that any meridian of $L$ lies in $\Gamma_k \pi_1(U_L)$ by Proposition \ref{prop:CGO}.
\end{proof}

Let $\Upsilon_L$ be the subgroup of $\pi_1(U_L)$ that
is normally generated by the loops~$\mu(e)$, where $e$ runs through non-root edges of $\mathcal{T}_1,\ldots,\mathcal{T}_l$.
By Lemma~\ref{lem:mu(e)},
\begin{equation}
\label{eq:UpL}
\Upsilon_L \subset \Gamma_{k+1} \pi_1(U_L).
\end{equation}

\begin{prop}
\label{prop:mu_ij_ind}
Assume that $e$ is not the root edge of $\mathcal{T}_i$.
Then
\[
\mu(e) \equiv \Phi_e(\mathcal{T}_{i},c_-(\vec{\gamma}_{i}),\mu_i)
\mod \Gamma_{2k+1} \pi_1(U_L).
\]
\end{prop}

\begin{proof}
We use induction on $d:= d(v)$, where $v := i(e)$.
Assume that $d=1$.
Using the notation of Definition \ref{dfn:xi_j}, assume that $e=e_l$ (``Case 1''); 
the case $e=e_r$ (``Case 2'') is similar.
There exists an orientation-preserving embedding $\vartheta \colon V\to U$ 
such that $\vartheta(T)$ is the one-component tangle resulting from the union of the diagram $\widetilde{\mathcal{T}}_i$ 
with the part of $Z$ corresponding to~$z'_{i1},\dots,z'_{ik}$.

Let $y_l := \vartheta(x_l)$ and $y_r := \vartheta(x_r)$:
we regard them as elements in $\pi_1(U_L)$ by connecting them 
to $\star_-$ using the vertical segment from $\vartheta(*)$ 
to $(\operatorname{pr}(\vartheta(*)),-1)$ and the horizontal segment from $(\operatorname{pr}(\vartheta(*)),-1)$ to $\star_-$.
We can apply \eqref{eq:handlebody} to get $\mu(e) = [\mu_i, \overline{y_r}^{y_l}]$.
Let $\gamma_l := \xi(e_l)$ and $\gamma_r := \xi(e_r) \in \pi$.
Note that the difference of $y_l$ and $c_-(\gamma_l)$ lies in the group $\Upsilon_L$ 
and hence \eqref{eq:UpL} implies that $y_l \equiv c_-(\gamma_l) \mod \Gamma_{k+1} \pi_1(U_L)$.
By the same reason, we have $y_r \equiv c_-(\gamma_r) \mod \Gamma_{k+1} \pi_1(U_L)$.
Since $\mu_i \in \Gamma_k \pi_1(U_L)$, 
we conclude that 
$\mu(e)$ is congruent mod $ \Gamma_{2k+1} \pi_1(U_L)$ 
to $[\mu_i, c_-(\overline{\gamma_r}^{\gamma_l})]= \Phi_e(\mathcal{T}_{i},c_-(\vec{\gamma}_{i}),\mu_i)$
as required.

Next suppose that $d \ge 2$.
Again we assume that $e=e_l$, i.e$.$ ``Case 1'' in Definition \ref{dfn:xi_j},
the notation of which we still follow. The other case is similar.
Cut $\mathcal{T}$ at $i(e_0)$ and take the connected component containing $v$.
This component and the corresponding part in the diagram of $L$ look like
\[
\input{chopd.tex}
\]
If we regard the duplicated strand $L_{d-1}$ as a single strand, the diagram above is connected.
There is an orientation-preserving embedding $\vartheta \colon V \to U$ such that $\vartheta(T)$ 
is the diagram above, and it defines elements $y_l,y_r \in \pi_1(U_L)$ in the same way as in the case $d=1$.
Again we have 
the congruences  $y_l \equiv c_-(\gamma_l)$ and $y_r \equiv c_-(\gamma_r)$ mod $\Gamma_{k+1} \pi_1(U_L)$
where $\gamma_l := \xi(e_l)$ and $\gamma_r := \xi(e_r)$.
Applying \eqref{eq:handlebody} to the image $\vartheta(V)$, we obtain $\mu(e) = [\mu(e_0),\overline{y_r}^{y_l}]$.
But, by the inductive assumption, we have 
$\mu(e_0) \equiv \Phi_{e_0}(\mathcal{T}_{i},c_-(\vec{\gamma}_{i}),\mu_i)$ mod $\Gamma_{2k+1} \pi_1(U_L)$.
Hence
\begin{align*}
\mu(e) &\equiv \big[\Phi_{e_0}(\mathcal{T}_{i},c_-(\vec{\gamma}_{i}),\mu_i),\overline{y_r}^{y_l}\big] \mod \Gamma_{2k+1} \pi_1(U_L) \\
&\equiv \big[\Phi_{e_0}(\mathcal{T}_{i},c_-(\vec{\gamma}_{i}),\mu_i),c_-(\overline{\gamma_r}^{\gamma_l})\big] \mod \Gamma_{2k+1} \pi_1(U_L) \\
&= \Phi_e(\mathcal{T}_{i},c_-(\vec{\gamma}_{i}),\mu_i).
\end{align*}
In the last line we have used the defining formula for $\Phi_e(\mathcal{T},\vec{g},h)$ in Definition~\ref{dfn:xi_j}.
This completes the induction and the proof of the proposition.
\end{proof}

Applying Proposition \ref{prop:mu_ij_ind} to the case where $e=e_{ij}$ is 
the leaf edge of the $j$th leaf of $\mathcal{T}_i$, we obtain the following:

\begin{cor}
\label{cor:mu_ij}
$
\mu_{ij} \equiv \Phi_j(\mathcal{T}_i,c_-(\vec{\gamma}_i),\mu_i) \mod \Gamma_{2k+1} \pi_1(U_L).
$
\end{cor}

\begin{proof}[Proof of Proposition \ref{prop:mu_ij}]
Since $\mu_i \in \Gamma_k \pi_1(U_L)$, Corollary \ref{cor:mu_ij}
shows that $\mu_{ij} \in \Gamma_{2k-1} \pi_1(U_L)$.
This proves the first assertion.

Recall the curve $L^{\parallel}$ in Lemma \ref{lem:parallel}.
Since $L^{\parallel}$ is the preferred parallel of~$L$, 
we have $\mu_i = (L^{\parallel})^{-\varepsilon} \in \pi_1(U_L)$, 
where $L^{\parallel} \subset {U \setminus \operatorname{int} T(L)}$ is regarded as a knot in $U_L$ and is suitably based.
 Observe that we have $L^{\parallel} = [y_l,y_r]$, 
  where we use the same notation as in the second paragraph of the proof of Proposition \ref{prop:mu_ij_ind}.
Since $y_l \equiv c_-(\gamma_l)$ and $y_r \equiv c_-(\gamma_r)$ mod $\Gamma_{k+1} \pi_1(U_L)$, 
we have $\mu_i = (L^{\parallel})^{-\varepsilon} \equiv c_-([\gamma_l,\gamma_r])^{-\varepsilon} 
= c_-(\gamma)^{-\varepsilon}$ mod $\Gamma_{k+2} \pi_1(U_L)$. 
Then, Corollary \ref{cor:mu_ij} implies that
\begin{equation}
\label{eq:muij1}
\mu_{ij} \equiv \Phi_j\big(\mathcal{T}_i,c_-(\vec{\gamma}_i),c_-(\gamma)^{-\varepsilon}\big)
\mod \Gamma_{2k+1} \pi_1(U_L)
\end{equation}
since $\Phi_j(\mathcal{T}_i,c_-(\vec{\gamma}_i),\mu_i)$ is a commutator of (a conjugate of) $\mu_i$ and $k-1$ elements of $\pi_1(U_L)$
and  replacing $\mu_i$ with $c_-(\gamma)^{-\varepsilon}$ does not change its class mod $\Gamma_{2k+1} \pi_1(U_L)$.
(To see this, we repeatedly use the following fact about commutator calculus: for any elements $x,x',y$ in a group $G$ such that $x \equiv x' \mod \Gamma_a G$ and $y \in \Gamma_b G$, one has $[x,y] \equiv [x',y] \mod \Gamma_{a+b} G$.)
By using similar arguments, we obtain
\begin{equation}
\label{eq:muij2}
\Phi_j\big(\mathcal{T}_i,c_-(\vec{\gamma}_i),c_-(\gamma)^{-\varepsilon}\big)
\equiv \Phi_j(\mathcal{T}_i,c_-(\vec{\gamma}_i),c_-(\gamma))^{-\varepsilon} \mod \Gamma_{2k+1} \pi_1(U_L).
\end{equation}
(Here we use another fact about commutator calculus:
for any $x \in \Gamma_a G$ and $y\in G$, one has $[x^{-1},y] \equiv [x,y]^{-1} \mod \Gamma_{2a+1} G$.)

By \eqref{eq:muij1} and \eqref{eq:muij2}, we conclude that $\mu_{ij} \equiv  \Phi_j(\mathcal{T}_i,c_-(\vec{\gamma}_i),c_-(\gamma))^{-\varepsilon} = c_-(\lambda_{ij})^{-\varepsilon}$ mod $\Gamma_{2k+1} \pi_1(U_L)$.
This completes the proof.
\end{proof}

\subsection{Proof of Theorem A}

We continue to work with the knot resolution~$L$ of $\gamma$ constructed in Section \ref{subsec:akr}.
The following theorem computes the 
automorphism $\rho_{2k}(U_L)={c_-}^{-1}\circ c_+$ of $\pi/\Gamma_{2k+1}\pi$.

\begin{thm}
\label{thm:HCmonodromy}
Let $x\colon [0,1] \to \Sigma$ be a loop based at $\star$ which does not enter the rectangle $R$ 
and intersects the paths $\{\gamma'_{ij} \}_{i,j}$ in finitely many transverse double points.
Denote these intersection points
by $\{ p_1,\ldots,p_n\}$, so that $x^{-1}(p_1) < x^{-1}(p_2) < \cdots < x^{-1}(p_n)$.
For each $m\in \{1,\ldots,n\}$, there is a unique pair $(i_m,j_m)$ such that $p_m \in x\cap \gamma'_{i_mj_m}$:
let $\varepsilon_m = \varepsilon(p_m; \gamma'_{i_mj_m},x)$.
Then,
\[
\big(\rho_{2k}(U_L)(x)\big)\, x^{-1} \equiv
\Big(  \prod_{m=1}^n \big( {\lambda_{i_mj_m}}^{\varepsilon_m} \big)^{x_{\star p_m} 
\overline{ (\gamma_{i_mj_m})_{\bullet p_m}} \nu}  \Big)^{\varepsilon}  \mod \Gamma_{2k+1} \pi.
\]
Here, $\gamma_{i_mj_m} \in \pi$ is represented by the loop in $\Sigma$ based at $\bullet$ 
that is obtained from $\gamma'_{i_mj_m}$ by joining its endpoints to $\bullet$ by straight segments in $R$. 
\end{thm}

\begin{proof}
We have
\[
c_- \big( \big(\rho_{2k}(U_L)(x)\big)\, x^{-1} \big) = c_+(x)\, c_-(x^{-1}).
\]
For each $m \in \{1,\dots,n\}$,
take a small disk $\Delta$ in $\operatorname{int}(U)$ which intersects $z'_{ij}$ 
transversely at a point $q_m \in z'_{ij}$ such that $\operatorname{pr}(q_m)=p_m$.
We orient it so that the orientation of $\Delta$ with the orientation of $z'_{ij}$ match the orientation of $U$.
Then the oriented boundary of $\Delta$ defines 
an oriented closed curve in~$U$ 
which clasps the $2^{d_{ij}}$ parallel strands of $L$ corresponding to $z_{ij}$.
Regard $\partial \Delta$ as a loop based at $(p_m,-1)$ by using an arc of the vertical segment $\{ p_m\} \times [-1,+1]$.
We denote this loop by $\mu_{p_m}$. 
Then,
\[
c_+(x)\, c_-(x^{-1})= \prod_{m=1}^n
c_-(x_{\star p_m}) \, \big( {\mu_{p_m}}^{-\varepsilon_m} \big)\,\,
 \overline{c_-(x_{\star p_m})}.
\]
Now let $y_m$ be the piece of  path $\gamma'_{i_mj_m}$ from the point $v'_{i_mj_m}$ to $p_m$,
and let $z'_m$ be a parallel copy 
(using the ``blackboard framing'' convention) of the piece of string $z'_{i_mj_m}$ from $(v'_{i_mj_m},0)$ to $(p_m,0)$.
Connecting $\overline{z'_m}$ and $y_m$ by using two vertical segments between their endpoints, we obtain a loop $w_m$ in~$U_L$.
Then, by a suitable arc-basing of $w_m$, we see that
the loop $\mu_{p_m}$ is conjugate to
$$c_-(\overline{(\gamma_{i_mj_m})_{\bullet p_m}}\, \nu)\, \mu_{i_mj_m}\, c_-(\bar{\nu}\, (\gamma_{i_mj_m})_{\bullet p_m})$$
by $w_m$.
Since $w_m \in \Upsilon_L \subset \Gamma_{k+1} \pi_1(U_L)$ 
and   $\mu_{i_mj_m} \in \Gamma_{2k-1}\pi_1(U_L)$ by Proposition \ref{prop:mu_ij}, we have
\[
c_+(x)\, c_-(x^{-1}) \equiv 
\prod_{m=1}^n c_-\big(x_{\star p_m} \overline{(\gamma_{i_mj_m})_{\bullet p_m}}\, \nu \big) 
\, \big( {\mu_{i_mj_m}}^{-\varepsilon_m} \big)\,
c_-\big(\bar{\nu}\,(\gamma_{i_mj_m})_{\bullet p_m} \overline{x_{\star p_m}}\big)
\]
mod $\Gamma_{2k+1} \pi_1(U_L)$.
Applying Proposition \ref{prop:mu_ij}, we obtain
\[
c_+(x)c_-(x^{-1}) \equiv 
c_- \Big( \prod_{m=1}^n
x_{\star p_m}\overline{(\gamma_{i_mj_m})_{\bullet p_m}}\, \nu\,
\big( {\lambda_{i_mj_m}}^{\varepsilon_m} \big)\,  \bar{\nu}\, (\gamma_{i_mj_m})_{\bullet p_m} \overline{x_{\star p_m}}
\Big)^{\varepsilon}
\]
mod $\Gamma_{2k+1} \pi_1(U_L)$.
This proves the theorem.
\end{proof}

As we recalled at the beginning of Section \ref{subsec:akr}, 
$\rho_{2k}(U_L)$ is independent of the choice of a knot resolution $L$ of $\gamma$.
Comparing Theorem \ref{thm:HCmonodromy} with Theorem \ref{thm:gdtaction}, we find that
\[
\rho_{2k}(U_L)(x) \equiv  (t_{\gamma})^{\varepsilon}  (x) \mod \hpi_{2k+1}
\]
for any $x\in \pi$.
This completes the proof of Theorem A.

\begin{rem}
The reader familiar with the surgery techniques of \cite{Habiro}
should have noticed that the knot resolution~$L$  of $\gamma$ constructed in Section \ref{subsec:akr}
is obtained from the unknot by  surgery along $l$ tree claspers with $k-1$ nodes.
(The ``shape'' of this forest clasper is prescribed by the planar binary rooted trees $\mathcal{T}_1,\dots,\mathcal{T}_l$.)
It is possible to prove Theorem \ref{thm:HCmonodromy} by clasper calculus,
but, to avoid heavy graphical calculus and to make the proof as accessible as possible,
we have preferred to develop here ad hoc arguments.
\end{rem}

\section{Diagrammatic descriptions} \label{sec:diag_descriptions}

In this section, we give a ``diagrammatic'' version of Theorem A and its companion formulas.

\subsection{Diagrammatic versions of the Dehn--Nielsen representations}

We review some diagrammatic versions of the two generalizations 
of the Dehn--Nielsen representations, namely \eqref{eq:DN1} and \eqref{eq:DN2}, for generalized Dehn twists and homology cobordisms, respectively.

We start with homology cobordisms following \cite{Ma,HM}.
It is easily verified that the monoid homomorphism
$
\hat\rho\colon \calC \to \Aut(\hpi)
$
takes values in the subgroup $\Aut_\zeta(\hpi)$ consisting of automorphisms that fix 
$$\zeta:=[\partial \Sigma]\in \pi \subset \hpi.$$
Let  $\frakM(\pi)$ be the \emph{Malcev Lie algebra} of $\pi$, which is the primitive part
of the complete Hopf algebra $\hat{A} = \varprojlim_{k} \Q[\pi]/I^k$.
Then, instead of $\hat\rho$, one can equivalently consider the map
$$
\varrho\colon \calC \longrightarrow \Aut_{\log \zeta}\big(\frakM(\pi)\big), \ C \longmapsto \log \circ\, \hat\rho(C) \circ \exp
$$
where $\log\colon 1+  \hat{I}  \to  \hat{I}$ and $\exp\colon  \hat{I} \to 1+  \hat{I}$ 
are defined by the usual formal power series,
and $\Aut_{\log\zeta}\big(\frakM(\pi)\big)$ denotes the group of filtration-preserving automorphisms of $\frakM(\pi)$ 
that fix the logarithm of $\zeta$. Furthermore,
if one restricts oneself to the submonoid $\calI\calC=\calC[1]$ of homology cylinders, then $\varrho$ takes values in the subgroup
$\IAut_{\log \zeta}\big(\frakM(\pi)\big)$ of such automorphisms that induce the identity on the associated graded.

To make this map $\varrho$ more concrete, one considers a \emph{symplectic expansion} of $\pi$.
This means a multiplicative map
$$
\theta\colon \pi \longrightarrow \hat T(H^\Q)
$$
with values in the degree-completion of the tensor algebra of $H^\Q$, which satisfies $\theta(\zeta)=\exp(-\omega)$  and 
$$
\forall x \in \pi, \quad \theta(x) = \underbrace{1 + \{x\}_1 + (\deg\geq 2)}_{\hbox{\small group-like}}.
$$
That symplectic expansions $\theta$ do exist is readily verified \cite[Lemma 2.16]{Ma};
 see also \cite{Ku12} for an explicit combinatorial construction. 
We \emph{choose} one.
Note that for any $k\ge 1$, the restriction of $\theta$ to $\Gamma_k \pi$ is independent of $\theta$: indeed we have
\begin{equation} \label{eq:thetagam_k}
\forall x\in \Gamma_k\pi, \quad \theta(x) = 1+ \{ x\}_k + (\deg\geq k+1),
\end{equation}
where we view $\{ x\}_k \in \Gamma_k \pi/\Gamma_{k+1}\pi$
as an element of $\hat T_k(H^\Q)=(H^\Q)^{\otimes k}$ through the canonical embedding $\Gamma_k \pi/\Gamma_{k+1}\pi \cong \frakL_k \subset \frakL_k^\Q \subset \hat T_k(H^\Q)$.
See e.g. \cite[\S 3]{Kaw05} for this fact.

Next, the map $\theta$ can be extended by linearity and continuity to an isomorphism 
$\theta\colon \hat A \to\hat T(H^\Q)$ of complete Hopf algebras which, in turn, 
can be restricted to an isomorphism $\theta\colon \frakM(\pi) \to\hat \frakL^\Q$ of complete Lie algebras. 
Then, instead of $\varrho$, one can equivalently consider the map
$$
\varrho^\theta\colon \calC \longrightarrow \Aut_{\omega}\big(\hat\frakL^\Q\big), \ C \longmapsto \theta \circ\, \varrho(C) \circ \theta^{-1}.
$$
As in Section \ref{subsec:Jacobi}, let $\calT^\Q$ be the graded vector space of Jacobi diagrams,
its degree-completion being denoted by $\hat\calT^\Q$.
Hence, we can consider the composition
$$
\xymatrix{
 \calC[1] \ar[r]^-{\varrho^\theta}  \ar@/_2pc/@{-->}[rrr]^-{r^\theta}  
& \IAut_{\omega}\big(\hat\frakL^\Q\big) \ar[r]^-\log_-\cong  
& \hat \frakH^\Q \ar[r]^-{\eta^{-1}}_-\cong  & \hat \calT^\Q.
}
$$
By construction, for every $k\geq 2$,
 the truncation of $r^\theta$ to the degrees $1,\dots,k-1$ is tantamount 
 to the  homomorphism $\rho_{k}\colon \calC[1] \to \Aut(\pi/\Gamma_{k+1}\pi)$.

The case of generalized Dehn twists goes parallel to the case of homology cobordisms.
Thus, for any symplectic expansion $\theta$, we have an injective map
$$
r^\theta\colon \calW[1] \longrightarrow  \hat \calT^\Q , 
\ f \longmapsto  \eta^{-1}\log\big(\theta \circ (\log \circ f \circ \exp) \circ \theta^{-1}\big).
$$
For every  $k\geq 2$, the truncation of $r^\theta$ to the degrees $1,\dots,k-1$ is tantamount 
 to the  homomorphism  $\calW[1] \hookrightarrow \Aut(\hpi) \to \Aut(\hpi / \hpi_{k+1})$.

In both situations (homology cobordisms and generalized Dehn twists), 
the diagrammatic Dehn--Nielsen representation $r^\theta$ is equivalent to the usual one valued in $\Aut(\hpi)$.
But, in contrast with $\Aut(\hpi)$, the target $\hat\calT^\Q$ of $r^\theta$ is well understood \cite{GL,HP}: 
it has a unique Lie bracket that makes $\eta$ an isomorphism of complete Lie algebras, 
and its dimension can be easily derived in each degree from Witt's dimension formula for the free Lie algebra~$\frakL^\Q$.
Note that the Lie algebra $\hat\calT^\Q$ can also be viewed as a group, using the Baker--Campbell--Hausdorff (BCH) formula.

Here is an informal ``comparison table'' between the two versions of the Dehn--Nielsen representation:\\ 

{\small
\begin{tabular}{c|c|c}
     & usual Dehn--Nielsen & diagrammatic Dehn--Nielsen \\ \hline
  target & the group $\Aut(\hpi)$& the  Lie algebra $\hat \calT^\Q$ \\
  & \dots\ difficult to study & \dots \ with easy combinatorics \\ \hline
  multiplicative? &  YES & YES \\
  & & \dots\ if $\hat \calT^\Q$ has the BCH product\\ \hline
  intrinsically-defined ?& YES & NO \\
  & & \dots\ it depends on  $\theta$ \\ \hline
\end{tabular}}\\

\noindent

\subsection{A reformulation of Theorem A}

We now reformulate Theorem A
using the diagrammatic versions of the Dehn--Nielsen representation.

\begin{thm} \label{thm:trees}
Let $k\geq 2$ be an integer, let $\gamma \subset \Sigma$  be a closed curve of nilpotency class $\geq k$
and let $L\subset U$ be a knot resolution  of $\gamma$   with framing $\varepsilon \in \{-1,+1\}$.
Then, for any symplectic expansion $\theta$ of $\pi$, we have
\begin{equation} \label{eq:trees}
r^\theta(U_L) \equiv \varepsilon \cdot r^\theta(t_\gamma)
\equiv \frac{\varepsilon  }{2} \cdot\log \theta ([\gamma]) \hbox{\textbf{\,-\! -\! -}\,}  \log \theta ([\gamma]) 
\mod {\hat\calT}^\Q_{\geq 2k}.
\end{equation}
\end{thm}

Before proving this theorem, we make a few comments about its statement.
Since $U\in \calC[2k-2]$, we have $U_L \in \calC[2k-2]$ by Lemma \ref{lem:surgery}\! (1).
Similarly, we have $t_\gamma\in \calW[2k-2]$ by  Corollary~\ref{cor:tgamx}\! (2).
Therefore both $r^\theta(U_L)$ and $r^\theta(t_\gamma)$ start in degree $2k-2$.
Besides, since $\gamma$ has nilpotency class $\geq k$, we have
$$
\theta([\gamma])  \equiv_{k+2}  1 + \theta_{k}([\gamma]) + \theta_{k+1}([\gamma]) \ \in \hat T (H^\Q)
$$
where $\theta_k([\gamma])=\{ [\gamma]\}_k \in (H^\Q)^{\otimes k}$ by \eqref{eq:thetagam_k} 
and $\theta_{k+1}([\gamma]) \in (H^\Q)^{\otimes (k+1)}$; 
hence, using that $k\geq 2$, we get 
$$
\log \theta([\gamma])  \equiv_{k+2}  \theta_{k}([\gamma])  + \theta_{k+1}([\gamma]) \ \in  \hat \frakL^\Q.
$$
Consequently the series of Jacobi diagrams $\log \theta ([\gamma]) \hbox{\textbf{\,-\! -\! -}\,}   \log \theta ([\gamma])$, 
which is obtained by gluing ``root-to-root'' two copies of the series of planar binary rooted trees $\log \theta ([\gamma])$, satisfies
$$
\log \theta ([\gamma]) \hbox{\textbf{\,-\! -\! -}\,}   \log \theta ([\gamma])
\equiv_{2k}
\underbrace{\theta_{k}([\gamma]) \hbox{\textbf{\,-\! -\! -}\,}  \theta_{k}([\gamma])}_{\in  \calT^\Q_{2k-2}}
+ \underbrace{2\cdot  \theta_{k}([\gamma]) \hbox{\textbf{\,-\! -\! -}\,}  \theta_{k+1}([\gamma])}_{\in \calT^\Q_{2k-1}}.
$$
Thus the identity \eqref{eq:trees} in Theorem \ref{thm:trees} is equivalent to the following system:
\begin{equation} \label{eq:trees_bis}
\left\{\begin{array}{l}
r^\theta_{2k-2}(U_L) = \varepsilon \cdot r^\theta_{2k-2}(t_\gamma) 
= \frac{\varepsilon}{2} \cdot \theta_{k}([\gamma]) \hbox{\textbf{\,-\! -\! -}\,}  \theta_{k}([\gamma])\\
r^\theta_{2k-1}(U_L) = \varepsilon \cdot r^\theta_{2k-1}(t_\gamma) = \varepsilon \cdot \theta_{k}([\gamma]) \hbox{\textbf{\,-\! -\! -}\,}  \theta_{k+1}([\gamma])
\end{array}\right.
\end{equation}
In particular, we can deduce from this reformulation of \eqref{eq:trees} 
that its right-hand side is independent of the choices of orientation and  arc-basing  of the closed curve~$\gamma$, implicit in the statement of Theorem \ref{thm:trees} to consider $[\gamma] \in \pi$.

\begin{proof}[Proof of Theorem \ref{thm:trees}]
Of course, the fact that $r^\theta(U_L) \equiv_{2k} \varepsilon \cdot r^\theta(t_\gamma)$ follows directly from Theorem A. Therefore, it is enough to prove that
\begin{equation} \label{eq:r(gdt)}
r^\theta(t_\gamma) = \frac{1}{2} \cdot\log \theta ([\gamma]) \hbox{\textbf{\,-\! -\! -}\,}   \log \theta ([\gamma]) \ \in {\hat\calT}^\Q_{}.
\end{equation}
Recall from Section \ref{subsec:gdt} that $t_\gamma \in \Aut(\hpi)$ is 
the restriction of an automorphism  of the complete Hopf algebra $\hat A$, 
namely the exponential of the derivation 
$$
 D_{\gamma} = \sigma(L(\gamma)) 
\qquad \hbox{where} \
L(\gamma) = \left| \frac{1}{2} (\log [\gamma])^2 \right| \in \hat{\mathfrak{g}}.
$$ 
Therefore
\begin{eqnarray*}
r^\theta(t_\gamma) &=& \eta^{-1} \log \big( \theta \circ \exp( D_\gamma ) \big\vert_{\frakM(\pi)} \circ \theta^{-1} \big) \\
&=& \eta^{-1}  \big( \theta \circ \sigma(L(\gamma))\big\vert_{\frakM(\pi)} \circ \theta^{-1} \big) \\
&=& \eta^{-1} \Big( \big( \theta \circ \sigma(L(\gamma))  \circ \theta^{-1} \big)\big\vert_{\hat\frakL^\Q} \Big).
\end{eqnarray*}
Let $C(H^\Q)$ be the quotient of $T(H^\Q)$ by its subspace of commutators,
let $\vert -\vert \colon T(H^\Q)\to C(H^\Q)$ be the associated projection, 
and denote by $\hat C(H^\Q)$ the  degree completion. Since the vector space $\mathfrak{g}$
can be identified to the quotient of $A$ by its subspace of commutators,
 $\theta$ induces an isomorphism between $\hat{\mathfrak{g}}$ and  $\hat C(H^\Q)$.
According to \cite[Theorem 1.2.2]{KK14}
(see also \cite[~\S 10]{MT13}), there is a commutative diagram
$$
\xymatrix{
\hat{\mathfrak{g}} \times \hat A \ar[rr]^-\sigma  \ar[d]^-\cong _-{\theta \times \theta} && \hat A \ar[d]_-\cong ^-{\theta} \\
 \hat C(H^\Q)  \times \hat T(H^\Q) \ar[rr]_-{S} && \hat T(H^\Q)
}
$$ 
where the action $S$ of $\hat C(H^\Q)$  on $\hat T(H^\Q)$ by derivations is defined as follows:
any $c=\vert h_1\otimes \cdots \otimes h_m\vert \in C_m(H^\Q)$,  where $h_1,\ldots,h_m \in H^\Q$,  is mapped by $S$
to  the unique derivation of $\hat T(H^\Q)$ whose restriction to $H^\Q$ is
$$
\sum_{i=1}^m \omega(h_i,-) \cdot h_{i+1}\otimes \cdots \otimes h_m \otimes h_1 \otimes \cdots \otimes h_{i-1} 
\in \Hom\big(H^\Q, \hat T (H^\Q)\big).
$$
We deduce that 
$$
r^\theta(t_\gamma) = \eta^{-1} \Big( S\Big(\Big\vert \frac{1}{2} (\log \theta[\gamma])^2 \Big\vert \Big)\Big\vert_{\hat\frakL^\Q}\Big).
$$
 Finally, we obtain \eqref{eq:r(gdt)} by using the following formula:
\[
S(|uv|)|_{\frakL^{\Q}} = \eta( u \hbox{\textbf{\,-\! -\! -}\,} v)
\quad \text{for any $u,v \in \frakL^{\Q}$.}
\]
In fact, since $\big|[a,b]c\big|=\big|a[b,c]\big|$ for any $a,b,c \in T(H^{\Q})$, 
the proof of this formula is reduced to the case where $u$ is homogeneous of degree $1$, which can be seen from the definition of $\eta$ using Lemma 2.7.1 in \cite{KK14}.
\end{proof}

\begin{rem} \label{rem:truncation}
Theorem \ref{thm:trees} only involves the restriction of $r^\theta$ to $\calC[k]$ and its truncation
to the degrees $k,k+1,\dots, 2k-1$, which we denote by
$$
r^\theta_{[k,2k[} \colon \calC[k] \longrightarrow \bigoplus_{j=k}^{2k-1} \calT_j^\Q.
$$
Note that $r^\theta_{[k,2k[}$ is equivalent to $\rho_{2k}\colon \calC[k] \to \Aut(\pi/\Gamma_{2k+1}\pi)$
and, consequently, to the $k$th Morita homomorphism $M_k$.
By the BCH formula, it is a group homomorphism if $\bigoplus_{j=k}^{2k-1} \calT_j^\Q$ is merely viewed as an additive group.
\end{rem}

\section{Applications to Johnson homomorphisms} \label{sec:applications}

In this section, we prove Theorem B and Theorem C.

\subsection{Johnson homomorphisms}

Let $j\geq 1$ be an integer. Recall from \cite{GL} that the \emph{$j$th Johnson homomorphism}
$$
\tau_j\colon \calC[j] \longrightarrow \Hom\big(H, \Gamma_{j+1} \pi/\Gamma_{j+2}\pi \big) 
$$
assigns to any homology cobordism $C \in \calC[j]$ the homomorphism defined by
$$
\{x\}_1 \longmapsto  \rho_{j+1}( C)(x) \cdot x^{-1} \in \pi/\Gamma_{j+2} \pi
$$
for all $x\in \pi$. It can be checked that $\tau_j$ is well-defined and is a monoid homomorphism.
By  identifying in the canonical way $\Gamma_{j+1} \pi/\Gamma_{j+2}\pi$ with~$\frakL_{j+1}$,
we identify the target of $\tau_j$  to the group of derivations of $\frakL$ increasing degrees by $j$. 
Then, using  that $\rho_{j+1}(C)$ fixes $\zeta=[\partial \Sigma]$ modulo $\Gamma_{j+2}\pi$, 
it can be verified that 
$\tau_j$  takes values in the subgroup of symplectic derivations:
$$
\tau_j\colon \calC[j] \longrightarrow  \frakH_j .
$$
Note also  that, by identifying $H$ with 
$\Hom(H,\Z)$ via $h\mapsto \omega(h,-)$, 
we get an isomorphism between $\Hom(H, \frakL_{j+1})$ and $H \otimes \frakL_{j+1}$,
through which $\frakH_j$ corresponds to the kernel of the Lie bracket
$H \otimes \frakL_{j+1} \to \frakL_{j+2}$.

For generalized Dehn twists, one can similarly define a group homomorphism
$$
\tau_j\colon \calW[j] \longrightarrow  \frakH_j^\Q.
$$
However, note that the Lie ring $\frakH$ is replaced by its rational version $\frakH^\Q$
since we consider here the homomorphism $\calW \hookrightarrow \Aut(\hpi) \to \Aut(\hpi / \hpi_{j+2})$
and the quotient $\hpi_{j+1}/\hpi_{j+2}$ is isomorphic to $\frakL_{j+1}^\Q$.

In both situations (homology cobordisms and generalized Dehn twists), 
all the Johnson homomorphisms are determined by the diagrammatic Dehn--Nielsen representation.
Specifically, for any symplectic expansion $\theta$ of $\pi$ and for every $j\geq 1$, 
the following diagrams are commutative:

\begin{equation} \label{eq:hcob_rt}
\xymatrix{
\calC[1] \ar[r]^-{r^\theta}  & \hat \calT^\Q \ar@{->>}[d]^-{\operatorname{proj}}\\
\calC[j] \ar@{^{(}->}[u] \ar[r]_-{\eta^{-1} \tau_j} &\calT_j^\Q 
} 
\qquad \hbox{and} \qquad
\xymatrix{
\calW[1] \ar[r]^-{r^\theta}  & \hat \calT^\Q \ar@{->>}[d]^-{\operatorname{proj}}\\
\calW[j] \ar@{^{(}->}[u] \ar[r]_-{\eta^{-1} \tau_j} &\calT_j^\Q 
} 
\end{equation}

\subsection{Proof of Theorem B and generalizations}

Let $k\geq 2$ be an integer.

\begin{proof}[Proof of Theorem B]
Since $U\in \calC[2k-2]$, we have $U_L \in \calC[2k-2]$ by Lemma~\ref{lem:surgery}\! (1).
Similarly, we have $t_\gamma\in \calW[2k-2]$ by Corollary~\ref{cor:tgamx}\! (2).
Next, we choose a symplectic expansion $\theta$ and compute
$$
\eta^{-1}\tau_{2k-2}(U_L)\stackrel{\eqref{eq:hcob_rt}}{=}  r_{2k-2}^\theta(U_L)
  \stackrel{\eqref{eq:trees_bis}}{=}   
  \frac{\varepsilon}{2} \cdot \theta_{k}([\gamma]) \hbox{\textbf{\,-\! -\! -}\,}  \theta_{k}([\gamma]). 
$$
Along the same lines, we obtain 
$\eta^{-1}\tau_{2k-2}(t_\gamma)=  \frac{1}{2} \cdot \theta_{k}([\gamma]) \hbox{\textbf{\,-\! -\! -}\,}  \theta_{k}([\gamma])$.
 We conclude since $\theta_k([\gamma])=\{[\gamma]\}_k$ by \eqref{eq:thetagam_k}.
\end{proof}

We now give a generalization of Theorem B. To state this, we firstly observe the following:
for any homology cylinder $C$, there is a canonical isomorphism 
of graded Lie rings
\begin{equation}\label{eq:canonical_iso}
\frakL \stackrel{\cong }{\longrightarrow} \bigoplus_{k\geq 1}\frac{\Gamma_k \pi_1(C)}{\Gamma_{k+1}\pi_1(C)}
\end{equation}
which is defined as the diagonal of the commutative diagram
$$
\xymatrix{
\frakL\ar[d]^-\cong _-{c_\pm} \ar[r]^-\cong  & \bigoplus_{k\geq 1} \frac{\Gamma_k  \pi}{\Gamma_{k+1} \pi}\ar[d]^-{c_\pm}_-\cong  \\
\frakL(H_1(C)) \ar[r] & \bigoplus_{k\geq 1}\frac{\Gamma_k \pi_1(C)}{\Gamma_{k+1}\pi_1(C)}
} 
$$
Here $\frakL(H_1(C))$ denotes the Lie ring freely generated by $H_1(C)$,
the vertical maps are induced by the boundary parametrizations $c_\pm \colon \Sigma \to C$, 
and the horizontal maps are the canonical homomorphisms of graded Lie rings. 
We can equally take $c_+$ or $c_-$ in that diagram since $c_+=c_-\colon H\to H_1(C)$.

\begin{thm} \label{thm:generalized_B}
Let $k\geq 2$ be an integer and let $C\in \calC[2k-2]$.
For any knot $L \subset C$ with framing $\varepsilon \in \{-1,+1\}$ of nilpotency class $\geq k$, 
we have $C_L \in \calC[2k-2]$ and 
\begin{equation} \label{eq:even_bis}
\tau_{2k-2}(C_L) = \tau_{2k-2}(C) + \frac{\varepsilon}{2} \cdot \eta( \Lambda \hbox{\textbf{\,-\! -\! -}\,} \Lambda)
\end{equation}
where $\Lambda \in \Gamma_k \pi_1(C)/ \Gamma_{k+1}\pi_1(C)$  denotes the class of $[L]$ 
and is regarded as an element of  $\frakL_k$  via the canonical isomorphism \eqref{eq:canonical_iso}.
\end{thm}

\noindent
Implicit in this statement are the choices of an orientation of $L$ 
and of an arc connecting $L$ to the base point of $C$ to have $[L] \in \pi_1(C)$ well-defined.
It is easily checked that  $\Lambda \hbox{\textbf{\,-\! -\! -}\,} \Lambda$ is independent of these choices.
(In fact, the same remark applies to Theorem B.)

\begin{proof}[Proof of Theorem \ref{thm:generalized_B}]
Since $C\in \calC[2k-2]$, we have $C_L \in \calC[2k-2]$ by Lemma~\ref{lem:surgery}.
We now prove the second statement of the theorem.

Here, we use the techniques of ``clasper surgery'' developed by Goussarov and Habiro,
and we  follow the terminology of \cite{Habiro}.
Let  $G$ be a tree clasper in ${C}$ with $2k$ nodes and, therefore, $2k+2$ leaves:
we assume that $G$ is disjoint from $L$ and that one leaf of $G$ bounds a disk $D$ which intersects transversely $L$ in one point.
Then, by elementary ``clasper calculus'', we get the following:
\begin{enumerate}
\item Surgery along $G$ does not change $C$ (up to diffeomorphism) but modifies $L$ to a new  $\varepsilon$-framed knot $L_G$ in $C$.
\item The homotopy class of $L_G$ is the product  of $[L]\in \pi_1(C)$ by the element of $\Gamma_{2k+1} \pi_1(C)$ 
that is defined by the planar binary rooted tree corresponding to the pair $(G,D)$.
\end{enumerate} 
Note also that any commutator of length $2k+1$ in $\pi_1(C)$ can be realized by such a tree clasper $G$.
Besides, $\rho_{2k}$ is invariant under surgery along a graph clasper with $2k$ nodes.
(This can be proved as follows: let $D$ be a homology cylinder and let $H\subset D$ be a graph clasper with $2k$ nodes; 
then $D_H$ can  be obtained  from $D$ by cutting it along a compact  oriented surface
$S\subset D$ and by regluing it with an element of the $2k$th term of the lower central series
of the Torelli group of~$S$ \cite[Theorem 5.7]{HM}; 
since this subgroup is contained in the $2k$th term of the Johnson filtration of the mapping class group of $S$,
a van Kampen argument shows that there exists an isomorphism 
between $\pi_1(D)/ \Gamma_{2k+1} \pi_1(D)$ and $\pi_1(D_H)/ \Gamma_{2k+1} \pi_1(D_H)$
which is compatible with the boundary parametrizations; it follows that $\rho_{2k}(D) =\rho_{2k}(D_H)$.)
Therefore
$$
\rho_{2k}\big(C_{(L_G)}\big) =  \rho_{2k}\big((C_{L})_G\big) = \rho_{2k}\big(C_L\big).
$$ 
Thus, we have improved the conclusions of Lemma \ref{lem:independence} as follows:
$\rho_{2k}(C_L)$ only depends on $C$, $\varepsilon$ 
and the conjugacy class of $L$  in the nilpotent quotient $\pi_1(C)/\Gamma_{2k+1} \pi_1(C)$. 
Choose a symplectic expansion $\theta$ of $\pi$.
By Remark \ref{rem:truncation}, $r_{[k,2k[}^\theta(C_L)$ only depends on the same data.

Since the homomorphism
$\pi/ \Gamma_{2k+1} \pi \to \pi_1(C)  /\Gamma_{2k+1} \pi_1(C)$
induced by the parametrization $c_+\colon \Sigma \to C$ of $\partial_+C$ is an isomorphism, 
we can assume that $L$ sits in a collar neighborhood $N$ of $\partial_+C$. 
Let $L' \subset U$ be the $\varepsilon$-framed knot in the usual cylinder that corresponds to $L \subset N$
through the homeomorphism 
$$
c_+ \times  \operatorname{id}_{[-1,+1]}  \colon U= \Sigma \times [-1, + 1] \stackrel{\cong}{\longrightarrow} \partial_+C \times [-1,+1] = N.
$$
Therefore
\begin{eqnarray}
\label{eq:r}\qquad r_{[k,2k[}^\theta\big(C_L\big) &=& r_{[k,2k[}^\theta \big( C \circ U_{L'} \big) \\
\notag &=& r_{[k,2k[}^\theta (C) +  r_{[k,2k[}^\theta \big(  U_{L'} \big)\\
\notag &\stackrel{\eqref{eq:trees_bis}}{=}& r_{[k,2k[}^\theta (C) + \frac{\varepsilon}{2} \cdot \theta_{k}(\ell) \hbox{\textbf{\,-\! -\! -}\,}  \theta_{k}(\ell) +
\varepsilon \cdot \theta_{k}(\ell) \hbox{\textbf{\,-\! -\! -}\,}  \theta_{k+1}(\ell)
\end{eqnarray}
where $\ell \in \pi$ is any element such that $\{\ell\}_{k+1}$ is mapped to $\{[L]\}_{k+1}$ by 
the isomorphism $c_+\colon \pi / \Gamma_{k+2} \pi \to \pi_1(C)  /\Gamma_{k+2} \pi_1(C)$.
In particular, using  \eqref{eq:hcob_rt}, we obtain
$$
\tau_{2k-2}(C_L) =  \tau_{2k-2}(C) + \frac{\varepsilon}{2} \cdot \theta_{k}(\ell) \hbox{\textbf{\,-\! -\! -}\,}  \theta_{k}(\ell).
$$
To conclude the proof, it remains to notice that the element $\theta_k(\ell) \in \frakL_k$, which is canonically identified with $\{\ell\}_k \in \Gamma_k \pi / \Gamma_{k+1}\pi$ by \eqref{eq:thetagam_k}, is mapped to $\{[L]\}_k \in \Gamma_k \pi_1(C) / \Gamma_{k+1}\pi_1(C)$ via the isomorphism induced by $c_+$.
\end{proof}

\begin{rem}\label{rem:special_cases_B}
Some special cases of Theorem B and Theorem \ref{thm:generalized_B} are already known.
Most of the case $k=2$ can be  found in the literature:
when $\gamma$ is a bounding simple closed curve in~$\Sigma$, formula \eqref{eq:even} for $\tau_2(t_\gamma)$ appears in \cite[Proposition 1.1]{Mor89}
and, when $L$ is a null-homologous $\varepsilon$-framed knot in a homology cylinder~$C$, formula~\eqref{eq:even_bis} for $\tau_2(C_L)$ can be deduced from \cite[Lemma~3.3]{MM_Y3}.
Furthermore, if  $L \subset C $ is obtained from the  $\varepsilon$-framed unknot by surgery along an ``admissible tree clasper'' with ${(k-1)}$ ``nodes''
(in the sense of \cite{Habiro}), then 
\eqref{eq:even} for  $\tau_{2k-2}(C_L)$  may be deduced fro \cite[\S 3.8]{CST}.
\end{rem} 

\begin{rem}
In fact, the arguments used to prove \eqref{eq:r} show the following: under the assumptions of Lemma \ref{lem:independence}, 
$\rho_{2k}(C_L)$ depends only on $\rho_{2k}(C)$, $\varepsilon$ and the conjugacy class of $L$ in $\pi_1(C)/\Gamma_{k+2}\pi_1(C)$.
\end{rem}

\subsection{Proof of Theorem C and generalizations}

Let $k\geq 2$ be an integer. 

\begin{proof}[Proof of Theorem C]
By assumption, $\{[\gamma_+]\}_k = \{[\gamma_-]\}_k \in \Gamma_k \pi/\Gamma_{k+1} \pi$
and we deduce from Theorem B that $\tau_{2k-2}(U_{L_+})=-\tau_{2k-2}(U_{L_-})$: 
consequently, $U_{L_-}U_{L_+}$ belongs to $\calC[2k-1]$. 
Next, we have 
\begin{eqnarray*}
\eta^{-1}\tau_{2k-1}(U_{L_-}U_{L_+})
&\stackrel{\eqref{eq:hcob_rt}}{=}&  r_{2k-1}^\theta(U_{L_-}U_{L_+}) \\
&= &r_{2k-1}^\theta(U_{L_-}) +r_{2k-1}^\theta( U_{L_+}) \\
&\stackrel{\eqref{eq:trees_bis}}{=} & 
  \theta_{k}([\gamma_+]) \hbox{\textbf{\,-\! -\! -}\,}  \theta_{k+1}([\gamma_+]) 
-   \theta_{k}([\gamma_-]) \hbox{\textbf{\,-\! -\! -}\,}  \theta_{k+1}([\gamma_-]).
\end{eqnarray*}
Besides, we have
\begin{eqnarray*}
\theta(\delta) &= \quad \ & \theta([\gamma_+]) \cdot \theta([\gamma_-])^{-1}\\
 &\equiv_{k+2}& (1+\theta_k([\gamma_+]) + \theta_{k+1}([\gamma_+]) ) \cdot (1- \theta_k([\gamma_-]) - \theta_{k+1}([\gamma_-]) ) \\
&\equiv_{k+2}& 1 + \theta_{k+1}([\gamma_+]) - \theta_{k+1}([\gamma_-]). 
\end{eqnarray*}
Therefore, we get
$$
\eta^{-1}\tau_{2k-1}(U_{L_-}U_{L_+}) = \theta_{k}([\gamma_+]) \hbox{\textbf{\,-\! -\! -}\,} \theta_{k+1}(\delta).
$$

Along the same lines, we prove that $(t_{\gamma_-})^{-1}t_{\gamma_+}$ belongs to $\calW[2k-1]$, and we obtain 
$$
\eta^{-1}\tau_{2k-1}\big((t_{\gamma_-})^{-1}t_{\gamma_+}\big)=  \theta_{k}([\gamma_+]) \hbox{\textbf{\,-\! -\! -}\,} \theta_{k+1}(\delta).
$$
We conclude  using the fact \eqref{eq:thetagam_k}, which says that $\theta_k([\gamma_+])=\{[\gamma_+]\}_k$ and $\theta_{k+1}(\delta)=\{\delta\}_{k+1}$.
\end{proof}
 
We finish with a wide generalization of  Theorem C. 
Recall from \cite{GL_blinks} that a \emph{blink} 
in a compact oriented $3$-manifold $M$ is a $2$-component framed link $L=(L_+,L_-)$ in  $M$, 
such that $L_+$ and $L_-$ cobound a compact orientable surface (called a \emph{connecting surface} of $L$) 
with respect to which the framing of $L_\pm$  is $\pm 1$. A~blink $L$ is \emph{coherently oriented}
if the knots $L_+$ and $L_-$ are oriented in such a way that $L_+\cup (-L_-)=\partial B$ 
for an oriented connecting surface $B$ of $L$.

Let $k\geq 2$ be an integer. We shall say  that a blink $L$ is of \emph{class~$\geq k$}
if both $L_+$ and $L_-$ are of nilpotency class~$\geq k$ and they are the same modulo $\Gamma_{k+1}\pi_1(M)$,
after  an arbitrary choice of arcs connecting $L_\pm$ to the base point of $M$
and the choice of a coherent orientation  of $L$.

\begin{thm} \label{thm:generalized_C}
Let $k\geq 2$ be an integer, let $C \in \calC[2k-1]$ and let $L$ be a blink of class $\geq k$ in $C$.
We set
\begin{equation} \label{eq:Lambda}
\Lambda := \{[L_\pm]\}_k \in \frac{\Gamma_k \pi_1(C)}{\Gamma_{k+1}\pi_1(C)} \stackrel{\eqref{eq:canonical_iso}}{\cong } \frakL_k
\end{equation}
and 
\begin{equation} \label{eq:Delta}
\Delta := \big\{ [L_+] [L_-]^{-1} \big\}_{k+1}\in 
 \frac{\Gamma_{k+1}\pi_1(C)}{\Gamma_{k+2}\pi_1(C)} \stackrel{\eqref{eq:canonical_iso}}{\cong } \frakL_{k+1}.
\end{equation}
Then,  $C_{L}$ belongs to $\calC[2k-1]$ and we have
\begin{equation} \label{eq:odd_bis}
\tau_{2k-1}(C_L) = \tau_{2k-1}(C) +
 \eta ( \Lambda \hbox{\textbf{\,-\! -\! -}\,} \Delta).
\end{equation}
\end{thm}

\noindent
Implicit in this statement are the choices of {a coherent orientation of $L$}
and of arcs connecting $L_\pm$ to the base point of $C$ to have $[L_\pm] \in \pi_1(C)$ well-defined.
It is easily checked that  $\Lambda \hbox{\textbf{\,-\! -\! -}\,} \Delta$ is independent 
of these choices.

Observe that Theorem \ref{thm:generalized_C} is indeed a generalization of Theorem C.
With the assumptions of the latter, let $\Sigma_\pm$ be a Seifert surface of $L_\pm \subset U$
and, in the product of two copies of $U$, connect by a tube  $\Sigma_+$ (in the ``top'' copy of~$U$)
to $\Sigma_-$ (in the ``bottom'' copy of $U$) in order to get an oriented surface $B \subset U\, U \cong U$:
then $L:=(L_+,L_-)$ is a blink with connecting surface $B$.

We need two preparation lemmas to prove Theorem \ref{thm:generalized_C}.
Recall that there is a notion of ``linking number'' in any homology cylinder  \cite[Appendix~B]{MM_Y3}.

\begin{lem} \label{lem:blink_Lk}
Let $C\in \calI\calC$ and $L$ be a blink in $C$ of class $k\geq 2$, which we orient coherently.
Then the preferred parallel of $L_\pm$ differs from the parallel defined by a connecting surface $B$ of $L$
by the integer $\operatorname{Lk}(L_+,L_-)$.
\end{lem} 

\begin{proof}
Let $\lambda_0:=\lambda_0(L_+) \subset \partial \operatorname{T}(L_+)$ 
be the preferred parallel of $L_+$ and let $\lambda:= \partial (C\setminus \operatorname{int\,T}(L_+)) \cap B$ 
be the parallel defined by $B$.
There is a unique integer $r$ such that
$$
[\lambda_0] = [\lambda] + r [\mu]  \ \in H_1\big(\partial \operatorname{T}(L_+) \big)
$$
where $\mu:=\mu(L_+)$ denotes the oriented meridian of $L_+$, and we are asked to show that
\begin{equation} \label{eq:rrr}
r=\operatorname{Lk}(L_+,L_-).
\end{equation}

Let $S$ be a Seifert surface for $L_+$ which is transverse to $L_-$. The boundary of the compact oriented surface
$$
S^\circ  := S \cap E
\quad \hbox{where} \
E:= C \setminus \operatorname{int}\big(\operatorname{T}(L_+) \cup \operatorname{T}(L_-)\big)
$$
consists of $\lambda_0$, $r_+$ copies of the oriented meridian $\mu':=\mu(L_-)$ and $r_-$ copies of $-\mu'$.
Let $\lambda'$ be the parallel of $L_-$ defined by $B$. Then we obtain 
$$
r= [\lambda] \bullet_{\partial E }  [\lambda_0]
=  [\lambda] \bullet_{ E } [ S^\circ] = [\lambda'] \bullet_{ E } [ S^\circ] =  
[\lambda'] \bullet_{\partial E } [(r_+-r_-)\mu'] = r_+ - r_-  .
$$
Since $\operatorname{Lk}(L_+,L_-)=r_+ -r_-$, this proves \eqref{eq:rrr}.
\end{proof}

\begin{lem} \label{lem:one_fact}
Let $k\geq 2$ be an integer. Let $C\in \calI\calC$   and  let $K\subset C $ be a $\varepsilon$-framed knot of nilpotency class $\geq k$, 
 where $\varepsilon \in \{ -1,+1\}$. 
\begin{enumerate}
\item Let $J\subset C$ be an oriented knot of nilpotency class $\geq k$  disjoint from~$K$. 
Then $J \subset {C}_K$ is also of nilpotency class $\geq k$ and, if ${\operatorname{Lk}(K,J)=0}$ in $C$, 
then its class $\{[J] \}_k \in  \frac{\Gamma_{k}\pi_1(C_K)}{\Gamma_{k+1}\pi_1(C_K)}$ corresponds to the class
$\{[J] \}_k \in  \frac{\Gamma_{k}\pi_1(C)}{\Gamma_{k+1}\pi_1(C)}$ through the canonical isomorphisms
\begin{equation*}
\frac{\Gamma_{k}\pi_1(C_K)}{\Gamma_{k+1}\pi_1(C_K)} \stackrel{\eqref{eq:canonical_iso}}{\cong } \frakL_{k}
\stackrel{\eqref{eq:canonical_iso}}{\cong } \frac{\Gamma_{k}\pi_1(C)}{\Gamma_{k+1}\pi_1(C)}.
\end{equation*}
\item Let $\mu(K)$ be the oriented meridian of $K$ regarded as a based loop in $C_K$. Then $\{\mu(K)\}_k \in \frac{\Gamma_{k}\pi_1(C_K)}{\Gamma_{k+1}\pi_1(C_K)}$ corresponds to $\{[K]^{-\varepsilon} \}_k \in \frac{\Gamma_{k}\pi_1(C)}{\Gamma_{k+1}\pi_1(C)}$ 
through the above isomorphisms.
\end{enumerate}
\end{lem}

\begin{proof}
 The second assertion follows immediately from the first one.
In fact,  we have $\mu(K) = (\lambda_0)^{-\varepsilon}$ in $\pi_1(C_K)$, 
where $\lambda_0 = \lambda_0(K)$ is the preferred parallel of $K$ with a suitable  arc-basing. 
Since $\operatorname{Lk}(K,\lambda_0)=0$ in $C$, and since $\lambda_0$  is homotopic to $[K]$ in $C$, the first assertion applies. 

In the sequel, we prove the first assertion. We  denote by $E:= C \setminus \operatorname{int\, T}(K)$ the exterior of the knot $K$,
we set $C':=C_K$ and, for the sake of clarity, we denote by $J'$ the knot $J$ regarded in $C'$.
According to Proposition \ref{prop:CGO}\! (ii), surgery along $K$ induces an isomorphism $\phi$ such that
$$ 
\begin{array}{c} 
\xymatrix @!0 @R=1cm @C=3cm {
  & \frac{\pi_1(C )}{\Gamma_k \pi_1(C)}  \ar@{->}[dd]^-\phi_-\cong  \\
\frac{\pi_1(E)}{\Gamma_k \pi_1(E)} \ar@{->>}[ru]  \ar@{->>}[rd] & \\
  & \frac{\pi_1(C')}{\Gamma_k \pi_1(C')}.
 }   
\end{array} 
$$
Since $\{J\}_{k-1}=1\in \pi_1(C)/\Gamma_k \pi_1(C)$, we have 
$$\{J'\}_{k-1}= \phi(\{J\}_{k-1})=1 \in  \pi_1(C')/\Gamma_k \pi_1(C')$$
which proves that $J'$ is of nilpotency class $\geq k$.

Let $\ell  \in \pi_1(C)$ and $\ell^* \in \pi_1(E)$ be the homotopy classes of $J$
in $C$ and $E$, respectively. Since $\ell \in \Gamma_k  \pi_1(C)$, we have that
\begin{equation} \label{eq:ell^*}
\ell^* = c \cdot \prod_{i=1}^r \mu_i^{\epsilon_i} \in \pi_1(E)
\end{equation}
where $c \in \Gamma_k  \pi_1(E)$, $\epsilon_1,\dots, \epsilon_r$ are signs and $\mu_1,\dots,\mu_r$ are copies of the oriented meridian of $K$ with possibly different arc-basings. 
Let $j\colon  E \to C$ and $j'\colon  E \to C'$ be the inclusions,  
and let $\ell'\in \pi_1(C')$ be the homotopy class of $J'$ in $C'$. Then 
\begin{eqnarray*}
  \ell' \ = \ j'(\ell^*)
  &=& j'(c) \cdot   \prod_{i=1}^r j'(\mu_i)^{\epsilon_i} \\
  &\equiv & j'(c) \cdot j'(\mu_1)^{\sum_i \epsilon_{i}} \mod \Gamma_{k+1} \pi_1(C')
 \end{eqnarray*}
where the last identity follows from the fact that  $j'(\mu_1),\dots,j'(\mu_r)$  belong to $\Gamma_k \pi_1(C')$ 
(by Proposition \ref{prop:CGO}\! (i)) and are conjugate to each other. 
Besides, \eqref{eq:ell^*} shows that the homology class of $J$ in $E$ is $\sum_i \epsilon_{i}$ times the oriented meridian of $K$: assuming now that $\operatorname{Lk}(K,J)=0$ in $C$, we deduce that $\sum_i \epsilon_{i}=0$ and get
\begin{equation} \label{eq:elle}
\{\ell'\}_k = \{j'(c)\}_k  \ \in \Gamma_k \pi_1(C')/ \Gamma_{k+1} \pi_1(C').
\end{equation}
Besides, we have $\ell=j(\ell^*)= j(c)$ which implies that 
\begin{equation} \label{eq:elle'}
\{\ell\}_k = \{  j(c)\}_k \ \in \Gamma_k \pi_1(C)/ \Gamma_{k+1} \pi_1(C).
\end{equation}
Observe that we have two commutative diagrams
$$
\xymatrix{
\frakL_k(H_1(E)) \ar@{->>}[r]  \ar@{->>}[d]_-j & 
\frac{\Gamma_k \pi_1(E)}{\Gamma_{k+1} \pi_1(E)} \ar@{->>}[d]^-j \\ 
\frakL_k(H_1(C)) \ar[r]_-\cong    &\frac{\Gamma_k \pi_1(C)}{\Gamma_{k+1} \pi_1(C)}
}
\quad \hbox{and} \quad
\xymatrix{
\frakL_k(H_1(E)) \ar@{->>}[r]  \ar@{->>}[d]_-{j'} & 
\frac{\Gamma_k \pi_1(E)}{\Gamma_{k+1} \pi_1(E)} \ar@{->>}[d]^-{j'} \\ 
\frakL_k(H_1(C')) \ar[r]_-\cong    &\frac{\Gamma_k \pi_1(C')}{\Gamma_{k+1} \pi_1(C')}
}
$$
and that the two compositions
$$
H_1(E)\stackrel{j}{\longrightarrow} H_1(C) \mathop{\longleftarrow}^{c_+}_{\cong } H
\quad \hbox{and} \quad
H_1(E)\stackrel{j'}{\longrightarrow} H_1(C') \mathop{\longleftarrow}^{c'_+}_{\cong } H
$$
are equal (because the surgery $C \leadsto C'$ has been performed on a null-homologous knot of $C$ with framing $\pm 1$).
Using those two observations, we deduce from \eqref{eq:elle} and \eqref{eq:elle'} 
that $\{\ell\}_k$ corresponds to $\{\ell'\}_k$ through the canonical isomorphisms.
\end{proof}

\begin{proof}[Proof of Theorem \ref{thm:generalized_C}]
We call $\operatorname{Lk}(L_+,L_-)\in \Z$ the \emph{linking number} of the blink $L$.
We first deal with the case of blinks $L$ with trivial linking number
(which, for instance, includes the situation considered in Theorem C). 
Then, by Lemma \ref{lem:blink_Lk}, 
the knot $L_\pm \subset C$ is $(\pm 1)$-framed in the sense of Section~\ref{subsec:surgeries}.
Let $E:= C \setminus \operatorname{int\, T}(L_-)$ be the exterior of the knot $L_-$.
We denote by $L'_+$ the knot $L_+\subset E$ seen in~$C':=C_{L_-}$, and
we shall consider $C_L$ as the result of doing surgery on $C'$ along $L'_+$.
By Lemma~\ref{lem:one_fact}\! (1), $L'_+$ is of nilpotency class~$\geq k$.

We claim that the framing of $L'_+$ inherited from the framing of $L_+$ is still~$+1$.
Then, a double application of Lemma~\ref{lem:surgery} will imply that $C_L \in \calC[2k-2]$.
Since $\operatorname{Lk}(L_+,L_-)=0$ and $L_+$ is null-homologous in $C$, we can find a compact oriented  surface $S$ in $E$ 
which can serve as  a Seifert surface for $L_+ \subset C$
as well as  for $L'_+ \subset C'$, 
showing that 
\begin{eqnarray*}
\partial \operatorname{T}(L_+) \supset \lambda_0(L_+)&=& 
\partial\big( S \cap (C \setminus \operatorname{int\, T}(L_+))\big)  \\
&=& \partial\big( S \cap (C' \setminus \operatorname{int\, T}(L'_+))\big) = \lambda_0(L'_+)
\subset   \partial \operatorname{T}(L'_+)
\end{eqnarray*}  
and proving our claim. 

We choose a symplectic expansion $\theta$ and apply twice the formula \eqref{eq:r} 
that has been shown in the proof of Theorem \ref{thm:generalized_B}. On the one hand, we have
\begin{equation} \label{eq:ell-}
r^\theta_{[k,2k[}(C_{L_-}) = 
 r_{[k,2k[}^\theta (C) -  \frac{1}{2} \cdot \theta_{k}(\ell_-) \hbox{\textbf{\,-\! -\! -}\,}  \theta_{k}(\ell_-) -
 \theta_{k}(\ell_-) \hbox{\textbf{\,-\! -\! -}\,}  \theta_{k+1}(\ell_-)
\end{equation}
where $\ell_- \in \pi$ is such that $\{\ell_-\}_{k+1}$ is mapped to $\{L_-\}_{k+1}$ by 
the isomorphism $c_+\colon  \pi / \Gamma_{k+2} \pi  \to \pi_1(C)/\Gamma_{k+2} \pi_1(C)$.
On the other hand, we have
\begin{equation} \label{eq:ell+}
r^\theta_{[k,2k[}(C_{L}) = 
 r_{[k,2k[}^\theta (C') + \frac{1}{2} \cdot \theta_{k}(\ell'_+) \hbox{\textbf{\,-\! -\! -}\,}  \theta_{k}(\ell'_+) +
 \theta_{k}(\ell'_+) \hbox{\textbf{\,-\! -\! -}\,}  \theta_{k+1}(\ell'_+)
\end{equation}
where $\ell'_+ \in \pi$ is such that $\{\ell'_+\}_{k+1}$ is mapped to $\{L'_+\}_{k+1}$ by 
the isomorphism $c'_+\colon  \pi / \Gamma_{k+2} \pi  \to \pi_1(C')  /\Gamma_{k+2} \pi_1(C')$.

We choose $\ell_-$ as above to have \eqref{eq:ell-}, 
and we also choose a $\delta\in \pi$ such that $\{\delta\}_{k+1}$ is mapped to $\Delta$ 
by $c_+\colon  \pi/ \Gamma_{k+2} \pi\to \pi_1(C)  /\Gamma_{k+2} \pi_1(C)$.
We set $\ell_+ := \delta \ell_-\in \pi$: then $\{\ell_+\}_{k+1}$ is mapped to $\{L_+\}_{k+1}$ by $c_+$.
Thus there exists $y\in \Gamma_{k+2} \pi_1(C)$ such that
$[L_+] = c_+(\ell_+)\, y \in \pi_1(C)$. Let $j\colon E \to C$ and $j'\colon E\to C'$ denote the inclusions.
There exists $\tilde y\in \Gamma_{k+2} \pi_1(E)$ such that $j(\tilde y)=y$, and then
\begin{equation} \label{eq:L+_E}
[L_+] = c^*_+ (\ell_+) \tilde y \cdot \prod_{i=1}^r \mu_i^{\epsilon_i}\in \pi_1(E)
\end{equation}
where  $c_+^* \colon \Sigma \to E$ is the corestriction of $c_+\colon \Sigma \to C$, 
 $\mu_1,\dots,\mu_r$ are copies of the oriented meridian of $L_-$ with possibly different  arc-basings, and $\epsilon_1,\dots,\epsilon_r$ are signs. We fix one of this based oriented meridian, say $\mu:=\mu_1$.
Then, for any $i\in \{2,\dots,r\}$, there exists $x_i\in \pi_1(E)$ such that $\mu_i = x_i \mu x_i^{-1}$.
We deduce the following identities in $\pi_1(C')/\Gamma_{k+2} \pi_1(C')$:
\begin{eqnarray}
\notag \{L'_+\}_{k+1} 
 &=& \Big\{  c'_+ (\ell_+) \cdot \prod_{i=1}^r \big[j'(x_i), j'(\mu)^{\epsilon_i}\big]\, j'(\mu)^{\epsilon_i} \Big\}_{k+1}\\
\notag &=& \Big\{  c'_+ (\ell_+) \cdot \prod_{i=1}^r \big[j'(x_i), j'(\mu)^{\epsilon_i}\big]
\cdot \prod_{i=1}^r j'(\mu)^{\epsilon_i} \Big\}_{k+1}\\
\notag &=& \Big\{  c'_+ (\ell_+) \cdot \prod_{i=1}^r \big[j'(x_i), j'(\mu)^{\epsilon_i}\big] \Big\}_{k+1} \\
\label{eq:so_long} &=& \Big\{  c'_+ (\ell_+) \cdot  \Big[\prod_{i=1}^r j'(x_i)^{\epsilon_i}, j'(\mu)\Big] \Big\}_{k+1}.
\end{eqnarray}
(Here the second and last identities follow from the fact that $j'(\mu) \in \Gamma_k \pi_1(C')$ 
by Proposition \ref{prop:CGO}\! (i), and
the third one is deduced from \eqref{eq:L+_E} and the fact that $\operatorname{Lk}(L_+,L_-)=0$.)
We now compute
$$
X:=\big\{ \widetilde{X} \big\}_{k+1}
\in \frac{\Gamma_{k+1} \pi_1(C')}{\Gamma_{k+2} \pi_1(C')}
\stackrel{\eqref{eq:canonical_iso}}{\cong } \frakL_{k+1},
\quad \hbox{where } \widetilde{X} := \Big[\prod_{i=1}^r j'(x_i)^{\epsilon_i}, j'(\mu)\Big]
$$
which is the Lie bracket of 
$$
X_1:=\Big\{\prod_{i=1}^r j'(x_i)^{\epsilon_i}\Big\}_1 \in \frac{\pi_1(C')}{\Gamma_2 \pi_1(C')} 
\stackrel{\eqref{eq:canonical_iso}}{\cong } \frakL_{1}=H
$$
with 
$$
X_2:=  \big\{ j'(\mu)\big\}_k  \in \frac{\Gamma_{k} \pi_1(C')}{\Gamma_{k+1} \pi_1(C')}
\stackrel{\eqref{eq:canonical_iso}}{\cong } \frakL_{k}.
$$ 
As it will appear later, we do not need to compute $X_1$ in a more explicit way.
As for $X_2$, we have $X_2=\{ \ell_-\}_k \in \frakL_k$ since $\{j'(\mu)\}_k=\{ [L_-]\}_k$ 
by Lemma \ref{lem:one_fact}\! (2) and we have choosen $\ell_-$ such that $c_+(\ell_-)$ represents $L_-$ modulo $\Gamma_{k+2} \pi_1(C)$. 
Hence we get
$
X =\big[X_1,\{\ell_-\}_k\big] \in \frakL_{k+1}.
$

Choose an $x_1 \in \pi$  representing the homology class $X_1$. Then, the previous paragraph shows
that $\widetilde{X} = c'_+\big([x_1, \ell_-]\big)$ modulo $\Gamma_{k+2} \pi_1(C')$,
and it follows from \eqref{eq:so_long} that we can take $\ell'_+:=\ell_+\,[x_1, \ell_-]$ to have \eqref{eq:ell+}.
We now estimate its value by the symplectic expansion $\theta$:
\begin{eqnarray*}
\theta(\ell'_+) 
&= \quad \ & \theta(\delta) \cdot \theta(\ell_-) \cdot \theta ([x_1,\ell_-])\\
&\equiv_{k+2}& (1+ \theta_{k+1}(\delta))\cdot (1+ \theta_{k}(\ell_-) + \theta_{k+1}(\ell_-)) \cdot (1+ [\theta_1(x_1),\theta_k(\ell_-)] ) \\
&= \quad \ & (1+ \{\delta\}_{k+1}) \cdot ( 1+ \{\ell_-\}_k + \theta_{k+1}(\ell_-)) \cdot ( 1+ [X_1,\{\ell_-\}_k]) \\
&\equiv_{k+2}& 1+ \{\ell_-\}_k +  \{\delta\}_{k+1} + \theta_{k+1}(\ell_-) + [X_1,\{\ell_-\}_k].
\end{eqnarray*}
Hence $\theta_k(\ell'_+) = \{\ell_-\}_k  = \theta_k(\ell_-)$ and  
$$\theta_{k+1}(\ell'_+) =  \{\delta\}_{k+1} + \theta_{k+1}(\ell_-) + \big[X_1,\{\ell_-\}_k\big].$$
Combining \eqref{eq:ell-} and \eqref{eq:ell+}, we obtain
\begin{eqnarray*}
r^\theta_{[k,2k[}(C_{L}) - r^\theta_{[k,2k[}(C)&=&
-  \frac{1}{2} \cdot \theta_{k}(\ell_-) \hbox{\textbf{\,-\! -\! -}\,}  \theta_{k}(\ell_-) -
 \theta_{k}(\ell_-) \hbox{\textbf{\,-\! -\! -}\,}  \theta_{k+1}(\ell_-)\\
&& + \frac{1}{2} \cdot \theta_{k}(\ell'_+) \hbox{\textbf{\,-\! -\! -}\,}  \theta_{k}(\ell'_+) +
 \theta_{k}(\ell'_+) \hbox{\textbf{\,-\! -\! -}\,}  \theta_{k+1}(\ell'_+)
\end{eqnarray*}
which, by the above estimate of $\theta(\ell'_+)$, implies that
\begin{eqnarray*}
r^\theta_{[k,2k[}(C_{L}) - r^\theta_{[k,2k[}(C) &=& 
\{\ell_-\}_k \hbox{\textbf{\,-\! -\! -}\,}  \{\delta\}_{k+1} + \{\ell_-\}_k \hbox{\textbf{\,-\! -\! -}\,} \big[X_1,\{\ell_-\}_k\big] \\
&=&  \{\ell_-\}_k \hbox{\textbf{\,-\! -\! -}\,}  \{\delta\}_{k+1}.
\end{eqnarray*}
(Here the second identity follows from the (AS) relation.)
In particular, the degree $2k-2$ part of $r^\theta_{[k,2k[}(C_L)$ is the same as that of $r^\theta_{[k,2k[}(C)$, 
and so is trivial: we deduce that $C_L \in \calC[2k-1]$.
Besides, using \eqref{eq:hcob_rt}, we deduce that~\eqref{eq:odd_bis} holds true 
for a blink $L$ with trivial linking number.

We now deal with blinks of arbitrary linking numbers.
Let $l\in \Z$ be an integer and let $\varepsilon \in\{-1,+1\}$.
We assume that the theorem holds true for any blink with linking number $l$,
and we wish to prove it for any blink of linking number $l-\varepsilon$.
This will prove the theorem by induction.
We shall only consider the case  $\varepsilon=+1$, the  case $\varepsilon=-1$ being similar.

Let  $ L=(L_+, L_-)$ be a blink of class $\geq k$ such that $\hbox{Lk}(L_+, L_-)=l$. 
Choose a connecting surface $B$ for $L$, and let $K \subset B$ be a separating simple closed curve
such that the closure of one of the two connected components of ${B \setminus K}$ 
is a $2$-hole disk $P$ with inner boundary  $L_- \cup K$ and with outer boundary $L_+$. 
We thicken $P$ to a genus two handlebody $A$ in $C$
and we assume that $\partial A$ cuts $B$ transversely along a curve  parallel to $K$,
which is shown as a dashed circle below: 
$$
\labellist
\scriptsize\hair 2pt
 \pinlabel {$K$} [lb] at 295 85
  \pinlabel {$L_-$} [lb] at 80 74
 \pinlabel {$L_+$} [lb] at 50 39
\endlabellist
\centering
\includegraphics[scale=0.27]{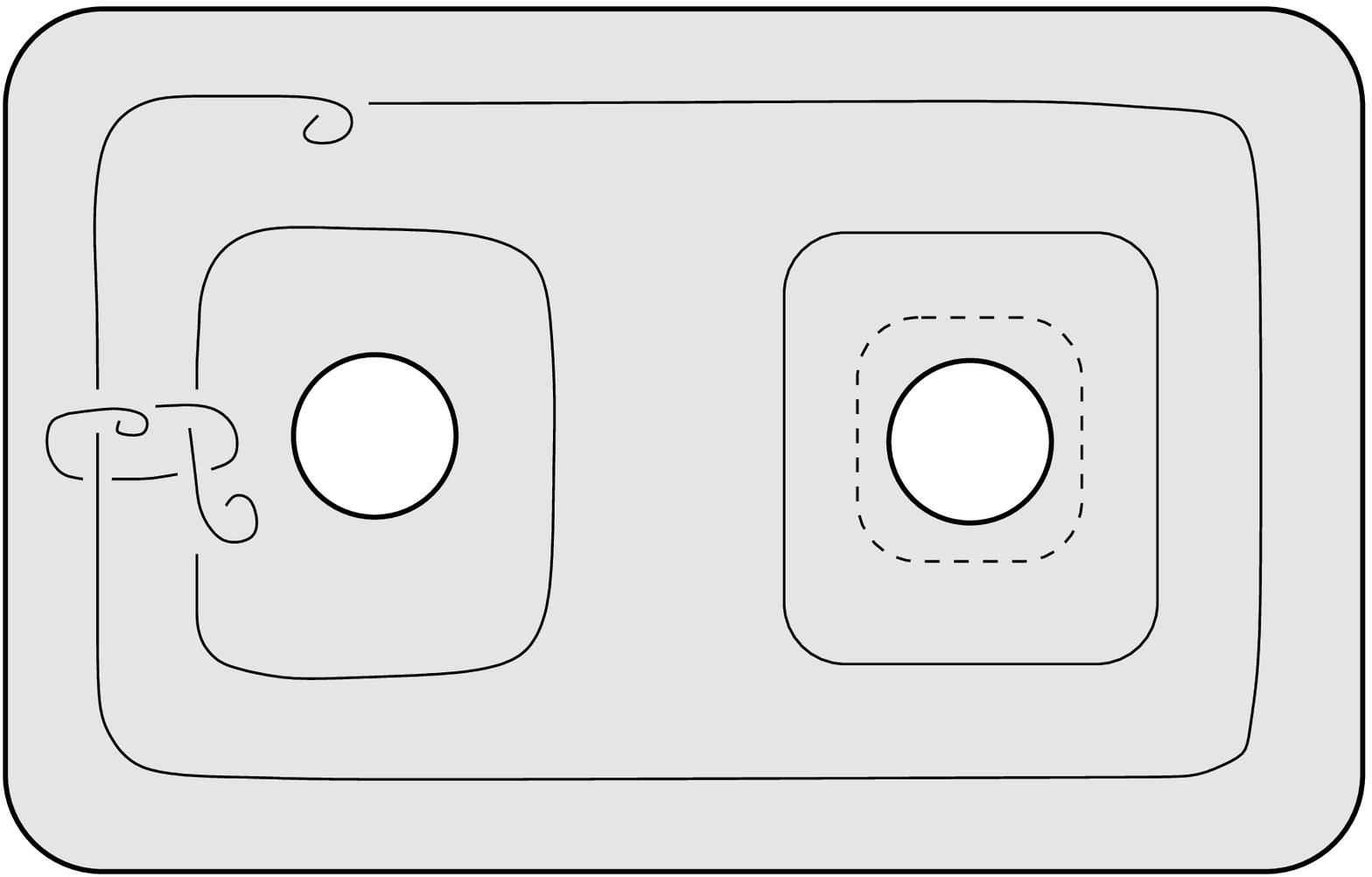}
$$
Let $\tilde L$ be the blink obtained from $L$ by surgery along the $(+1)$-framed trivial knot that is also shown in the above figure.
This surgery transforms $B$
to another surface $\tilde B$ with boundary $\tilde L$: hence $\tilde L$ is a blink (with connecting surface $\tilde B$). Note that 
$$
\hbox{Lk}(\tilde L_+, \tilde L_-)= \hbox{Lk}(L_+, L_-) -1 = l-1,
$$
and any blink $\tilde L$ of class $\geq k$ with  linking number $l-1$ 
arises in this way from a blink $L$ of class $\geq k$ with  linking number $l$.
Thus, we wish to prove that $\tilde L$ satisfies the two conclusions of the theorem.

The following sequence of handle-slide moves and isotopies show that $C_{\tilde L}$ 
is diffeomorphic to $C_{L\cup K}$ where the knot $K$ is now $(+1)$-framed:
$$
\labellist
\scriptsize\hair 2pt
 \pinlabel {$\cong$} at 546 160
  \pinlabel {$\cong$} at 1120 160
\endlabellist
\centering
\includegraphics[scale=0.21]{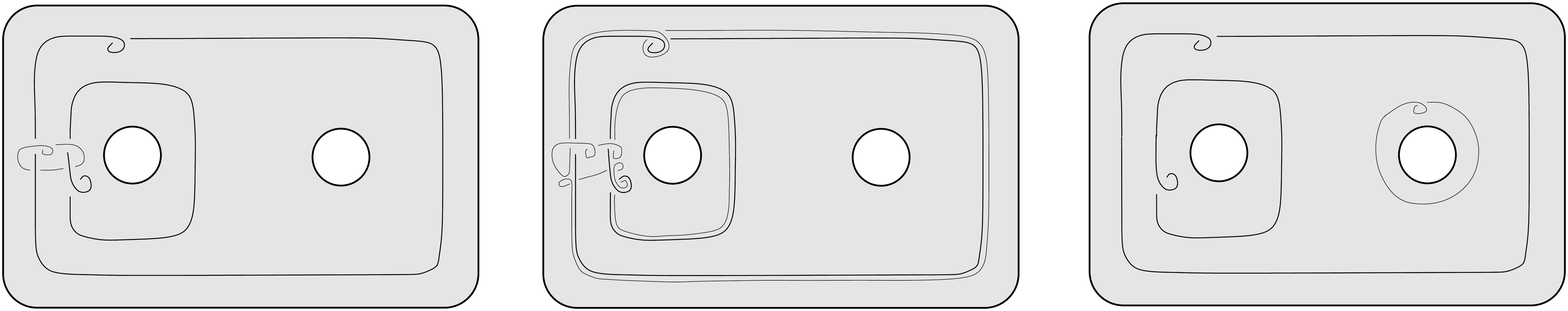}
$$
Let $L'$ be the framed link $L$ regarded in $C':=C_K$. Since $K$ is a $(+1)$-framed knot of nilpotency class $\geq k+1$,
we deduce from Lemma \ref{lem:surgery} that $C'\in {\calC[2k-1]}$ and that $\rho_{2k}(C)=\rho_{2k}(C')$, 
which implies that $\tau_{2k-1}(C)=\tau_{2k-1}(C')$.
 Furthermore, we deduce from Lemma \ref{lem:one_fact}\! (1)  that 
\begin{itemize} 
\item[(i)] $L'$ remains a blink of class $\geq k$ of linking number $l$ in $C'$,
\item[(ii)]  the classes $\Lambda'$ and $\Delta'$ assigned to $L'$ in $C'$, by  \eqref{eq:Lambda} and \eqref{eq:Delta} respectively,
 coincide with the classes $\Lambda$ and $\Delta$ assigned to $L$ in $C$.
\end{itemize} 
(Here we have used  that $\hbox{Lk}(L_\pm,K)=0$, 
which follows from the fact that $K$ bounds the surface $\overline{B\setminus P}$ in $C \setminus L_\pm$.)
Therefore, by our induction hypothesis, $C_{\tilde L}\cong C'_{L'}$ belongs to ${\calC[2k-1]}$ and 
\begin{eqnarray*}
\tau_{2k-1}(C_{\tilde L}) &=& \tau_{2k-1}(C') +  \eta ( \Lambda' \hbox{\textbf{\,-\! -\! -}\,} \Delta') 
\ = \ \tau_{2k-1}(C) + \eta ( \Lambda \hbox{\textbf{\,-\! -\! -}\,} \Delta).
\end{eqnarray*}
Finally, the classes assigned to $\tilde L$ in $C$ by  \eqref{eq:Lambda} and \eqref{eq:Delta}
are the same as those for $L$ in $C$ since $\tilde L$ is homotopic to $L$.
We conclude that $\tilde L$ too satisfies the conclusions of the theorem. 
\end{proof}

\begin{rem} \label{rem:special_cases_C}
There is a special case of  Theorem \ref{thm:generalized_C} that can be proved in a very different way.
Let $k\geq 2$ be an integer and let $C\in \calC[2k-1]$. Following the terminology of~\cite{Habiro},
let $G$ be a tree  clasper with $2k-1$ nodes. Then it is known that $C_G \in \calC[2k-1]$ and
\begin{equation}  \label{eq:wk}
\tau_{2k-1}(C_G) = \tau_{2k-1}(C) \pm \eta(\tilde G)
\end{equation}
where $\tilde{G}$ is the $H$-colored Jacobi diagram of degree $2k-1$ defined by~$G$.
(This can be deduced from the results of \cite{GL}, or, by combining Theorem~7.11 and Theorem 8.19 in \cite{CHM}.)
Assume now that there is an internal  edge $e$ of $G$ which subdivides it into two subtrees having $k$ and $k-1$ nodes, respectively.
Then, by applying  Move 9 in \cite{Habiro} sufficiently many times, 
we see that $G$ is equivalent to a basic clasper whose two leaves have nilpotency classes $\geq k+1$ and $\geq k$, respectively.
Replacing this basic clasper by the corresponding 2-component framed link and performing a handle-slide move, 
we see that surgery along $G$ can be realized by surgery along a blink $L$ of class $k$.
Then  \eqref{eq:odd_bis} for this  $L$ follows from~\eqref{eq:wk}.
Along the same lines, it can be proved that Theorem \ref{thm:generalized_C} also holds true for $k=1$.
\end{rem}

\end{document}